\documentclass[letter, 11pt]{article}
\usepackage{amsmath}
\usepackage{mathrsfs} 
\usepackage{amsthm}
\usepackage{amsfonts}
\usepackage{amssymb}
\usepackage{latexsym}
\usepackage{tikz} 
\usepackage{esint}
\usepackage{mathtools}
\usepackage{stmaryrd}

\usepackage{tikz-cd}
\usepackage{cancel}
\usepackage{sectsty}
\usepackage{hyperref}
\usepackage{caption}
\usepackage{float}
\usepackage{setspace}

\usepackage{comment}
\usepackage{quiver}

\usepackage{xcolor}

\DeclareMathAlphabet{\mathcal}{OMS}{cmsy}{m}{n}
\usepackage{mathptmx}

\definecolor{mediumtealblue}{rgb}{0.0, 0.33, 0.71}
\definecolor{red(munsell)}{rgb}{0.96, 0.12, 0.24}
\definecolor{forestgreen}{rgb}{0.13, 0.55, 0.13}
\definecolor{ginger}{rgb}{0.69, 0.4, 0.0}

\newcommand\myshade{85}
\colorlet{mylinkcolor}{red(munsell)}
\colorlet{mycitecolor}{mediumtealblue}
\colorlet{myurlcolor}{forestgreen}

\usepackage[capitalize,noabbrev]{cleveref}

\hypersetup{
	linkcolor  = mylinkcolor!\myshade!black,
	citecolor  = mycitecolor!\myshade!black,
	urlcolor   = myurlcolor!\myshade!black,
	colorlinks = true
}

\mathchardef\mhyphen="2D

\usepackage{titlesec}
\titleformat{\subsection}{\normalfont\bfseries}{\thesubsection}{1em}{}
\titleformat{\subsubsection}[runin]{\normalfont\bfseries}{\thesubsubsection}{1em}{}

\titleformat{\subsection}{\normalfont\bfseries}{\thesubsection}{1em}{}

\renewcommand{\thesubsubsection}{(\thesubsection.\arabic{subsubsection})}

\usepackage{tocloft} 

\usepackage[OT2,T1]{fontenc}
\DeclareSymbolFont{cyrletters}{OT2}{wncyr}{m}{n}
\DeclareMathSymbol{\Sha}{\mathalpha}{cyrletters}{"58}

\sectionfont{\fontsize{12}{15}\selectfont}
\subsectionfont{\fontsize{11}{15}\selectfont}
\subsubsectionfont{\fontsize{10}{15}\selectfont}

\usepackage[colorinlistoftodos,textsize=scriptsize]{todonotes}
\usepackage{marginnote}

\newcommand{\mytodo}[2][]{{%
 \let\marginpar\marginnote
 \reversemarginpar
 \renewcommand{\baselinestretch}{0.8}%
 \todo[#1]{#2}}}

\usepackage[matrix,tips,graph,curve]{xy}
\usepackage{enumitem}

\makeatletter

\usepackage[
backend=biber,
style=alphabetic,
sorting=anyt, doi=false, url=false, isbn=false, giveninits=true, maxnames = 6,
]{biblatex}

\makeatother


\makeatletter

\theoremstyle{plain}
\newtheorem{theorem}[subsubsection]{Theorem}

\newtheorem{lemma}[subsubsection]{Lemma}
\newtheorem{proposition}[subsubsection]{Proposition}

\newtheorem*{theoremA}{Theorem A}
\newtheorem*{theoremB}{Theorem B}
\newtheorem*{theoremC}{Theorem C}

\theoremstyle{definition}
\newtheorem{definition}[subsubsection]{Definition}

\newtheorem{remark}[subsubsection]{Remark}

\newtheorem{example}[subsubsection]{Example}
\newtheorem{set-up}[subsubsection]{Set-up}

\newtheorem{notation}[subsubsection]{Notation}

\newtheorem{assumption}[subsubsection]{Assumption}

\newcommand{\IA}{\mathbb{A}}

\newcommand{\IC}{\mathbb{C}}

\newcommand{\IF}{\mathbb{F}}
\newcommand{\IG}{\mathbb{G}}

\newcommand{\IK}{\mathbb{K}}
\newcommand{\IL}{\mathbb{L}}

\newcommand{\IN}{\mathbb{N}}

\newcommand{\IP}{\mathbb{P}}
\newcommand{\IQ}{\mathbb{Q}}
\newcommand{\IR}{\mathbb{R}}
\newcommand{\IS}{\mathbb{S}}

\newcommand{\IV}{\mathbb{V}}

\newcommand{\IZ}{\mathbb{Z}}

\newcommand{\sC}{\mathcal{C}}
\newcommand{\sK}{\mathcal{K}}

\newcommand{\End}{\mathrm{End}}
\newcommand{\tr}{\mathrm{tr}}

\newcommand{\Hom}{\mathrm{Hom}}
\newcommand{\Aut}{\mathrm{Aut}}

\newcommand{\Res}{\mathrm{Res}}

\renewcommand{\deg}{\mathrm{deg}}
\newcommand{\supp}{\mathrm{supp \,}}

\newcommand{\Spec}{\mathrm{Spec}}

\newcommand\iso{\,{\cong}\,} 
\newcommand\tensor{{\otimes}}

\newcommand{\<}{\langle}
\renewcommand{\>}{\rangle}

\newcommand{\into}{\hookrightarrow}

\def\d/{/\mspace{-6.0mu}/}

\def\wt{\widetilde}

\def\what{\widehat}

\newcommand{\p}{\partial}

\newcommand{\fU}{\mathfrak{U}}

\newcommand{\sU}{\mathcal{U}}

\newcommand{\sV}{\mathcal{V}}
\newcommand{\sI}{\mathcal{I}}

\newcommand{\sF}{\mathcal{F}}
\newcommand{\sH}{\mathcal{H}}

\newcommand{\w}{\omega}

\newcommand{\Ohm}{\Omega}

\newcommand{\fN}{\mathfrak{N}}
\newcommand{\fm}{\mathfrak{m}}
\newcommand{\fp}{\mathfrak{p}}

\newcommand{\Proj}{\mathrm{Proj\,}}
\newcommand{\Pic}{\mathrm{Pic}\,}

\newcommand{\sE}{\mathcal{E}}
\newcommand{\sL}{\mathcal{L}}
\newcommand{\sM}{\mathcal{M}}
\newcommand{\sO}{\mathcal{O}}
\newcommand{\sP}{\mathcal{P}}
\newcommand{\sZ}{\mathcal{Z}}

\newcommand{\Cl}{\mathrm{Cl}}

\newcommand{\Gal}{\mathrm{Gal}}
\newcommand{\T}{\mathrm{T}}
\newcommand{\NS}{\mathrm{NS}}
\newcommand{\Hdg}{\mathrm{Hdg}}

\newcommand{\Isom}{\mathrm{Isom}}

\newcommand{\et}{\mathrm{{\acute{e}}t}}

\newcommand{\Sh}{\mathrm{Sh}}
\newcommand{\shS}{\mathscr{S}}
\newcommand{\shA}{\mathscr{A}}
\newcommand{\shB}{\mathscr{B}}

\newcommand{\CSpin}{\mathrm{CSpin}}
\newcommand{\SO}{\mathrm{SO}}

\renewcommand{\sp}{\mathrm{sp}}
\newcommand{\GSp}{\mathrm{GSp}}
\newcommand{\GL}{\mathrm{GL}}

\newcommand{\sA}{\mathcal{A}}

\newcommand{\SL}{\mathrm{SL}}

\newcommand{\val}{\mathrm{val}}
\newcommand{\Mod}{\mathsf{Mod}}

\newcommand{\sX}{\mathsf{X}}
\newcommand{\cris}{\mathrm{cris}}
\newcommand{\MT}{\mathrm{MT}}

\newcommand{\bpi}{\mathbf{\pi}}
\newcommand{\dR}{\mathrm{dR}}
\newcommand{\Nm}{\mathrm{Nm}}
\renewcommand{\Isom}{\underline{\mathrm{Isom}}}
\newcommand{\Fil}{\mathrm{Fil}}

\newcommand{\Ab}{\mathsf{Ab}}
\newcommand{\Sing}{\mathrm{Sing}}

\newcommand{\sT}{\mathcal{T}}

\newcommand{\LEnd}{\mathrm{L}}

\newcommand{\shL}{\mathscr{L}}

\renewcommand{\Spec}{\mathrm{Spec\,}}

\newcommand{\fS}{\mathfrak{S}}

\newcommand{\bH}{\mathbf{H}}
\newcommand{\bP}{\mathbf{P}}
\newcommand{\Mon}{\mathrm{Mon}}

\renewcommand{\t}{\theta}
\renewcommand{\O}{\mathrm{O}}
\newcommand{\sW}{\mathcal{W}}

\newcommand{\Def}{\mathrm{Def}}

\newcommand{\AH}{\mathrm{AH}}
\newcommand{\an}{\mathrm{an}}

\newcommand{\shP}{\mathscr{P}}
\renewcommand{\sX}{\mathcal{X}}
\newcommand{\Mot}{\mathsf{Mot}}

\renewcommand{\bpi}{\boldsymbol{\pi}}

\newcommand{\sY}{\mathcal{Y}}
\newcommand{\bxi}{\boldsymbol{\xi}}
\newcommand{\bzeta}{\boldsymbol{\zeta}}

\newcommand{\sQ}{\mathcal{Q}}

\renewcommand{\H}{\mathrm{H}}

\newcommand{\sto}{\stackrel{\sim}{\to}}

\newcommand{\sG}{\mathcal{G}}

\newcommand{\PH}{\mathrm{PH}}
\newcommand{\PNS}{\mathrm{PNS}}

\newcommand{\bL}{\mathbf{L}}

\newcommand{\shC}{\mathscr{C}}

\newcommand{\sfM}{\mathsf{M}}

\newcommand{\fh}{\mathfrak{h}}
\newcommand{\fn}{\mathfrak{n}}
\newcommand{\fu}{\mathfrak{u}}
\newcommand{\fw}{\mathfrak{w}}

\newcommand{\sfK}{\mathsf{K}}

\newcommand{\shM}{\mathscr{M}}

\newcommand{\htt}{h}

\newcommand{\ff}{\mathfrak{f}}

\newcommand{\fs}{\mathfrak{s}}
\newcommand{\ft}{\mathfrak{t}}

\newcommand{\fD}{\mathfrak{D}}
\newcommand{\red}{\mathrm{red}}

\newcommand{\sfP}{\mathsf{P}}
\newcommand{\sfQ}{\mathsf{Q}}

\newcommand{\rN}{\mathfrak{N}}

\newcommand{\Disc}{\mathrm{Disc}}

\newcommand{\bvarphi}{\overline{\varphi}}

\newcommand{\sfV}{\mathsf{V}}
\newcommand{\sfW}{\mathsf{W}}
\newcommand{\sfH}{\mathsf{H}}
\newcommand{\sfR}{\mathsf{R}}

\newcommand{\ssfV}{{\wt{\sfV}}}

\newcommand{\am}{\mathrm{am}}

\newcommand{\sfN}{\mathsf{N}}

\newcommand{\bc}{\mathrm{bc}}
\newcommand{\bv}{\mathbf{v}}

\newcommand{\geo}{\mathrm{geo}}

\newcommand{\bsfM}{\wt{\sfM}}

\usepackage{marvosym}

\makeatother

\title{\large{\textbf{On the Tate conjecture for divisors on varieties with $h^{2, 0} = 1$ \\ in positive characteristics}}}
\author{\normalsize{Paul Hamacher, Ziquan Yang \Letter, and Xiaolei Zhao}}
\date{\vspace{-1em}}

\begin{document}

\maketitle

\begin{abstract}
    We prove that the Tate conjecture for divisors is ``generically true'' for mod $p$ reductions of complex projective varieties with $h^{2, 0} = 1$, under a mild assumption on moduli. By refining this general result, we establish a new case of the BSD conjecture over global function fields, and the Tate conjecture for a class of general type surfaces of geometric genus $1$. 
\end{abstract}

\tableofcontents

\section{Introduction}

In this paper, we prove a general theorem stating that the Tate conjecture for divisors is “generically true” for fibers in an arithmetic family of varieties with $h^{2, 0} = 1$, under some mild assumptions, and then we provide a method to refine this theorem in concrete situations. To motivate the discussion, we begin with an example of particular interest.

For an elliptic curve $\sE$ over a global field, (the basic version of) the BSD conjecture predicts that the analytic rank is equal to the Mordell-Weil rank. When the analytic rank of is at most 1, the conjecture is known in many cases over number fields (e.g., \cite{Kolyvagin}) and in full generality over function fields (\cite{Ulmer, Qiu}). When the analytic rank is $> 1$, much less is known, even over function fields. However, a special feature of function fields is that the BSD conjecture is equivalent to the Tate conjecture for the corresponding elliptic surfaces (\cite{KatoTrihan, MilneAT}), so one can alternatively attack through the Tate conjecture. As a refinement our general theorem, we provide a new class of examples over function fields. 



\begin{theoremA}
\label{thm: BSD}
Let $C$ be a smooth projective curve over a finite field $k$ of characteristic $p \ge 5$, $\sE$ be an elliptic curve over the function field $k(C)$, and $\pi : X \to C$ be the corresponding minimal elliptic surface. If $g(C) = h(\sE) = 1$, and all geometric fibers of $\pi$ are irreducible, then the BSD conjecture holds for $\sE$, i.e., $\mathrm{rank\,} \sE(k(C)) = \mathrm{ord}_{s = 1} L(\sE, s)$.   
\end{theoremA}

Here $g(C)$ denotes the genus of $C$, and $h(\sE)$ denotes the \textit{height} of $\sE$ (i.e., the degree of the \textit{fundamental line bundle} $L \coloneqq (R^1 \pi_* \sO_X)^\vee$). The above theorem is analogous to the classical result of Artin and Swinnerton-Dyer \cite{ASD} for elliptic K3 surfaces, which correspond to the case when $g(C) = 0$ and $h(\sE) = 2$. We remark that our elliptic surfaces have Kodaira dimension $1$, and the highest analytic rank achieved by the corresponding elliptic curves is $10$ (see \ref{rmk: highest MW}). The condition that all geometric fibers of $\pi$ are irreducible is saying that the Weierstrass normal form of $X$ is smooth. In the sequel of this paper by the second author and Guo \cite{GY25}, this condition is dropped in many cases, and completely removed when $p \ge 11$.

Now we introduce our general theorem. We shall consider arithmetic families of the following type: Suppose that $K$ is a field of characteristic $0$ and $S$ is a smooth connected $K$-variety. We say that a smooth projective morphism $f : \sX \to S$ is a \textbf{$\heartsuit$-family} if for every geometric point $s \to S$, the fiber $\sX_s$ is connected, $\dim \H^0(\Ohm^2_{\sX_s}) = 1$ (i.e., $h^{2,0}(\sX_s) = 1$), and the Kodaira-Spencer map 
\begin{equation}
\label{eqn: KS}
    \nabla \colon T_{S/K} \to \underline{\Hom}(R^1 f_* \Ohm^1_{\sX/S}, R^2 f_* \sO_{\sX})
\end{equation}
is nontrivial. Our general theorem states: 
\begin{theoremB}
\label{thm: generic}
Let $\sfM$ be a connected and separated scheme over $\mathrm{Spec}(\IZ)$ which is smooth and of finite type and let $f \colon \sX \to \sfM$ be a smooth projective morphism with geometrically connected fibers. If the restriction of $\sX$ to $\sfM_\IQ$ is a $\heartsuit$-family, then for $p \gg 0$, every fiber of $\sX$ over a point $s \in \sfM_{\IF_p}$ satisfies the Tate conjecture for divisors, i.e., for every prime $\ell \neq p$, the Chern class map 
$$c_1 : \Pic(\sX_{s}) \tensor \IQ_\ell \to \H^2_\et(\sX_{\bar{s}}, \IQ_\ell(1))^{\Gal(\bar{s}/s)} $$
is surjective, where $\bar{s}$ is a geometric point over $s$.\footnote{Over a finite field there is also a crystalline version of the Tate conjecture, which for divisors is known to be equivalent to the $\ell$-adic version (see \cite[Prop.~4.1]{Morrow}).} 
\end{theoremB}


Clearly, it is both a positive characteristic analogue of Moonen's main result in \cite{Moonen}, and a generalization of the Tate conjecture for K3 surfaces, for which people made great progress in the past decade (e.g., \cite{MPTate, Maulik, Charles}). The base scheme $\sfM$ above should be thought of as the moduli of varieties of a certain type. Note that we do not require in the theorem above that $s$ is a closed point---it is allowed to have positive transcendence degree over $\IF_p$. Also, although the above theorem is non-effective in $p$, for concretely given families we often can make it effective, as in the case of Thm~A. To illustrate this, we also analyze a class of surfaces of general type: 
\begin{theoremC}
\label{thm: general type}
Assume that $k$ is a field finitely generated over $\IF_p$ for $p \ge 5$. Let $X$ be a minimal smooth projective geometrically connected surface over $k$. Let $K_X$ denote its canonical divisor, $p_g$ the geometric genus $h^0(K_X)$, and $q$ the irregularity $h^1(\sO_X)$. If $p_g = K_X^2 = 1$, $q = 0$ and $K_X$ is ample, then $X$ satisfies the Tate conjecture.
\end{theoremC}

We remark that condition (ii) in Thm~A is an analogue of the condition ``$K_X$ is ample'' above. Over $\IC$, surfaces with the above invariants were classified by Todorov and Catanese (\cite{Todorov, Catanese0}). They are simply connected, have a coarse moduli space of dimension $18$ and were among the first examples of $p_g = 1$ surfaces for which both the local and the global Torelli theorems fail (\cite{Catanese}).

\paragraph{Sketch of Proofs} We first explain how to prove Theorem B. We build on the overall strategy of Madapusi-Pera \cite{MPTate}, which has two main steps. The first is to construct an integral period morphism $\rho : \sfM \to \shS$, where $\shS$ is the canonical integral model of a Shimura variety $\Sh(G)$ defined by a suitable special orthogonal group $G$.\footnote{For the exposition of ideas, we temporarily suppress Hermitian symmetric domains and level structures from the notation of Shimura varieties.} Up to passing to its spinor cover, $\shS$ is equipped with a family of abelian schemes $\sA$. The second is to construct, for each geometric point $s \to \sfM$ and $t = \rho(s)$, a morphism $\t : \LEnd(\sA_{t}) \to \NS(\sX_s)$, where $\LEnd(\sA_{t})$ is a distinguished subspace of $\End(\sA_t)$. Then the Tate conjecture for $\NS(\sX_s)$ follows from a variant of Tate's theorem for $\LEnd(\sA_{t})$. As a crucial step, Madapusi-Pera proved that $\rho$ is \textit{\'etale}. This boils down to the geometric fact that a K3 surface $X$ has unobstructed deformation and $\H^1(\Ohm_X) \iso \H^1(T_X)$. Unfortunately, this is rarely true when $X$ is not a close relative of a hyperk\"ahler variety. 

The main contribution of our paper is a method to remove the dependence of the above strategy on any good local property of $\rho$ (\textit{not even flatness}). Indeed, the condition on the Kodaira-Spencer map contained in the definition of a $\heartsuit$-family just amounts to asking $\dim \mathrm{im}(\rho_\IC) > 0$, so even the situation $\dim \sfM < \dim \shS$ is allowed in Thm~B. Below we explain in more detail the difficulties in extending the two steps above and how to overcome them.

\subsubsection{} It is not hard to construct a period morphism $\rho_\IC : \sfM_\IC \to \Sh(G)_\IC$ over $\IC$, as the target is a moduli of variations of Hodge structures with some additional data. As in \cite{MPTate}, the idea to construct $\rho$ is to descend $\rho_\IC$ to a morphism $\rho_\IQ$ over $\IQ$, and then appeal to the extension property of $\shS$. If for every $s \in \sfM(\IC)$, the motive $\fh^2(\sX_s)$ defined by $\H^2(\sX_s, \IQ)$ is an abelian motive (e.g., in the K3 case), then one shows that the action of $\rho_\IC$ on $\IC$-points are $\Aut(\IC)$-equivariant, so that $\rho_\IC$ descends to $\IQ$ just as in \cite{MPTate}. 

To treat the general case, we absorb inputs from Moonen's work \cite{Moonen}. Let $b \in \sfM(\IC)$ be a point lying over the generic point of $\sfM$, and let $\T(\sX_b)$ be the orthogonal complement of $(0,0)$-classes in $\H^2(\sX_b, \IQ(1))$. Let $E$ be the endomorphism algebra of the Hodge structure $\T(\sX_b)$, which is known to be either a totally real or a  CM field. In the latter case, one can still show that $\fh^2(\sX_s)$ is abelian for each $s$ (see \ref{thm: Moonen}), so that the argument of \cite{MPTate} still applies. In the former case, we need to consider an auxiliary Shimura subvariety $\Sh(\sG)_\IC \subseteq \Sh(G)_\IC$, defined by the Weil restriction $\sG$ of a special orthogonal group over $E$. 

Up to replacing $\sfM$ by a connected \'etale cover and $b$ by a lift, the restriction of $\rho_\IC$ to $\sfM^\circ$ factors through a morphism $\varrho_\IC : \sfM^\circ \to \Sh(\sG)_\IC$, where $\sfM^\circ$ is the connected component of $\sfM_\IC$ containing $b$. We will show that the field of definition $F$ of $\sfM^\circ$ always contains $E$. Interestingly, to descend $\rho_\IC$ to $\IQ$ it suffices to descend $\varrho_\IC$ to $F$. The trick is to consider the left adjoint to the base change functor from $\IQ$-schemes to $F$-schemes. To show that $\varrho_\IC$ descends to $F$, consider the submotive $\ft(\sX_b)$ of $\fh^2(\sX_b)$ defined by $\T(\sX_b)$. When $\dim_E \T(\sX_b)$ is odd, although we do not know that $\ft(\sX_b)$ (or equivalently $\fh^2(\sX_b)$) is abelian, Moonen's work \cite{Moonen} tells us that (a slight variant of) ``the $E/\IQ$-norm'' $\Nm_{E/\IQ}(\ft(\sX_b))$ is abelian. This allows us to show that $\varrho_\IC$ descends to $F$ because considering $\Nm_{E/\IQ}(\ft(\sX_b))$ amounts to considering a different \textit{faithful} representation of $\sG$. As we document in a separate reference file \cite{YangSystem}, under some hypotheses the canonical models of Shimura varieties of abelian type over their reflex fields have a moduli interpretation attached to any faithful representation. When $\dim_E \T(\sX_b)$ is even, some adaptation is needed, which we omit here. 

\subsubsection{} Next, we explain how to construct $\t : \LEnd(\sA_t) \to \NS(\sX_s)$, where the main novelty of our method lies. As in \cite{MPTate}, the key to construct $\t$ is to find, for each $\zeta \in \LEnd(\sA_{t})$, a characteristic $0$ point $\wt{t}$ on $\shS$ lifting $t$, such that (i) $\zeta$ deforms to $\sA_{\wt{t}}$ and (ii) $\wt{t}$ comes from a lifting $\wt{s}$ of $s$ via $\rho$. The existence of $\wt{t}$ which satisfies (i) was already shown in \cite{CSpin}. When $\rho$ is \'etale (or at least smooth), (ii) is then automatically satisfied. This is where \cite{MPTate} crucially relies on the \'etaleness of $\rho$. 

The challenge to generalize this, especially when $\dim \sfM < \dim \shS$, is that there is no general way to characterize locally the image of $\rho$ in $\shS$, so for any given $\wt{t}$ one cannot decide directly whether it satisfies (ii) or not. Indeed, this is essentially a local Schottky problem, which is famously hard. To overcome this, we revisit Deligne's insight for \cite[Thm~1.6]{Del02} but replace the local crystalline analysis by a global and topological argument. 

Pretend for a moment that $\sfM_\IC$ and $\sfM_k$ are both connected, where $k$ is the (separably closed) field defining $s$, and set $p = \mathrm{char\,} k$. Let $\bar{\eta}$ and $\bar{\eta}_p$ be geometric generic points over $\sfM_\IC$ and $\sfM_k$ respectively. Pick a prime $\ell \neq p$ and restrict $f : \sX \to \sfM$ to $\IZ_{(p)}$. Suppose for the sake of contradiction that for some $\zeta$, there is no lifting $\wt{t}$ which satisfies (i) and (ii). Then using that the deformation of $\zeta$ is controlled by a single equation, we can show that $\zeta$ deforms along the formal completion of $\sfM_k$ at $s$, and hence gives rise to an element of $\LEnd(\sA_{\rho(\bar{\eta}_p)})$. This further induces an element in $(R^2 f_* \IQ_\ell)_{\bar{\eta}_p}$ which is stabilized by an open subgroup of $\pi_1^\et(\sfM_{k}, \bar{\eta}_p)$. On the other hand, we show using the theorem of the fixed part that all elements of $(R^2 f_* \IQ_\ell)_{\bar{\eta}}$ stabilized by an open subgroup of $\pi_1(\sfM_\IC, \bar{\eta})$ come from $\LEnd(\sA_{\bar{\eta}}) \tensor \IQ_\ell$. Therefore, to derive a constradiction it suffices to show that $\sfM_k$ does not have more ``$\pi_1$-invariants'' than $\sfM_\IC$. 

Using Hironaka's resolution of singularities and a spreading out argument, we can find an open subscheme $U$ of $\Spec(\IZ)$ such that $\sfM_U$ admits a compactification whose boundary is a relative normal crossing divisor. Then we apply Grothendieck's specialization theorems for tame fundamental group and Abhyankar's lemma to show that $\sfM_k$ indeed cannot have more ``$\pi_1$-invariants'' than $\sfM_\IC$, when $(p) \in U$.\footnote{We thank Aaron Landesman for pointing out to us the applicability of Abhyankar's lemma, which simplifies our original argument.} This proves Thm~B.

\subsubsection{} To prove Theorem~A and C, we need to avoid the spreading out argument above. Although compactifying moduli spaces is in general a hard geometric problem, one can find for $\sfM$ in question a \textit{partial compactification}. That is, a morphism $\sfM \to B$ and a smooth proper $\shP$ over $B$ such that $\sfM$ is an open subscheme of $\shP$. Then we can find lots of smooth proper curves in $\shP$. If the boundary $\fD = \shP - \sfM$ is generically reduced modulo $p$, then we can find a curve on $\sfM_k$ which deforms to characteristic $0$ such that by looking at the curve we can already prove that $\sfM_k$ does not have more ``$\pi_1$-invariants'' than $\sfM_\IC$. Of course, such curve needs to be chosen wisely, and we do this by repeatedly applying the Baire category theorem. We will also need $\rho$ to satisfy a stronger condition, namely $\dim \mathrm{im}(\rho_\IC) > \dim B_\IC$, as opposed to just $\dim \mathrm{im}(\rho_\IC) > 0$. This is known to hold for the surfaces in question. 

The boundary $\fD$ is essentially a discriminant scheme, i.e., $\sX$ extends to a family over $\shP$, and $\fD$ is precisely the locus where the extension fails to be smooth. In general, it is possible for a discriminant scheme over $\IZ$ to be generically non-reduced modulo a certain prime (cf. \cite[Thm~4.2]{Saitodisc}), so the task is to determine an effective range of $p$ for which this does not happen. Drawing ideas from enumerative geometry, we show that this happens only when a general fiber over $\fD_{\bar{\IF}_p}$ is ``more singular'' than that over $\fD_{\IC}$. To exclude this possibility when $p \ge 5$ for the surfaces in Thm~C, it suffices to adapt Katz' results on Lefschetz pencils. Doing this for Thm~A is much more involved. In particular, we need to develop some nonlinear Bertini theorems (see \S\ref{Sec: Nonlinear Bertini}) tailored to handle Weierstrass equations, the key input being Kodaira's classification for the singular fibers in an elliptic fibration. 

\begin{remark}
\label{rmk: GM}
Recently Fu and Moonen proved the Tate conjecture for Gushel-Mukai varieties in characteristic $p \ge 5$. The middle cohomology of these varieties behaves like that of a K3 surface up to a Tate twist. An earlier version of the paper also discussed these varieties. For our method, these varieties can be treated in a similar way as the surfaces in Thm~C. However, as Fu-Moonen \cite{FuMoonen} gave much more thorough treatment of these varieties and proved that the relevant integral period morphism $\rho$ is indeed smooth, we removed the section from the current version. In general, it is hard to determine whether $\rho$ is smooth integrally even when $\rho_\IC$ is smooth, as one can tell from \cite{FuMoonen}, but when this can be achieved, $\rho$ has many more potential applications than the divisorial Tate conjecture (e.g., CM lifting and the Tate conjecture for self-correspondences, as shown in \cite{IIK}). 

On the other hand, our main purpose is to deal with the situation when the smoothness of $\rho$ cannot be hoped for. In particular, Thm~B implies that the characteristic $p$ counterparts of the surfaces in Moonen's list \cite[Thm~9.4]{Moonen} satisfy the Tate conjecture for $p \gg 0$, and we expect that our methods for the refinements Thm A and C can be adapted to make Thm B effective for most classes of these surfaces.
\end{remark}

\paragraph{Organization of Paper} In section 2, we review and mildly extend Moonen's results in \cite{Moonen} on the motives of fibers of $\heartsuit$-families. In particular, we recap the norm functors used in \textit{loc. cit.} In section 3, we discuss the moduli interpretations of Shimura varieties of abelian type over their reflex fields, as documented in \cite{YangSystem}, and recap the integral models of orthogonal Shimura varieties from \cite{CSpin}. In section 4, we construct the period morphisms for $\heartsuit$-families in characteristic $0$, using results from sections 2 and 3. In section 5, we prove Thm~B after giving a more effective version (see \ref{prop: TateStrong}). In section 6, we set up some basic tools to analyze deformation of curves on parameter spaces, including applications of the Baire category theorem. Finally, in section 7 and 8, we use tools from section 6 as well as \ref{prop: TateStrong} to prove Thm~A and C respectively. In particular, in section 7 we study the geometry of natural parameter spaces of elliptic surfaces, which we hope to be of independent interest. 


\subsubsection{}\label{sec: notations} Finally we introduce some notations and conventions. 
\begin{enumerate}[label=\upshape{(\alph*)}]
    \item Let $f : X \to S$ be a morphism between schemes. If $T \to S$ is another morphism, denote by $X_T$ the base change $X \times_S T$ and by $X(T)$ the set of morphisms $T \to X$ as $S$-schemes. By a geometric fiber of $X$, we mean $X_s$ for some geometric point $s \to S$. If $x$ is a point (resp. geometric point) on a scheme $X$, we write $k(x)$ for its residue field (resp. field of definition). By a variety over a field $k$ we mean a scheme which is reduced, separated and of finite type over $k$. 
    \item The letters $p$ and $\ell$ will always denote some prime numbers and $\ell \neq p$ unless otherwise noted. We write $\what{\IZ}$ for the profinite completion of $\IZ$ and $\what{\IZ}^p$ for its prime-to-$p$ part. Define $\IA_f := \what{\IZ} \tensor \IQ$ and $\IA^p_f := \what{\IZ}^p \tensor \IQ$. If $k$ is a perfect field of characteristic $p$, we write $W(k)$ for its ring of Witt vectors. 
    \item For a field $k$ of characteristic $\neq 2$, we consider a quadratic form $q$ on a finite dimensional $k$-vector space $V$ simultaneously a symmetric bilinear pairing $\<-, -\>$ such that $q(v) = \< v, v\>$ for every $v \in V$. 
    \item For a finite free module $M$ over a ring $R$, we write $M^\tensor$ the direct sum of all the $R$-modules which can be formed from $M$ by taking duals, tensor products,
    symmetric powers and exterior powers. We also use this notation for sheaves of modules on some Grothendieck topology whenever it makes sense.
    \item We use the following abbreviations: VHS for ``variations of Hodge structures'', LHS (resp. RHS) for ``left (resp. right) hand side'', and ODP for ``ordinary double point''. Unless otherwise noted, the local system in a VHS is a local system of $\IQ$-vector spaces. Moreover, we always assume that the VHS is \textit{pure}, i.e., it is a direct sum of those pure of some weight (or the weight filtration is split).
\end{enumerate}

\section{Preliminaries}
\subsection{Motives and norm functors} 

\subsubsection{} Let $k$ be a field of characteristic $0$. We denote by $\Mot_\AH(k)$ the neutral $\IQ$-linear Tannakian category of motives over $k$ with morphisms defined by absolute Hodge correspondences (cf. \cite[\S2]{Pan94} where it is denoted by $\shM_k$). We have the Tate objects $\mathbf{1}(n)$ for every $n \in \IZ$ in this category. For any object $M \in \Mot_\AH(k)$, we write $M(n)$ for $M \tensor \mathbf{1}(n)$ and by an absolute Hodge cycle on $M$ we mean a morphism $\mathbf{1} \to M$. 

Following \cite[\S2]{MPTate}, we denote by $\w_\ell$ the $\ell$-adic realization functor which sends $\Mot_\AH(k)$ to the category of finite dimensional $\IQ_\ell$-vector spaces with an action of $\Gal_k := \Gal(\bar{k}/k)$, where $\bar{k}$ is some chosen algebraic closure of $k$. Putting $\w_\ell$'s together we obtain $\w_\et$, which takes values in the category of finite free $\IA_f$-modules with a $\Gal_k$-action. Let $\w_\dR$ denote the de Rham realization functor, which takes values in the category of filtered $k$-vector spaces. If $k$ is a subfield of $\IC$, we additionally consider the Betti realization $\w_B$ after base change to $\IC$ (resp. the Hodge realization $\w_\Hdg$) which takes values in the category of $\IQ$-vector spaces (resp. Hodge structures). For a smooth projective variety $X$ over $k$, $\fh^i(X)$ denotes the object such that $\w_?(\fh^i(X))$ is the $i$th $?$-cohomology of $X$, for $? = B, \ell, \dR$ whenever applicable. 

Let $\Mot_\Ab(k) \subseteq \Mot_\AH(k)$ be the full Tannakian sub-category generated by the Artin motives and the motives attached to abelian varieties. We will repeatedly make use of the following fact (\cite[Ch~I]{LNM900}, cf. \cite[Thm~2.3]{MPTate}): 

\begin{theorem}
\label{thm: tensor on AV always AH}
     The functor $\w_\Hdg$ is fully faithful when restricted to $\Mot_\Ab(\IC)$. In particular, for every $M \in \Mot_\Ab(\IC)$, every element $s \in \w_B(M) \cap \Fil^0 \w_\dR(M)$ is given by an absolute Hodge cycle.  
\end{theorem}

We often refer to objects in $\Mot_\Ab(k)$ as abelian motives. 


\subsubsection{} \label{sec: twist tensors} We will often consider the automorphism $-\tensor_\sigma \IC$ on $\Mot_\AH(\IC)$ defined by an element $\sigma \in \Aut(\IC)$ (cf. \cite[{\S}II~6.7]{LNM900}, see also \cite[Prop.~2.2]{MPTate}). For $M \in \Mot_\AH(\IC)$, we write $M^\sigma$ for $M \tensor_\sigma \IC$. Base change properties of \'etale (resp. de Rham) cohomology give us a $\IA_f$-linear (resp. $\sigma$-linear) canonical isomorphism $\w_\et(M) \iso_\bc \w_\et(M^\sigma)$ (resp. $\w_\dR(M) \iso_\bc \w_\dR(M^\sigma)$). Here the subscript ``bc'' is short for ``base change''. For an absolute Hodge class $s \in \w_B(M)$, we write $s^\sigma$ for the class in $\w_B(M^\sigma)$ which has the same \'etale and de Rham realizations as $s$ under the ``$\iso_\bc$'' isomorphisms. 

Finally, we remark that for $M \in \Mot_\AH(\IC)$, it makes sense to say whether a tensor $s \in \w_B(M)^\tensor$ (see \ref{sec: notations}(d)) is absolute Hodge, because $s$ has to lie in $\w_B$ of a finite direct sum of tensorial constructions on $M$. And when $s$ is absolute Hodge, we may form $s^\sigma \in \w_B(M^\sigma)^\tensor$ for $\sigma \in \Aut(\IC)$, extending the notation in the previous paragraph.

\subsubsection{} \label{sec: norm functors} Next, we recall the basics of norm functors. The reader may refer to \cite[\S3]{Moonen} for more details. Let $k$ be a field of characteristic $0$ and $E$ be a finite \'etale $k$-algebra. Let $\shC$ be any Tannakian $k$-linear category and $\Mod_E(\shC)$ be the category of $E$-modules in $\shC$. For any object $M \in \Mod_E(\shC)$, we write $M_{(k)}$ for the underlying object in $\shC$ when we forget the $E$-linear structure. In \cite{Ferrand}, Ferrand gave a general construction of a norm functor $\Nm_{E/k} \colon \Mod_E(\shC) \to \shC$, which was summarized in \cite[\S3.6]{Moonen}.

We first consider the case when $\shC$ is the category of $k$-modules $\Mod_k$. For any $M \in \Mod_E$, there is a functorial polynomial map $\nu_M \colon M \to \Nm_{E/k}(M)$ such that $\nu_M(em) = \mathrm{Norm}_{E/k}(e) \nu_M(m)$ for any $e \in E$ and $m \in M$. The norm functor $\Nm_{E/k}$ is a $\tensor$-functor and is non-additive (unless $E = k$). However, for any $M_1, M_2 \in \Mod_E$, there is an identification 
$$ \Nm_{E/k}(\Hom_E(M_1, M_2)) = \Hom_k(\Nm_{E/k}(M_1), \Nm_{E/k}(M_2)). $$
For any $\sV \in \Mod_E$, there is a natural map 
$$ \eta : \Res_{E/k} \GL(\sV) \to \GL(\Nm_{E/k}(\sV)) $$ which sends an $E$-linear automorphism $f$ to $\Nm_{E/k}(f)$. Let $T_{E/k}$ denote the torus $\Res_{E/k} \IG_{m, E}$ and $T_{E/\IQ}^1$ denote the kernel of the norm map $T_{E/k} \to T_k$. When $\IG_{m, E}$ is viewed as the diagonal torus of $\GL(\sV)$, so that $T_{E/k}^1$ is a subgroup of $\Res_{E/k} \GL(\sV)$, we have $\ker(\eta) = T^1_{E/k}$.



\begin{notation}
\label{not: totally real}
    Take $k = \IQ$ and $E$ to be a totally real field. Let $\sV$ be an $E$-vector space equipped with a quadratic form $\wt{\phi} : \sV \to E$. We often drop $\wt{\phi}$ from the notation when it is assumed.
    \begin{enumerate}[label=\upshape{(\alph*)}]
        \item Write $\sV_{(\IQ)}$ for the underlying $\IQ$-vector space of $\sV$, equipped with a quadratic form $\phi : \sV_{(\IQ)} \to \IQ$ given by $\tr_{E/\IQ} \circ \wt{\phi}$. 
        \item Write $\sG(\sV)$ for $\Res_{E/\IQ} \SO(\sV)$, $\sZ(\sV)$ for $T^1_{E/\IQ} \cap \sG(\sV)$, and $\sH(\sV)$ for $\sG(\sV)/ \sZ(\sV)$, or simply $\sG, \sZ$ and $\sH$ when $\sV$ is understood.
        \item Denote by $\Cl^+(\sV)$ the even Clifford algebra of $\sV$ over $E$ and by $\Cl^+_{E/\IQ}(\sV)$ its norm $\Nm_{E/\IQ} \Cl^+(\sV)$. 
        \item Let $\sE$ denote the $1$-dimensional quadratic form over $E$ given by equiping $E$ with the form $\wt{\phi}(\alpha) = \alpha^2$. 
    \end{enumerate}
\end{notation}
Recall our convention \ref{sec: notations}(c). The association $\sV \mapsto \sV_{(\IQ)}$ in (a) above defines an equivalence of categories between quadratic forms over $E$ and quadratic forms over $\IQ$ with an self-adjoint $E$-action (\cite[Ch~1, Thm~7.4.1]{Knus}). We call $\sV_{(\IQ)}$ the transfer of $\sV$, and $\sV$ the $E$-bilinear lift of $\sV_{(\IQ)}$. 

\subsubsection{} \label{motivic Galois group} Let us recall some formalism about motives for readers' convenience. Let $k$ be a subfield of $\IC$. For any object $\fm \in \Mot(k)$ or $\Mot_\AH(k)$, let $G_{\mathrm{mot}}(\fm)$ be the motivic Galois group of $\fm$, which is defined in \cite[\S4.6]{AndreMot}. Here we use Betti cohomology as the reference Weil cohomology theory, through which the Tannakian subcategory $\< \fm \> \subseteq \Mot_\AH(k)$ generated by $\fm$ is naturally equivalent to the category of $G_{\mathrm{mot}}(\fm)$-representations over $\IQ$. In particular, $\fm$ itself corresponds to a representation $G_{\mathrm{mot}}(\fm) \to \GL(\w_B(\fm))$, such that every vector or subspace of $\w_B(\fm)^\tensor$ invariant under $G_{\mathrm{mot}}(\fm)$ is defined by a submotive of $\fm^\tensor$. The reader may check out \cite[\S3.1]{Moonen-Fom} for an exposition of these notions.

\begin{lemma}
\label{lem: truncate}
Let $E$ be a totally real field, and let $\fw \in \Mot_\AH(\IC)$ be a motive equipped with an $E$-action and a symmetric $E$-bilinear form $\wt{\phi} : \fw \tensor_E \fw \to \mathbf{1}_E$. If $\dim_E \w_B(\fw)$ is odd, then the motive $\rN(\fw) := \Nm_{E/\IQ}(\fw) \tensor_\IQ \det(\fw_{(\IQ)})$ is (noncanonically) isomorphic to a submotive of $\Cl^+_{E/\IQ}(\fw, \wt{\phi})$.
\end{lemma} 
The meaning of $\Cl^+_{E/\IQ}(\fw, \wt{\phi})$ is explained in the proof.
\begin{proof}
This follows from the content in \cite[\S5.4]{Moonen}. We give a sketch so that the reader can easily check the details from \textit{loc. cit}. Set $W:= \w_B(\fw)$ and let $\sW$ be the $E$-bilinear lift of $W$. Since the $E$-action and the pairing $\wt{\phi}$ are motivic, the representation $G_{\mathrm{mot}}(\fw) \to \GL(W)$ defined by $\fw$ takes values in $\O_{E/\IQ}(\sW, \wt{\phi})$. Then $\Cl^+_{E/\IQ}(\fw, \wt{\phi}) \in \< \fw \>$ is the motive defined by the adjoint representation of $G_{\mathrm{mot}}(\fw)$ on $\Cl^+_{E/\IQ}(\sW, \wt{\phi})$. 

Let $\Sigma$ be the set of embeddings $\sigma : E \into \IC$, $\fS$ be the symmetric group of $\Sigma$, and $Q$ be the set of $\fS$-orbits of $\IN^\Sigma$. Then under the assumption that $\dim_E \sW$ is odd, there is an ascending filtration $\Fil_\bullet$ on $\Cl^+_{E/\IQ}(\fw, \wt{\phi})$ indexed by $Q$ such that for some $q_1, q_2 \in Q$, $\Fil_{q_1} / \Fil_{q_2} \iso \rN(\fw)$. Therefore, $\rN(\fw)$ is a subquotient of $\Cl^+_{E/\IQ}(\fw, \wt{\phi})$. As $\Mot_\AH(\IC)$ is semisimple, $\fN(\fw)$ is in fact (non-canonically) a sub-object. 
\end{proof}

\subsection{Motives of varieties with $h^{2, 0} = 1$}
\label{sec: summarize Moonen}
\begin{definition}
A polarized Hodge structure $V$ of weight $0$ is said to be of K3-type if $\dim V_\IC^{(-1, 1)} = \dim V_\IC^{(1, -1)} = 1$, and $V_\IC^{i, j} = 0$ when $|i - j| > 2$. The transcendental part $T(V)$ of $V$ is the orthogonal complement of $V^{(0,0)} := V \cap V_\IC^{(0, 0)}$. 
\end{definition}

We recall the following fundamental result of Zarhin: 
\begin{theorem}[{\cite[\S2]{Zarhin}}]
\label{thm: Zarhin}
Let $V$ be a Hodge structure of K3-type such that $V^{(0,0)} = 0$ and let $\phi$ be its polarization form. Then the endomorphism algebra $E \coloneqq \End_\Hdg V$ is either a totally real field or a CM field, and the adjoint map $e \mapsto \bar{e}$ defined by $\phi(ex, y) = \phi(x, \bar{e}y)$ is the identity map when $E$ is totally real and is complex conjugation when $E$ is CM. 
\end{theorem}

To discuss motives in families, we first give a definition: 

\begin{definition}
\label{def: maximal monodromy}
    Let $S$ be a connected smooth $\IC$-variety. For every $\IQ$-local system $\sfV_B$ over $S$ and $b \in S(\IC)$, we write $\Mon(\sfV_B, b)$ for the Zariski closure of the image of $\pi_1(S,b)$ in $\GL(\sfV_{B, b})$, and $\Mon^\circ(\sfV_B, b)$ for its identity component. When $\sfV= (\sfV_B, \sfV_\dR)$ is a polarizable VHS\footnote{This means that $\sfV_B$ is the $\IQ$-local system and $\sfV_\dR$ is the filtered flat vector bundle in the VHS. Similar conventions apply throughout the paper.} over $S$, we say that $\sfV$ has \textbf{maximal monodromy} if $\Mon^\circ(\sfV_B, b)$ is equal to the derived group of the Mumford-Tate group $\MT(\sfV_b)$, where $b$ is any Hodge-generic point. 
\end{definition}

For the terminology ``Hodge-generic points'', see for example \cite[31]{Moonen-Fom}. Note that \cite[\S5]{AndreMT} says that in the above notation, $\Mon^\circ(\sfV_B, b)$ is always a normal subgroup of the derived group of $\MT(\sfV_b)$ (see also \cite[Thm~16]{Peters}). 

\subsubsection{} \label{sec: the 4 cases} For the rest of (\ref{sec: summarize Moonen}), let $S$ be a connected smooth $\IC$-variety and $f : \sX \to S$ be a $\heartsuit$-family of relative dimension $d$. Let $\sfV = (\sfV_B, \sfV_\dR)$ be the VHS on $S$ defined by $R^2 f_* \IQ(1)$. Note that by the definition of a $\heartsuit$-family, $\sfV$ satisfies condition (P) in \cite[Prop.~6.4]{Moonen}. Let $\bxi$ be a relatively ample line bundle on $\sX/S$, which defines a symmetric bilinear pairing $\< -, -\>$ on $\sfV$ such that $\< \alpha, \beta \> = \alpha \cup \beta \cup c_1(\bxi)^{d-2}$ for local sections $\alpha, \beta$. Choose a Hodge-generic base point $b \in S(\IC)$.

Suppose for a moment that $\Mon(\sfV_B, b)$ is connected. Then for $\rho := \dim \NS(\sX_b)_\IQ$, $\sfV$ admits an orthogonal decomposition $\mathbf{1}_S^{\oplus \rho} \oplus \sfP$, where $\sfP$ is a VHS polarized by the pairing induced by $\bxi$ and $\mathbf{1}_S$ is the unit VHS. Note that the Hodge structure $\sfP_b = (\sfP_{B, b}, \sfP_{\dR, b})$, together with its polarization, is of K3-type and satisfies the hypothesis of \ref{thm: Zarhin}. Let $E$ be the endomorphism field of $\sfP_b$. As $\Mon(\sfV_B, b) = \Mon^\circ(\sfV_B, b) \subseteq \MT(\sfV_{b}) = \MT(\sfP_b)$ and $\MT(\sfP_{b})$ commutes with $E$, the $E$-action on $\sfP_{B, b}$ commutes with $\pi_1(S, b)$ and hence extends to an action on $\sfP$ (see e.g., \cite[Cor.~12]{Peters}). If $\sfV$ (or equivalently $\sfP$) has maximal monodromy, we say that $(\sX/S, \bxi)$ is in case (R+) if $E$ is totally real and (CM) if $E$ is CM; we further divide (R+) into case (R1) for $\dim_E \sfP_{B, b}$ odd and case (R2) for $\dim_E \sfP_{B, b}$ even. We say we are in case (R2') if $\sfV$ has non-maximal monodromy, which can only happen when $E$ is totally real and $\dim_E \sfP_{B, b} = 4$. See \cite[Prop.~6.4(iii)]{Moonen} and its proof.

If $\Mon(\sfV_B, b)$ is not connected, we say that $(\sX/S, \bxi)$ is in case $?$ for $?$ = (R1), (R2), (CM) or (R2') if it is in case $?$ up to replacing $S$ by a connected \'etale cover $S'$ and $b$ by a lift such that $\Mon(\sfV_B, b)$ becomes connected. The definition is clearly independent of these choices. 

\begin{proposition}
\label{prop: dim = 1 for R2'}
    Suppose that $\sfV$ has non-maximal monodromy, or equivalently $(\sX/S, \bxi)$ belongs to case (R2'). Then for a general $s \in S$ and $X := \sX_s$, the Kodaira-Spencer map $\nabla_s : T_s S \to \Hom(\H^1(\Ohm^1_X), \H^2(\sO_X))$ has rank $1$. 
\end{proposition}
\begin{proof}
    We may assume that $\Mon(\sfV_B, b)$ is conncted, so that there is a decomposition $\mathbf{1}_S^{\oplus \rho} \oplus \sfP$ as above. The rank of $\nabla_s$ achieves its maximum on an open dense $U \subseteq S$. Choose some $s \in U$. By \cite[Thm~3.5]{Voisin} (cf. \cite[Prop.~6.4]{Moonen}), for some point $z$ in a small analytic neighborhood of $s$, $\sfP_{B, z}^{(0, 0)}$ contains a nonzero class $\zeta$. Let $Z' \subseteq S$ be the irreducible component of the Noether-Lefschetz loci defined by $(z, \zeta)$ and $Z$ be its smooth locus. Up to replacing $z$ by a different point on $Z'$, we may assume that $z$ lies in $Z$, is Hodge-generic for the VHS $\sfP|_{Z}$, and $\mathrm{rank\,} \nabla_s = \mathrm{rank\,} \nabla_{z}$. Let $W$ be the Hodge structure of K3-type defined by the fiber of $\sfP$ at $z$ and let $T$ be its transcendental part. Then $\dim_E T < \dim_E W = 4$. Set $F := \End_{\Hdg} T$, which contains $E$. 

    By assumption on a $\heartsuit$-family, $\mathrm{rank\,} \nabla_s > 0$. Suppose by way of contradiction that $\mathrm{rank\,} \nabla_s > 1$. Then as $Z\subseteq S$ has codimension $1$ (cf. \cite[Lem.~3.1]{Voisin}), $\nabla_z$ does not vanish on $T_{z} Z$. This implies that $\Mon^\circ(\sfP_B|_Z, z) \neq 1$ and $\dim_E T = 3$ by \cite[Prop.~6.4]{Moonen} and its proof. Indeed, if $F$ is CM, then $\dim_F T \ge 2$, so that $\dim_E T \ge 4$, which is impossible. Therefore, $F$ is totally real and $\dim_F T \ge 3$. This forces $E = F$ and $\dim_E T = 3$. Now let $\sT$ be the $E$-bilinear lift of $T$. Then by \cite[Thm~2.2.1]{Zarhin} $\MT(T) = \Res_{E/\IQ} \SO(\sT)$. As this group is simple, $\Mon^\circ(\sfP_B|_Z, z) \iso \Res_{E/\IQ} \SO(\sT)$. On the other hand, \cite[\S8.1]{Moonen} tells us that $\Mon^\circ(\sfP_B, s) \iso \Res_{E/\IQ} L$ for some $E$-form $L$ of $\SL_2$. However, as $\SL_2$ and $\SO(3)$ are not even isomorphic over $\IC$, $\Res_{E/\IQ} L$ cannot have a subgroup isomorphic to $\Res_{E/\IQ} \SO(\sT)$, which constradicts the fact that parallel transport (noncanonically) sends $\Mon^\circ(\sfP_B|_Z, z)$ into $\Mon^\circ(\sfP_B, s)$. 
\end{proof}


\subsubsection{} The following statements are the key inputs we will use from Moonen's paper \cite{Moonen}. A little adaptation we make is that we will uniformly use the category of motives $\Mot_\AH(\IC)$ with absolute Hodge cycles, whereas Moonen used Andr\'e's category with motivated cycles (\cite{AndreMot}). This adaptation makes a difference only for \ref{thm: Moonen}(d) below. We use $\Mot(\IC)$ to denote Andr\'e's category (for base field $\IC$) when explaining this difference, but otherwise all motives are considered in $\Mot_\AH(\IC)$. Note that motivated cycles are automatically absolute Hodge, so that $\Mot(\IC)$ is a subcategory of $\Mot_\AH(\IC)$.



Recall our notations in \ref{sec: norm functors} and \ref{sec: the 4 cases} for the theorem below: 

\begin{theorem} \label{thm: Moonen} Assume that $\Mon(\sfV_B, b)$ is connected and let $\sfP$ and $E$ be as in \ref{sec: the 4 cases}. Let $s \in S(\IC)$ be any point and write $\fp_s \in \Mot_\AH(\IC)$ for the submotive of $\fh^2(\sX_s)(1)$ such that $\w_\Hdg(\fp_s) = \sfP_{s}$. Then the action of $E$ on $\sfP_{s}$ is absolute Hodge, i.e., induced by an action of $E$ on $\fp_s$. Moreover:
\begin{enumerate}[label=\upshape{(\alph*)}]
    \item In case (R1), $\Nm_{E/\IQ}(\fp_s) \tensor_\IQ \det(\fp_{s, (\IQ)})$ is an object of $\Mot_\Ab(\IC)$. 
    \item In case (R2), for $\mathbf{1}_E := \mathbf{1} \tensor E$ and $\fp_s^\sharp := \fp_s \oplus \mathbf{1}_E$, $\Nm_{E/\IQ}(\fp^\sharp_s) \tensor_\IQ \det(\fp^\sharp_{s, (\IQ)})$ is an object of $\Mot_\Ab(\IC)$. 
    \item In case (R2'), $\Nm_{E/\IQ}(\fp_s)$ is an object of $\Mot_\Ab(\IC)$. 
    \item In case (CM), $\fp_s$ is an object of $\Mot_\Ab(\IC)$. 
\end{enumerate}
\end{theorem}
\begin{proof}
The statement that the action of $E$ on $\sfP_{s}$ is absolute Hodge is implied by \cite[Prop.~6.6]{Moonen}. Note that our $\fp_s$ is Moonen's $\boldsymbol{V}_s$, and our $\sfP$ is Moonen's $\IV$. Below we view $\fp_s$ as an object of $\Mot_\AH(\IC)_{(E)}$. In this proof, $\underline{\End}$ (or $\underline{\Hom}$) always mean internal $\End$ (or $\Hom$) in the category $\Mot_\AH(\IC)$.

(a) We first treat the case (R+). Let $\sV$ be the $E$-bilinear lift of the quadratic form given by $\sfP_{B, b}$ with its self-dual $E$-action. By \cite[\S6.9, 6.10]{Moonen} there is a family of abelian schemes $\shA \to S$ with multiplication by $D := \Cl^+_{E/\IQ}(\sV)$ (see \ref{not: totally real}) such that there is an isomorphism
\begin{equation}
    \label{eqn: promoted}
    \Cl^+_{E/\IQ}(\fp_s) \sto \underline{\End}_D(\fh^1(\shA_s))
\end{equation}
of algebra objects in $\Mot_\AH(\IC)$. In case (R1), $\dim_E \sV$ is odd, and by \ref{lem: truncate} $\Nm_{E/\IQ}(\fp_s) \tensor \det(\fp_{s, (\IQ)})$ is non-canonically a submotive of $\Cl^+_{E/\IQ}(\fp_s)$, and hence is an object of $\Mot_\Ab(\IC)$. 

(b) In case (R2), (\ref{eqn: promoted}) still holds, so that $\Cl^+_{E/\IQ}(\fp_s)$ is still an object of $\Mot_\Ab(\IC)$, but a further trick is needed to recover (a variant of) $\Nm_{E/\IQ}(\fp_s)$ from $\Cl^+_{E/\IQ}(\fp_s)$. Recall $\sE$ defined in \ref{not: totally real}(d). As in \cite[\S6.11, 6.12]{Moonen}, we consider the VHS $\sfP^\sharp := \sfP \oplus (\mathbf{1}_S \tensor \sE_{(\IQ)})$, where $\mathbf{1}_S$ stands for the unit VHS on $S$. Let $\sV^\sharp := \sV \oplus \sE$ and $D^\sharp := \Cl^+_{E/\IQ}(\sV^\sharp)$. Then by \cite[\S6.9, 6.10]{Moonen} again, there is an abelian scheme $\shA^\sharp \to S$ with multiplication by $D^\sharp$ such that 
\begin{equation}
\label{eqn: iso for E sharp}
\Cl^+_{E/\IQ}(\sfP^\sharp) \iso \underline{\End}_{D^\sharp}(\sfH)
\end{equation}
where $\sfH$ is the VHS given by the first relative cohomology of $\shA^\sharp$. It is shown in \textit{loc. cit.} that the fiber of the isomorphism (\ref{eqn: iso for E sharp}) at every $\IC$-point on $S$ is induced by an absolute Hodge cycle. Here is a summary of the argument in our notations: Choose a point $s_0 \in S(\IC)$ such that $\sfP_{s_0}^{(0, 0)} \neq 0$. By the last paragraph of \cite[\S6.12]{Moonen}, there is an isomorphism 
$$ \Cl^+_{E/\IQ}(\fp^\sharp_{s_0}) \iso \Cl^+_{E/\IQ}(\fp_{s_0})^{\oplus 2^{[E : \IQ]}} $$
of objects in $\Mot_\AH(\IC)$. As $\Cl^+_{E/\IQ}(\fp_{s_0})$ is an object of $\Mot_\Ab(\IC)$, so is $\Cl^+_{E/\IQ}(\fp_{s_0}^\sharp)$. Therefore, by \ref{thm: tensor on AV always AH} the isomorphism (\ref{eqn: iso for E sharp}) is absolute Hodge at $s_0$, and hence so at every other $s$ by Deligne's Principle B. This implies $\Cl^+_{E/\IQ}(\fp^\sharp_s) \in \Mot_\Ab(\IC)$, so that (b) follows from \ref{lem: truncate} again.

(c) This follows directly from \cite[Prop.~8.5]{Moonen}.

(d) In case (CM), \cite[\S7.4]{Moonen} tells us that there exists a motive $\fu$ (denoted by $\boldsymbol{U}$ therein), an abelian variety $A$ over $\IC$ and an abelian scheme $\shB \to S$, all equipped with multiplication by $E$, such that there is an isomorphism 
$$ \fp_s \tensor_E \fu \sto \fh_s := \underline{\Hom}_E(\fh^1(A), \fh^1(\shB_s)). $$
This implies that $\fp_s \tensor \fu \in \Mot_\Ab(\IC)$. In order to show $\fp_s \in \Mot_\Ab(\IC)$, it suffices to argue that $\fu \iso \mathbf{1}_E$ in $\Mot_\AH(\IC)$.

In \textit{loc. cit.}, $\fu$ is in fact constructed as an object $\Mot(\IC)$. Moonen remarked that conjecturally there should be an isomorphism $\fu \iso \mathbf{1}_E$ in $\Mot(\IC)$ and proved that $\fu$ indeed has trivial Hodge and $\ell$-adic realizations. We note that his argument in \cite[Lem.~7.5]{Moonen} in fact implies that $\fu \iso \mathbf{1}_E$ in $\Mot_\AH(\IC)$, i.e., $\w_B(\fu)$ is spanned by absolute Hodge classes: The idea is to take advantage of the fact that $\fu$ is independent of $s$, and that for some $s_0 \in S$, the algebraic part $\w_B(\fp_{s_0})^{(0, 0)}$ of $\w_B(\fp_{s_0})$ is nonempty. Since every class in $\w_B(\fu)$ is of type $(0, 0)$, we have $\w_B(\fp_{s_0})^{(0, 0)} \tensor_E \w_B(\fu) = \w_B(\fh_{s_0})^{(0, 0)}$. By Lefschetz $(1, 1)$-theorem, every class in $\w_B(\fp_{s_0})^{(0, 0)}$ comes from a line bundle on $\sX_s$ and hence is absolute Hodge. As classes in $\w_B(\fh_{s_0})^{(0, 0)}$ are also absolute Hodge, we may now conclude by \ref{prop: N is abs Hdg} below. 
\end{proof}

\begin{proposition}
\label{prop: N is abs Hdg}
Let $E$ be a number field and let $\fm, \fn$ be objects of $\Mot_\AH(\IC)$ with $E$-action. Let $\fh \coloneqq \fm \tensor_E \fn$. Let $m \in \w_B(\fm), n \in \w_B(\fn)$ be nonzero Hodge cycles and define $h \in \w_B(\fh)$ to be $m \tensor_E n$. If $h$ and $m$ are both absolute Hodge, then so is $n$. 
\end{proposition}
\begin{proof}
Recall notations in \ref{sec: twist tensors}. We need to show that for every $\sigma \in \Aut(\IC)$, the image $n'$  of $n \tensor 1$ under the canonical isomorphisms 
$$ \w_B(\fn) \otimes (\IC \otimes \IA_f) \cong \w_\dR(\fn) \times \w_{\IA_f}(\fn) \iso_\bc \w_\dR(\fn^\sigma) \times \w_{\IA_f}(\fn^\sigma) \cong \w_B(\fn^\sigma) \otimes (\IC \otimes \IA_f). $$
is contained in $\w_B(\fn^\sigma)$ (i.e., is of the form $n^\sigma \tensor 1$ for some $n^\sigma \in \w_B(\fn^\sigma)$). Since $m$ and $h$ are absolute Hodge, we know that $m^\sigma \in \w_B(\fm^\sigma)$ and $h^\sigma = m^\sigma \otimes n' \in \w_B(\fh^\sigma)$. Then apply \ref{lem: descend n to Q} with $k = \IQ$, $R = \IA_f \times \IC$, $M = \w_B(\fm^\sigma)$ and $N = \w_B(\fn^\sigma)$ to deduce that $n' \in \w_B(\fn^\sigma)$. 
\end{proof}

\begin{lemma}
\label{lem: descend n to Q}
Let $E/k$ be a field extension, $R$ be any $k$-algebra and $M, N$ be finite dimensional $E$-vector spaces. Let $H \coloneqq M \tensor_E N$ and $h \in H \tensor_k R$ be a nonzero element of the form $m \tensor n$ under the canonical isomorphism 
$$ H \tensor_k R \iso (M \tensor_k R) \tensor_{E \tensor_k R} (N \tensor_k R), $$
where $m \in M \tensor_k R$ and $n \in N \tensor_k R$. If $h \in H$ and $m \in M$, then $n \in N$. 
\end{lemma}
\begin{proof}
 Let $\alpha = \dim_E M$ and $\beta = \dim_E N$. By choosing bases of $M$ and $N$, we may assume $M = E^{\alpha}$ and $N = E^{\beta}$. Identify $H$ with the space of $(\alpha \times \beta)$-matrices over $E$ and $h$ with $m \cdot n^T$. We denote by $(m_i),(n_j),(h_{i,j})$ the respective $E \tensor_k R$-coordinates of $m,n$ and $h$ and fix $i$ such that $m_i \not= 0$. Then for every $j$, $h_{i,j} = m_i \cdot n_j$ and hence $n_j = m_i^{-1} h_{i,j}  \in E$ because $m_i,h_{i, j} \in E \subseteq E \tensor_k R$ by assumption. 
\end{proof}

\section{Moduli interpretations of Shimura varieties}

In this section, we first recall the moduli description for the canonical models of some Shimura varieties of abelian type from \cite{YangSystem}. Then we give some prelimary results on those of orthogonal type over totally real fields and review their integral models when the reflex field is $\IQ$.  


\subsection{Systems of realizations}
\begin{definition}
\label{def: system of realizations}
    Let $k$ be a subfield of $\IC$ and $S$ be a smooth $k$-variety. By a system of realizations we mean a tuple $\sfV = (\sfV_B, \sfV_\dR, \sfV_\et, i_\dR, i_\et)$ where 
    \begin{itemize}
        \item $\sfV_B$ is a $\IQ$-local system over $S_\IC := S \tensor_k \IC$; 
        \item $\sfV_\dR$ is a filtered flat vector bundle over $S$; 
        \item $\sfV_\et$ is an \'etale local system of $\IA_f$-coefficients over $S$; 
        \item $i_\dR : (\sfV_B \tensor \sO^\an_{S_\IC}, \mathrm{id} \tensor d) \sto (\sfV_\dR|_{S_\IC})^\an$ is an isomorphism of  flat (holomorphic) vector bundles such that $(\sfV_B, \sfV_\dR|_{S_\IC})$ is a \textit{polarizable} VHS;
        \item $i_\et : \sfV_B  \tensor \IA_f \sto \sfV_\et|_{S_\IC}$ is an isomorphism between (the pro-\'etale sheaf associated to) $\sfV_B  \tensor \IA_f$ and $\sfV_\et|_{S_\IC}$.
    \end{itemize}  
\end{definition}
We may often omit $(i_\dR, i_\et)$ in the notation and write $\sfV_B \tensor \IA_f = \sfV_\et|_{S_\IC}$ and $\sfV_B \tensor \sO_{S_\IC}^\an = \sfV_\dR|_{S_\IC}$ a little abusively. Let $\sfR(S)$ denote the category of systems of realizations over $S$, with morphisms defined in the obvious way. Then $\sfR(S)$ is a naturally a Tannakian category. Let $\mathbf{1}_S$ be its unit object and for each $n \in \IZ$ let $\mathbf{1}_S(n)$ be the Tate object. For every $\sfV \in \sfR(S)$, set $\sfV(n) := \sfV \tensor \mathbf{1}_S(n)$. Note that if $k = \IC$, then $\sfR(S)$ is naturally identified with the category of polarizable VHS over $S$. We write $\H^0(\sfV)$ for $\Hom(\mathbf{1}, \sfV)$.

Our definition is a little different from \cite[\S6.1]{FuMoonen} because we \textit{fixed an embedding $k \into \IC$}, but see \cite[(3.1.3)]{YangSystem} for a comparison. Also, recall our convention in \ref{sec: notations}(e) for VHS. We remind the reader that since $S$ is defined over $k \subseteq \IC$, when we write $s \in S(\IC)$, we mean a \textit{$k$-linear} morphism $\Spec(\IC) \to S$, which we view simultaneously as a \textit{closed point} on $S_\IC = S \tensor_k \IC$.


Next, we recall how to define a $\sfK$-level structure in this context. Interestingly this definition can be leveraged to define additional structures on $\sfV$, as \ref{ex: PEL AV} below shows.

\begin{definition}
\label{def: level structures}
 Let $G$ be a reductive group over $\IQ$, let $V \in \mathrm{Rep}(G)$ be a finite dimensional $G$-representation and $\sfK \subseteq G(\IA_f)$ be a compact open subgroup. For any subfield $k \subseteq \IC$, smooth $k$-variety $S$ and any system of realizations $\sfV = (\sfV_B,\sfV_\dR,\sfV_\et)$ on $S$ we make the following definitions.
 \begin{enumerate}[label=\upshape{(\alph*)}]
     \item A $(G,V,\sfK)$-level structure on $\sfV_\et$, or simply a $\sfK$-level structure, is a global section $[\eta]$ of the quotient sheaf $\sfK \backslash \Isom(V\tensor \IA_f, \sfV_\et)$. Here $\Isom(V\tensor \IA_f, \sfV_\et)$ denotes the pro-\'etale sheaf consisting of isomorphisms $V \tensor \IA_f \sto \sfV_\et$, and $\sfK$ acts by pre-composition through its image in $\GL(V\tensor\IA_f)$. For each geometric point $s \to S$, write the $\sfK$-orbit of isomorphisms $V \tensor \IA_f \sto \sfV_{\et, s}$ determined by $[\eta]$ as $[\eta]_s$. 
     \item If $[\eta]$ is a $(G,V,\sfK)$-level structure on $\sfV_\et$, every element $v \in I := (V^\otimes)^G$ gives rise to a global section of $\sfV_\et$, which we denote by $\eta(v)_\et$, such that for every geometric point $s \to S$, every representative of $[\eta]_s$ sends $v \tensor 1$ to $\eta(v)_{\et, s}$. We say that $[\eta]$ is $\sfV$-rational if (i) for all $v \in I$, there exists a global section $\bv = (\bv_B, \bv_\dR, \bv_\et) \in \H^0(\sfV^\tensor)$ such that $\bv_\et = \eta(v)_\et$, and (ii) for every $s \in S(\IC)$, $(\sfV_{B, s}, \{ \bv_{B, s}\}_{v \in I}) \iso (V, \{ v\}_{v \in I})$. 
     \item Let $\Omega$ be a $G(\IR)$-conjugacy class of morphisms $\IS \to G_\IR$. When $V$ is faithful, we say that a $\sfV$-rational $\sfK$-level structure is of type $\Omega$ if for every $s \in S(\IC)$, under some (and hence every) isomorphism $(\sfV_{B, s}, \{ \bv_{B, s}\}_{v \in I}) \iso (V, \{ v\}_{v \in I})$ the Hodge structure on $\sfV_{B, s}$ is defined by an element of $\Omega$.
     \item If there is a morphism $G' \to G$ from another reductive subgroup $G'$, and $\sfK' \subseteq G'(\IA_f)$ is a compact open subgroup whose image is contained in $\sfK$, then we say that a $\sfK'$-level structure $[\eta']$ on $\sfV_\et$ refines $[\eta]$ if $[\eta]$ is induced by $[\eta']$ via the natural (forgetful) map $\sfK' \backslash \Isom(V \tensor \IA_f, \sfV_\et) \to \sfK \backslash \Isom(V \tensor \IA_f, \sfV_\et)$. 
     \end{enumerate}
\end{definition}

The definitions (b) and (c) above are used to simplify the formalism of the moduli interpretation of Shimura varieties. To explain this, we give a PEL-type example:

\begin{example} \label{ex: PEL AV}
    Let $(B,\ast,(V,\psi))$ be a simple PEL-datum (i.e., $B$ is a simple $\IQ$-algebra with positive involution $\ast$ of type $A$ or $C$ and $(V,\psi)$ is a symplectic $B$-module). We denote by $G \coloneqq \GSp_B(V)$ the $\IQ$-group of $B$-linear similitudes and fix an open compact subgroup $\sfK \subset G(\IA_f)$. Then there exists a unique $G(\IR)$-conjugacy class $\Omega$ such that $(G,\Omega)$ is a Shimura datum and $\Sh_\sfK(G,\Omega)(\IC)$ is in canonical bijection with the set $\sI$ of isomorphism classes of tuples $(\sfV,\lambda,i,[\eta])$ where
    \begin{itemize}
        \item $\sfV$ is a $\IQ$-Hodge structure $(\sfV_B,\sfV_{\dR})$ of type $\{(1,0),(0,1)\}$,
        \item $i\colon B \hookrightarrow \End(\sfV)$ is an algebra morphism,
        \item $\lambda$ is a $\IQ^\times$-equivalence class of $B$-linear symplectic pairing of $\sfV$ and
        \item $[\eta]$ is a $\sfK$-orbit of a $B$-linear similitude $V \otimes \IA_f \to \sfV_\et$ 
    \end{itemize}
    such that $(\spadesuit)$ there exists a $B$-linear similitude $V \iso \sfV_B$ through which the Hodge structure on $\sfV$ is defined by an element of $\Ohm$ (see e.g.\ \cite[Prop.~8.14, Thm.~8.17]{MilIntro}). Here we implicitly applied the equivalence $A \mapsto \H^1(A,\IQ)$ between the category of complex abelian varieties up to isogeny and that of polarizable Hodge structures of type $\{(1,0),(0,1)\}$. 

    Note that an object of $\sfR(\IC) := \sfR(\Spec(\IC))$ is nothing but a polarizable Hodge structure. We can more concisely define $\sI$ as the isomorphism classes of pairs $(\sfV, [\eta])$ where 
    \begin{itemize}
        \item $\sfV \in \sfR(\IC)$ and
        \item $[\eta]$ is a $\sfV$-rational $(G, V, \sfK)$-level structure on $\sfV_\et$ of type $\Ohm$.
    \end{itemize}
    Indeed, if we view each $b \in B$ as a tensor in $\End_\IQ(V) = V^\vee \otimes V$ and $c \coloneqq \IQ^\times \cdot \psi$ as a tensor in $V^2 \otimes (V^\vee)^2$ (cf.~\cite[Ex.~2.1.6]{Kim:RZ}), we have $i(b) \tensor \IA_f = \eta(b)_\et$ and $\lambda \tensor \IA_f = \eta(c)_\et$ in the notation of \ref{def: level structures}(b). In particular, the datum of $i$ and $\lambda$ is remembered by the condition that $\eta(b)_\et$ and $\eta(c)_\et$ come from global sections (i.e., elements) of $\sfV_B^\otimes$ of Hodge type $(0,0)$. Since $G$ is the stabilizer of $B$ and $c$ in $\GL(V)$, we may equivalently say that  this is true for all $v \in (V^\otimes)^G$. Therefore, the existence of $i$ and $\lambda$ such that there exists a $B$-linear similitude $V \iso \sfV_B$ is equivalent to $[\eta]$ being $\sfV$-rational. Then the condition $(\spadesuit)$ can be simply stated as ``$[\eta]$ is of type $\Omega$'' in the sense of \ref{def: level structures}(c).
\end{example}




Now we introduce some notation to keep track of Galois descent data in a system of realizations: 
\begin{notation} \label{not: galois descent} Let $S$ be a smooth variety over a subfield $k$ of $\IC$. Let $s \in S(\IC)$ and $\sigma \in \Aut(\IC/k)$ be any elements. Denote by $\sigma(s) \in S(\IC)$ the point given by pre-composing the $k$-linear morphism $s : \Spec(\IC) \to S$ with $\Spec(\sigma)$. Given $(\sfV_B, \sfV_\et, \sfV_\dR) \in \sfR(S)$, we write $\sigma_{\sfV_\et, s} : \sfV_{B, s} \tensor \IA_f \sto \sfV_{B, \sigma(s)} \tensor \IA_f$ the natural isomorphism induced by $\sfV_\et$, viewed as a descent of $\sfV_B \tensor \IA_f$ from $S_\IC$ to $S$; similarly, we write $\sigma_{\sfV_\dR, s} : \sfV_{B, s} \tensor \IC \sto \sfV_{B, \sigma(s)} \tensor \IC$ for the natural $\sigma$-linear isomorphism induced by the descent $\sfV_\dR$ of (the algebraization of) $\sfV_B \tensor \sO_{S_\IC}^{\mathrm{an}}$. 
\end{notation}

To clarify the meaning of $\sigma_{\sfV_\et, s}$, we remark that $s$ and $\sigma(s)$ are usually different closed points on $S_\IC$, and they are equal if and only if $\sigma$ fixes the residue field $k(s_0)$, where $s_0 \in S$ is the image of $s$. The collection of $\sigma_{\sfV_\et, s}$ as $\sigma$ runs through $\Aut(k(s)/k(s_0))=\Aut(\IC/k(s_0))$ is nothing but the Galois action on the stalk $\sfV_{B, s} \tensor \IA_f = \sfV_{\et, s}$.

Using the notations of \ref{sec: twist tensors} and \ref{not: galois descent}, we define: 
\begin{definition}
\label{def: compatible with AM}
 We say that a system of realizations $\sfV$ is \textbf{weakly abelian-motivic (weakly AM)} if for every $s \in S(\IC)$, there exists $M \in \Mot_\Ab(\IC)$ such that $\w_{\Hdg}(M) \iso (\sfV_{B, s}, \sfV_{\dR, s})$; moreover, for any such isomorphism $\gamma : \w_{\Hdg}(M) \sto (\sfV_{B, s}, \sfV_{\dR, s})$ and $\sigma \in \Aut(\IC/k)$, there exists an isomorphism $\gamma^\sigma : \w_\Hdg(M^\sigma) \sto (\sfV_{B, \sigma(s)}, \sfV_{\dR, \sigma(s)})$ such that the Betti components $\gamma_{B}, \gamma^\sigma_{B}$ of $\gamma, \gamma^\sigma$ fit into commutative diagrams:
    \begin{equation}
    \label{diag: defining AM compatible}
        \begin{tikzcd}
	{\w_{\et}(M) } & {\sfV_{B,s} \tensor \IA_f} \\
	{\w_{\et}(M^\sigma)} & {\sfV_{B, \sigma(s)} \tensor \IA_f}
	\arrow["{\iso_{\mathrm{bc}}}"', from=1-1, to=2-1]
	\arrow["{\sigma_{\sfV_\et, s}}", from=1-2, to=2-2]
	\arrow["{\gamma^\sigma_{B} \tensor \IA_f}", from=2-1, to=2-2]
	\arrow["{\gamma_{B} \tensor \IA_f}", from=1-1, to=1-2]
    \end{tikzcd}
    \text{ and }
    \begin{tikzcd}
	{\w_{\dR}(M) } & {\sfV_{B,s} \tensor \IC} \\
	{\w_{\dR}(M^\sigma)} & {\sfV_{B, \sigma(s)} \tensor \IC}
	\arrow["{\iso_{\mathrm{bc}}}"', from=1-1, to=2-1]
	\arrow["{\sigma_{\sfV_\dR, s}}", from=1-2, to=2-2]
	\arrow["{\gamma^\sigma_B \tensor \IC}", from=2-1, to=2-2]
	\arrow["{\gamma_B \tensor \IC}", from=1-1, to=1-2]
    \end{tikzcd}
    \end{equation} 
    We note that $\gamma^\sigma$ is uniquely determined by $\gamma$ provided that it exists. Denote the full subcategory of $\sfR(S)$ given by these objects by $\sfR^*_\am(S)$. It is easy to check that if $T \to S$ is a morphism between smooth $k$-varieties, the natural pullback functor $\sfR(S) \to \sfR(T)$ sends $\sfR^*_\am(S)$ to $\sfR^*_\am(T)$. 
\end{definition}

\begin{lemma}
\emph{(\cite[(3.4.1)]{YangSystem})}
\label{lem: descent of morphisms of systems} 
    Let $S$ be a smooth variety over $k \subseteq \IC$ and take $\sfV, \sfW \in \sfR_\am^*(S)$. Let $\varphi_\IC$ be a morphism $\sfV|_{S_\IC} \to \sfW|_{S_\IC}$. Then $\varphi_\IC$ descends to a morphism $\sfV \to \sfW$ if and only if either the \'etale or the de Rham component of $\varphi_\IC$ descends to $S$. 
\end{lemma}

When applied to the case $\sfV = \mathbf{1}_S$, the above lemma says that for $\sfW \in \sfR^*_\am(S)$, if the \'etale realization of a global section of $\sfW_B|_{S_\IC}$ which is everywhere of Hodge type $(0,0)$ descends to $S$, then so does the de Rham realization, and vice versa. This is a global version of the following statement: Suppose that $M \in \Mot_\AH(k)$ and $v \in \w_B(M_\IC)$ is absolute Hodge. Then $v \tensor \IA_f \in \w_\et(M_\IC)$ descends to $k$ (i.e., is $\Aut(\IC/k)$-invariant) if and only if $v \tensor \IC \in \w_\dR(M_\IC)$ descends to $\w_{\dR}(M)$. Note that if $M \in \Mot_\Ab(k)$, then any Hodge cycle $v \in \w_B(M_\IC)$ is automatically absolute Hodge.

For future reference we give a handy lemma on (b) and (c) in \ref{def: level structures}.  
\begin{lemma}
\label{lem: situation in practice}
    Suppose that in \ref{def: level structures}(a), the representation $V \in \mathrm{Rep}(G)$ is faithful, $S$ is geometrically connected as a $k$-variety, and for some $b \in S(\IC)$ there is an isomorphism $\eta_b : V \sto \sfV_{B, b}$ such that $[\eta]_b = \sfK \cdot (\eta_b \tensor \IA_f)$; moreover, for some $\Ohm$ as in \ref{def: level structures}(c), the Hodge structure on $\sfV_{B, b}$ is defined by an element of $\Ohm$ via $\eta_b$. 
    
    Then, assuming either $k = \IC$ or $\sfV \in \sfR^*_\am(S)$, $[\eta]$ is $\sfV$-rational and of type $\Ohm$. 
\end{lemma}
\begin{proof}
    Note that by assumption $S_\IC$ is connected, and the fact that $\eta_b \tensor \IA_f$ represents a $\sfK$-level structure at $b$ implies that its $\sfK$-orbit is $\pi^\et_1(S, b)$-stable. In particular, for each $v \in (V^\tensor)^G$, $\pi_1(S_\IC, b)$ fixes $\eta_b(v)$, so there exists a $\bv_B \in \H^0(\sfV_B^\tensor)$ such that $\bv_{B, b} = \eta_b(v)$. Likewise, there exists $\bv_\et \in \H^0(\sfV_\et^\tensor)$ such that $\bv_{\et, b} = \eta_b(v) \tensor 1$, and $\bv_\et$ is precisely the $\eta(v)_\et$ in \ref{def: level structures}(b). 

    Note that $\bv_B$ is necessarily of Hodge type $(0, 0)$ at $b$, because the Mumford-Tate group of the Hodge structure $(\sfV_{B, b}, \sfV_{\dR, b})$ is contained in $G$ via $\eta_b$. This implies that $\bv_B$ is of Hodge type $(0, 0)$ everywhere by the theorem of the fixed part. Let $\bv_{\dR, \IC}$ be the global section in $(\sfV_B \tensor \sO_{S_\IC}^{\mathrm{an}})^\tensor$ induced by $\bv_B$. Then $\bv_{\dR, \IC}$ algebraizes to a global section of $(\sfV_\dR|_{S_\IC})^\tensor$ (cf. \cite[II~Thm~5.9]{DelVB}). 
    
    If $k = \IC$ (so that $S = S_\IC$), then we have already shown that $[\eta]$ is $\sfV$-rational; moreover, one deduces from the connectedness of $S_\IC$ and \cite[(1.1.12)]{DelVdShimura} that $[\eta]$ remains of type $\Ohm$ on $S_\IC$. If $k \subseteq \IC$ is a general subfield, it remains to show that $\bv_{\dR, \IC}$ descends to $S$ under the additional assumption that $\sfV \in \sfR^*_\am(S)$. However, as the \'etale realization of $\bv_B$ descends to a global section of $\sfV_\et^\tensor$ (i.e., $\bv_\et$), this follows from \ref{lem: descent of morphisms of systems}. 
     \end{proof}

\subsection{Shimura varieties}
Let $(G, \Omega)$ be a Shimura datum which satisfies the axioms in \cite[II~(2.1)]{Milne:CanonicalModels}. Let $E(G, \Ohm)$ be the reflex field, $Z$ be the center of $G$ and $Z_s$ be the maximal anisotropic subtorus of $Z$ that is split over $\IR$. In this paper, we always assume that 
\begin{equation}
\label{eqn: assumption on center}
    \text{the weight is defined over $\IQ$ and $Z_s$ is trivial}. 
\end{equation}
Note that in particular the latter condition ensures that $Z(\IQ)$ is discrete in $Z(\IA_f)$ (\cite[Rmk~5.27]{MilIntro}). We will often drop the Hermitian symmetric domain $\Ohm$ from the notation of Shimura varieties when no confusion would arise. 

For any compact open subgroup $\sfK \subseteq G(\IA_f)$, let $\Sh_\sfK(G)_\IC$ denote the resulting Shimura variety with a complex uniformization $G(\IQ) \backslash \Ohm \times G(\IA_f) / \sfK$ and let $\Sh_\sfK(G)$ denote the canonical model over $E(G, \Ohm)$. Let $\Sh(G)$ denote the inverse limit $\varprojlim_\sfK \Sh_\sfK(G)$ as $\sfK$ runs through all compact open subgroups. Under our assumptions, $\Sh(G)(\IC)$ is described by $G(\IQ) \backslash \Ohm \times G(\IA_f)$ (\cite[(5.28)]{MilIntro}). Note that $\Sh_\sfK(G) = \Sh(G)/ \sfK$.

\subsubsection{} \label{sec: automorphic sheaves} Let $G \to \GL(V)$ be a representation. For any neat compact open subgroup $\sfK \subseteq G(\IA_f)$, we can attach to $\Sh_\sfK(G)_\IC$ an \textbf{automorphic VHS} $(\sfV_B, \sfV_{\dR, \IC})$ (cf. \cite[Ch. II~3.3]{Milne:CanonicalModels}, \cite[\S2.2]{Taelman2}). In particular, $\sfV_B$ is defined to be the contraction product 
\begin{equation}
    \label{eqn: contraction for V_B}
    \sfV_B := V \times^{G(\IQ)} [\Ohm \times G(\IA_f) / \sfK]
\end{equation}
The filtration on $\sfV_{\dR, \IC}$ is obtained by descending the filtration on the tautological VHS on $V \times \Ohm$. Analogously, the \textbf{automorphic \'etale local system} $\sfV_\et$ on the pro\'etale site of $\Sh_\sfK(G)$ is defined as the contraction product 
\begin{equation}
    \label{eqn: contraction for V_et}
    \sfV_\et := V_{\IA_f} \times^\sfK \Sh(G)
\end{equation}
which comes with a comparison isomorphism $\sfV_B \otimes \IA_f \iso \sfV_\et|_{\Sh_\sfK(G)_\IC}$. Moreover, by construction $\sfV_\et$ over $\Sh_\sfK(G)$ comes with a \textbf{tautological $\sfK$-level structure} $[\eta_V]$, i.e., a global section of $\sfK \backslash \Isom(V \tensor \IA_f, \sfV_\et)$.  

Define $[\eta^\an_V]$ to be the restriction of $[\eta_V]$ to ${\Sh_\sfK(G)_\IC}$. Using $(\sfV_B, \sfV_{\dR, \IC})$, we can already give a moduli interpretation of $\Sh_\sfK(G)_\IC$. Below is a reformulation of \cite[Prop.~3.10]{Mil94} in our terminology (cf. \cite[(4.1.2)]{YangSystem})\footnote{Technically, the example discussed in \ref{ex: PEL AV} may not satisfy (\ref{eqn: assumption on center}), but it helps illustrate how to compare our formulation and Milne's.}.
\begin{theorem}
\label{thm: moduli over C}
    Assume that $V$ is faithful. For every smooth $\IC$-variety $T$, let $\sM_V(T)$ be the groupoid of tuples of the form $(\sfW, [\xi^\an])$ where $\sfW = (\sfW_B, \sfW_\dR)$ is a VHS over $T$ and $[\xi^\an]$ is a $\sfK$-level structure on $\sfW_B \tensor \IA_f$ which is $\sfW$-rational and is of type $\Ohm$. 
    
    Then $(\sfV_\IC := (\sfV_B, \sfV_{\dR, \IC}),  [\eta^\an_V])$ is an object of $\sM_V(\Sh_\sfK(G)_\IC)$ and for every object $(\sfW, [\xi^\an]) \in \sM_V(T)$ there exists a unique morphism $\rho : T \to \Sh_\sfK(G)_\IC$ such that $\rho^* (\sfV_\IC, [\eta^\an_V]) \iso (\sfW, [\xi^\an])$. 
\end{theorem} 

We remark that as $\sfK$ is neat, assumption (\ref{eqn: assumption on center}) ensures that the objects in $\sM_V(T)$ above have no nontrivial automorphisms (cf. \cite[Rmk~3.11]{Mil94}).  


To describe the moduli problem for the canonical model, we restrict to the following subclass of Shimura data, which contains all cases we will consider in the following chapters.
\begin{assumption} \label{assumption: Abelian type}
    There exists a morphism of Shimura data $(\wt{G}, \wt{\Ohm}) \to (G, \Ohm)$ such that 
    \begin{enumerate}[label=\upshape{(\roman*)}]
        \item $(\wt{G}, \wt{\Ohm})$ is of Hodge type and also satisfies assumption (\ref{eqn: assumption on center});
        \item $\wt{G} \to G$ is surjective and the kernel lies in the center of $\wt{G}$; 
        \item the embedding of reflex fields $E(G, \Ohm) \subseteq E(\wt{G}, \wt{\Ohm})$ is an equality.
    \end{enumerate}
\end{assumption}
Below we assume that $(G,\Omega)$ is a Shimura datum which satisfies the above assumptions. 

\subsubsection{}
\label{def: automorphic systems} 
For any representation $G \to \GL(V)$, and $\sfV_B, \sfV_{\dR, \IC}, \sfV_\et$ set up in \ref{sec: automorphic sheaves}, there exists a unique descent $\sfV_\dR$ of $\sfV_{\dR, \IC}$ to $\Sh_\sfK(G)$ such that $(\sfV_B, \sfV_\dR, \sfV_\et)$ is a weakly AM system of realizations (see (\cite[(4.2.2)]{YangSystem})). We call $\sfV := (\sfV_B, \sfV_\dR, \sfV_\et) \in \sfR^*_{\mathrm{am}}(\Sh_\sfK(G))$ the \textbf{automorphic system (of realizations)} on $\Sh_\sfK(G)$ attached to $V \in \mathrm{Rep}(G)$. The tautological $\sfK$-level structure $[\eta_V]$ on $\sfV$ is $\sfV$-rational, and is of type $\Ohm$ when $V$ is faithful.

\begin{theorem} \label{thm: shimura as moduli}
\emph{(\cite[(4.3.1)]{YangSystem})}
    Assume that $V$ is faithful. Let $E' \subseteq \IC$ be a subfield which contains $E = E(G, \Ohm)$ and $T$ be a smooth $E'$-variety. Let $\sM_{V, E'}(T)$ be the groupoid of pairs $(\sfW, [\xi])$ where $\sfW = (\sfW_B, \sfW_\dR, \sfW_\et) \in \sfR^*_\am(T)$, and $[\xi]$ is a $\sfK$-level structure on $\sfW_\et$ such that $[\xi]$ is $\sfW$-rational and $(\sfW, [\xi])$ is of type $\Ohm$. 
    
    Then $(\sfV = (\sfV_B, \sfV_\dR, \sfV_\et), [\eta_V])$ is an object of $\sM_{V, E}(\Sh_\sfK(G))$, and for each $(\sfW, [\xi]) \in \sM_{V, E'}(T)$, there exists a unique $\rho : T \to \Sh_\sfK(G)_{E'}$ such that $(\sfW, [\xi]) \iso \rho^* (\sfV, [\eta_V])$. 
\end{theorem}
Note that any object in $\sM_{V, E'}(T)$ above has no nontrivial automorphisms because this is already true when $E' = \IC$. This is a key fact that we shall use repeatedly throughout the paper. The above is proved by first constructing a morphism $\rho_\IC : T_\IC \to \Sh_\sfK(G)_\IC$ such that $(\sfW, [\xi])|_{T_\IC} \iso \rho_\IC^*(\sfV, [\eta_V])$ and then showing that the action of $\rho_\IC$ on the $\IC$-points are $\Aut(\IC/E')$-equivariant. This is clearly inspired by the proof of \cite[Cor.~5.4]{MPTate}. However, unlike \cite[Prop.~5.6(1)]{MPTate}, we show that $(\sfW, [\xi])|_{T_\IC} \iso \rho_\IC^*(\sfV, [\eta_V])$ descends over $T$ using a rigidity lemma about weakly-AM systems of realizations (see \cite[(3.4.4), (3.4.5) and (4.3.2)]{YangSystem}).


\subsection{Orthogonal Shimura varieties over totally real fields}
\label{sec: totally real SV}
\subsubsection{} \label{sssect: quadratic forms setup}
Let $E$ be a totally real number field and let $\sV$ be a quadratic form over $E$ which has signature $(2, \dim_E \sV -2)$ at a unique real place $\tau$ and is negative definitive at every other real place. We set $\sG:= \Res_{E/\IQ} \SO(\sV)$ and
\[
    \Omega_\sV \coloneqq \{ v \in \IP(\sV \tensor_{\tau} \IC) \mid \< v, v \> = 0, \< v, \bar{v} \> > 0 \}.
\] 

Let $V$ be the transfer $\sV_{(\IQ)}$ (recall \ref{not: totally real}). Set $G := \SO(V)$ and define $\Ohm := \Ohm_V$ as above (applied to the $E = \IQ$ case). Note that $\sG$ is the identity component of the centeralizer of the $E$-action in $G$. We view $\Ohm$ as a $G(\IR)$-conjugacy class of morphisms $\IS \to G_\IR$ such that an element $v \in \Ohm$ corresponds to the morphism $h$ which gives $V$ a Hodge structure of K3-type with $v = V_\IC^{(1, -1)}$ (cf. \cite[\S3.1]{CSpin}). Likewise we view $\Ohm_\sV$ as a $\sG(\IR)$-conjugacy class of morphisms $\IS \to \sG_\IR$, which are those whose composition with $\sG_\IR \to G_\IR$ defines a Hodge structure on $V$ which is preserved by the $E$-action on $V = \sV_{(\IQ)}$. It is well-known that $(\sG,\Omega_\sV)$ is a Shimura datum with reflex field $E$, viewed as a subfield of $\IC$ under $\tau$. 

Let $\wt{G}$ denote $\CSpin(V)$. Then $(\wt{G}, \Ohm)$ is a Shimura datum of Hodge type with reflex field $\IQ$ which admits a natural morphism to $(G, \Ohm)$. 

\begin{lemma}
    $(\sG,\Omega_\sV)$ satisfies Assumption~\ref{assumption: Abelian type}.
\end{lemma}
\begin{proof}
    Consider $\wt{\sG}' := \Res_{E/\IQ} \CSpin(\sV)$. It is well known that $(\wt{\sG}', \Ohm_\sV)$ is a Shimura datum with reflex field $E$. Unfortunately, it does not satisfy condition~\eqref{eqn: assumption on center} unless $E=\IQ$ as the center $Z'$ of $\wt{\sG}'$ has identity component $\Res_{E/\IQ} \IG_{m, E}$. Thus we modify this approach by dividing out the maximal anisotropic torus of $\Res_{E/\IQ} \IG_{m, E}$; explicitly we consider ($\Nm=$ the norm map)
    \[
        \wt\sG \coloneqq \wt{\sG}' / \ker(\Nm \colon \Res_{E/\IQ}\IG_{m, E} \to \IG_m)).
    \]
    Note that $\ker(\wt{\sG'} \to \sG) = \Res_{E/\IQ} \IG_{m, E}$, so we obtain a morphism $(\wt\sG,\Omega_\sV) \to (\sG,\Omega_\sV)$ of Shimura data. It remains to check (i) and (iii) in \ref{assumption: Abelian type}. First, note that $\wt{\sG}$ can be canonically identified with the fiber product $\sG \times_G \wt{G}$ (see \cite[\S4.2, 4.3]{Moonen}), so that there is a diagram with exact rows
    \[\begin{tikzcd}
	1 & {\IG_m} & {\wt{\sG}} & \sG & 1 \\
	1 & {\IG_m} & {\wt{G}} & G & 1.
	\arrow[from=1-1, to=1-2]
	\arrow[from=1-2, to=1-3]
	\arrow[from=1-3, to=1-4]
	\arrow[from=1-4, to=1-5]
	\arrow[equal, from=1-2, to=2-2]
	\arrow[hook, from=1-3, to=2-3]
	\arrow[from=2-3, to=2-4]
	\arrow[hook, from=1-4, to=2-4]
	\arrow[from=2-2, to=2-3]
	\arrow[from=2-4, to=2-5]
	\arrow[from=2-1, to=2-2]
	\arrow["\lrcorner"{anchor=center, pos=0.125}, draw=none, from=1-3, to=2-4]
    \end{tikzcd}\]
    Now $Z_s(\wt{\sG}) = 1$ is clear. As $(\wt{G}, \Ohm)$ is of Hodge type and $(\wt{\sG}, \Ohm_\sV)$ embeds into $(\wt{G}, \Ohm)$, $(\wt{\sG}, \Ohm_\sV)$ is also of Hodge type. As the reflex fields of $(\sG, \Ohm_\sV)$ and $(\wt{\sG}', \Ohm_\sV)$ are both well known to be $E$, the same must be true for $(\wt{\sG}, \Ohm_\sV)$. 
\end{proof}

\subsubsection{} \label{sec: preparations} We do some preparations to future reference. Below for any $\IQ$-linear Tannakian category $\shC$, we write $\rN : \Mod_E(\shC) \to \shC$ for either one of the following two functors: $M \mapsto \Nm_{E/\IQ}(M)$ or $M \mapsto \Nm_{E/\IQ}(M) \tensor_\IQ \det(M_{(\IQ)})$ (see notations in \ref{sec: norm functors}). Any conclusion applies to both functors.

Recall the notations from \ref{not: totally real}, and write $\sZ(\sV)$ and $\sH(\sV)$ simply as $\sZ$ and $\sH$. Then $N:= \rN(\sV)$ is a faithful representation of $\sH$, as the composite $\sG \to \Res_{E/\IQ} \GL(\sV) \to \GL(N)$ has kernel exactly $\sZ$. Let $\sK \subseteq \sG(\IA_f)$ and $\sC \subseteq \sH(\IA_f)$ be compact open subgroups such that $\sC$ contains the image of $\sK$. For each prime $\ell$, let $\sK_\ell$ be the image of $\sK$ under the projection $\sG(\IA_f) \to \sG(\IQ_\ell)$.

Let $\Ohm_\sH$ be the $\sH(\IR)$-conjugacy class of morphisms $\IS \to \sH_\IR$ which contains the image of $\Ohm_\sV$. Then $(\sH, \Ohm_\sH)$ is a Shimura datum which admits a natural morphism $(\sG, \Ohm_\sV) \to (\sH, \Ohm_\sH)$. Since $\sZ$ lies in the center of $\sG$ and is discrete, by \cite[Prop.~3.8]{DeligneTdShimura} $(\sH, \Ohm_\sH)$ has the same reflex field as $(\sG, \Ohm_\sV)$. Moreover, one easily checks that $(\sH, \Ohm_\sH)$ satisfies \ref{assumption: Abelian type} using that $(\sG, \Ohm_\sV)$ does.

\begin{remark}
\label{rmk: odd case sH = sG}
    Note that if $\dim_E \sV$ is odd, then $\sZ$ is trivial; \textit{in this case, the reader should read the content below with $(\sG, \Ohm_\sV, \sK) = (\sH, \Ohm_\sH, \sC)$ in mind.} The case when $\dim_E \sV$ is even will only be used for \S\ref{sec: non-maximal monodromy}. 
\end{remark} 

\subsubsection{}\label{sec: apply N to level} Let $k \subseteq \IC$ be a subfield and $T$ be a smooth $k$-variety. Note that the identification $V = \sV_{(\IQ)}$ gives $V$ an $E$-action, which commutes with $\sG$. Suppose that $\sfW \in \sfR(T)$ is a system equipped with a $(\sG, V, \sK)$-level structure $[\mu]$. Then $[\mu]$ transports the $E$-action on $V \tensor \IA_f$ to one on $\sfW_\et$, through which we may view $[\mu]$ as a global section of $\sK \backslash \Isom_E(\sV \tensor_\IQ \IA_f, \sfW_\et)$. Moreover, there is a natural map $$\sK \backslash \Isom_E(\sV \tensor_\IQ \IA_f, \sfW_\et) \to \sC \backslash \Isom(\rN(\sV) \tensor \IA_f, \rN(\sfW_\et)) $$
through which $[\mu]$ defines a $(\sH, N, \sC)$-level structure on $\rN(\sfW_\et)$, which we denote by $\rN([\mu])$. 


\begin{lemma}
\label{lem: pi_1 equivariance}
    Let $T$, $\sfW$ and $[\mu]$ be as above. Let $\sfW' \in \sfR(T)$ be another system with $\sK$-level structure $[\mu']$. Suppose that there is an isomorphism $\gamma_\et : (\sfW_\et, [\mu]) |_{T_\IC} \iso (\sfW_\et', [\mu'])|_{T_\IC}$ such that $\rN(\gamma_\et)$ descends to an isomorphism $(\rN(\sfW_\et), \rN([\mu])) \iso (\rN(\sfW_\et'), \rN([\mu']))$ over $T$.
    
    If $\ell$ is a prime such that $\sK_\ell \cap \sZ(\IQ_\ell) = 1$, then the $\ell$-adic component $\gamma_\ell$ of $\gamma_\et$ descends to an isomorphism $\sfW_\ell \iso \sfW_\ell'$ over $T$. 
\end{lemma}
\begin{proof}
    We may assume that $T$ is connected. Choose a base point $b \in T(\IC)$. Set $\gamma := \gamma_{\ell, b}$ and let $g \in \pi_1^\et(T, b)$ be any element. Our goal is to show that $\delta := g^{-1} \gamma^{-1} g \gamma = 1$. Note that as $\rN(\gamma_\et)$ descends to $T$, we already know that $\rN(\delta) = 1$. Let $\mu : \sV \tensor_\IQ \IA_f \sto \sfW_{\et, b}$ be a representative of $[\mu]_b$. Then $\mu^{-1} g \mu \in \sK_\ell$, because the $\sK$-orbit $[\mu]_b$ is $\pi_1^\et(S, b)$-stable. Set $\mu' := \gamma \mu$. Then $\mu'$ represents $[\mu']_b$, so that $(\mu')^{-1} g \mu' \in \sK_\ell$. Now we have 
    $$ \mu^{-1} \delta \mu = [\mu^{-1} g^{-1} \mu ][(\mu')^{-1} g \mu'] \in \sK_\ell. $$
    Note that $\rN(\mu^{-1} \delta \mu) = 1 \in \GL(N \tensor \IQ_\ell)$. However, as the kernel of the $\sK_\ell$-action on $N \tensor \IQ_\ell$ lies in $\sZ(\IQ_\ell)$, we must have $\mu^{-1} \delta \mu = 1$, i.e., $\delta = 1$. 
\end{proof}

\subsubsection{}\label{sec: N commutes with pullback} We will often consider the following diagram of Shimura data: 
\[\begin{tikzcd}
	{(\sG, \Ohm_\sV)} & {(G, \Ohm)} \\
	{(\sH, \Ohm_\sH)}
	\arrow["i", hook, from=1-1, to=1-2]
	\arrow["\pi"', from=1-1, to=2-1]
\end{tikzcd}\]
Let $\sfK \subseteq G(\IA_f), \sK \subseteq \sG(\IA_f)$ and $\sC \subseteq \sH(\IA_f)$ be neat compact open subgroups such that $\sK \subseteq \sfK$ and $\sC$ contains the image of $\sK$. Then we have Shimura morphisms $i\colon \Sh_\sK(\sG) \to \Sh_\sfK(G)_E$ and $\pi\colon \Sh_\sK(\sG) \to \Sh_\sC(\sH)$ defined over $E$. Let $\sfV$ (resp. $\wt{\sfV}$) be the automorphic system on $\Sh_\sfK(G)$ (resp. $\Sh_\sK(\sG)$) attached to the standard representation $V$ of $G$ (resp. $\sG$) and $[\eta_V]$ (resp. $[\eta_\sV]$) be the tautological $\sfK$-level structure (resp. $\sK$-level structure), as defined in \ref{def: automorphic systems}. Then there is a natural identification $\wt{\sfV} = i^*(\sfV)$ and $[\eta_\sV]$ refines $i^*([\eta_V])$ in the sense of \ref{def: level structures}(d). 

Note that $[\eta_\sV]$ endows $\wt{\sfV}$ with a canonical $E$-action. A priori it only defines an $E$-action on $\wt{\sfV}_\et$, but since $[\eta_\sV]$ is $\wt{\sfV}$-rational, its restriction to $\Sh_\sK(\sG)_\IC$ comes from an $E$-action on $\wt{\sfV}_B$. Then by \ref{lem: descent of morphisms of systems}, one deduces that the resulting $E$-action on $\wt{\sfV}_\dR|_{\Sh_\sK(\sG)_\IC}$ via the Riemann-Hilbert correspondence descends to $\Sh_\sK(\sG)$. Therefore, it makes sense to form $\rN(\wt{\sfV})$. Since $\wt{\sfV}$ is weakly AM in the sense of \ref{def: compatible with AM}, so is $\rN(\wt{\sfV})$. This is simply because we may apply the functor $\rN(-)$ to $\Mot_\Ab(\IC)_{(E)}$ and it commutes with the cohomological realizations. 

Now let $\sfN$ be the automorphic system on $\Sh_\sC(\sH)$ attached to $N \in \mathrm{Rep}(\sH)$ and let $[\eta_N]$ be its tautological $\sC$-level structure. We claim that there is an (necessarily unique) isomorphism 
\begin{equation}
    \label{eqn: N commutes with pullback}
        \pi^*( \sfN, [\eta_N]) \iso (\rN(\ssfV), \rN([\eta_\sV])).
\end{equation}
Indeed, to verify this isomorphism one first checks its implications on the automorphic VHS and \'etale local systems, which follow from the explicit descriptions in \ref{sec: automorphic sheaves}, then applies \ref{lem: descent of morphisms of systems}.

\subsection{Integral model}
\label{sec: integral model}
Let $V$ be a quadratic form over $\IQ$ and suppose that there is a self-dual $\IZ_{(p)}$-lattice $L_{(p)} \subseteq V$ for a prime $p > 2$. Then $G := \SO(V)$ extends to the reductive $\IZ_{(p)}$-group $\SO(L_{(p)})$, which we still write as $G$ by abuse of notation. Let $\sfK \subseteq G(\IA_f)$ be a neat compact open subgroup of the form $\sfK_p \sfK^p$ with $\sfK_p = G(\IZ_p)$ and $\sfK^p \subseteq G(\IA^p_f)$. Then by \cite{KisinInt} and \cite{CSpin}, $\Sh_\sfK(G)$ admits a canonical integral model $\shS_\sfK(G)$ over $\IZ_{(p)}$. 

The Shimura variety $\shS_\sfK(G)$ is typically studied via the corresponding spinor Shimura variety, which is of Hodge type. Let $\wt{G} \coloneqq \CSpin(L_{(p)})$. Set $\IK_p$ to be $\wt{G}(\IZ_p)$, $\IK^p \subseteq \wt{G}(\IA^p_f)$ to be a small enough compact open subgroup whose image in $G(\IA^p_f)$ is contained in $\sfK^p$, and $\IK$ to be the product $\IK_p \IK^p$. The reflex field of $(\wt{G}, \Ohm)$ is $\IQ$ and by \cite{KisinInt} there is an canonical integral model $\shS_{\IK}(\wt{G})$ over $\IZ_{(p)}$. There is a suitable sympletic space $(H, \psi)$ and a Siegel half space $\sH^\pm$ such that there is an embedding of Shimura data $(\wt{G}, \Ohm) \into (\GSp(\psi), \sH^\pm)$ which eventually equips $\shS_\IK(\wt{G})$ with a universal abelian scheme $\shA$.\footnote{Technically, in \cite{KisinInt} and \cite{CSpin}, $\shA$ is only defined as a sheaf of abelian schemes up to prime-to-$p$ quasi-isogeny. However, for $\IK^p$ sufficiently small, we can take $\shA$ to be an actual abelian scheme (cf. \cite[(2.1.5)]{KisinInt}).} Let $a : \shA \to \shS_\IK(\wt{G})$ be the structural morphism. Define the sheaves $\bH_B := R^1 a_{\IC *} \IZ_{(p)}$, $\bH_\ell := R^1 a_* \underline{\IQ}_\ell$ ($\ell \neq p$), $\bH_p := R^1 a_{\IQ *} \underline{\IZ}_p$, $\bH_\dR := R^1 a_* \Ohm^\bullet_{\shA / \shS_\IK(\wt{G})}$ and $\bH_\cris := R^1 \bar{a}_{\cris *} \sO_{\shA_{\IF_p}/ \IZ_p}$ ($\bar{a} := a \tensor \IF_p$). The abelian scheme $\shA$ is equipped with a ``CSpin-structure'': a $\IZ / 2 \IZ$-grading, $\Cl(L)$-action and an idempotent projector $\bpi_? : \End(\bH_?) \to \End(\bH_?)$ for $? = B, \ell, p, \dR, \cris$ on (various applicable fibers of) $\shS_{\IK}(\wt{G})$. We use $\bL_?$ to denote the images of $\bpi_?$, and recall the definition of special endomorphisms (\cite[Def.~5.2, see also Lem.~5.4, Cor.~5.22]{CSpin}): 
\begin{definition}
\label{def: special endomorphisms}
For any $\shS_\IK(\wt{G})$-scheme $T$, $f \in \End(\shA_T)$ is called a \textit{special endomorphism} if for some (and hence all) $\ell \in \sO_T^\times$, the $\ell$-adic realization of $f$ lies in $\bL_\ell|_T \subseteq \End(\bH_\ell|_T)$; if $\sO_T = k$ for a perfect field $k$ in characteristic $p$, then equivalently $f$ is called a special endomorphism if the crystalline realization of $f$ lies in $\bL_{\cris, T}$. We write the submodule of $\End(\shA_T)$ consisting of special endomorphisms as $\LEnd(\shA_T)$. 
\end{definition}

\subsubsection{} \label{sec: the L sheaves} It is explained in \cite[\S5.24]{CSpin} that the sheaves $\bL_?$ on (applicable fibers of) $\shS_\IK(\wt{G})$ in fact descend to the corresponding fibers of $\shS_\sfK(G)$. We denote the descent of these sheaves by the same letters. It is not hard to see that $\bH_B[1/p]$ together with the restrictions of $\prod_{\ell \neq p} \bH_\ell \times \bH_p[1/p]$ and $\bH_\dR$ to $\Sh_\IK(\wt{G})$ is nothing but the automorphic system attached to $H \in \mathrm{Rep}(\wt{G})$, in our terminology \ref{def: automorphic systems}. Similarly, $\bL_B[1/p]$ together with the restrictions of $\prod_{\ell \neq p} \bL_\ell \times \bL_p[1/p]$ and $\bL_\dR$ to $\Sh_\IK(\wt{G})$ (or $\Sh_\sfK(G)$) is precisely the automorphic system attached to $V \in \mathrm{Rep}(\wt{G})$ (or $\mathrm{Rep}(G)$). Readers who wish to check this can look at how the automorphic systems are constructed in \cite[\S4.2]{YangSystem}, which is essentially a generalization of \cite[\S5.24]{CSpin}. The construction of the $\bL_?$-sheaves are also summarized in more detail in \cite[(3.1.3)]{Yang}. In particular, there are natural identifications of $\sfV_B$ with $\bL_B[1/p]$, and $\sfV_\et$ (resp. $\sfV_\dR$) with the restriction of $\prod_{\ell \neq p} \bL_\ell \times \bL_p[1/p]$ (resp. $\bL_\dR$) to $\Sh_\sfK(G)$, where $(\sfV_B, \sfV_\dR, \sfV_\et)$ was the notation we used in \ref{sec: N commutes with pullback}. 

\subsubsection{} Let $\Sh_{\sfK_p}(G)$ be the limit $\varprojlim_{\sfK^p} \Sh_{\sfK_p \sfK'^p}(G)$ as $\sfK'^p$ runs through the compact open subgroups of $G(\IA^p_f)$ and define $\shS_{\sfK_p}(G)$ similarly. The canonical extension property of $\shS_{\sfK_p}(G)$ is that for any regular, formally smooth $\IZ_{(p)}$-scheme $S$, every morphism $S_\IQ \to \shS_{\sfK_p}(G)$ extends to $S$. We give an extension property for finite level, which is certainly well known to experts: 

\begin{theorem}
\label{thm: extension property}
    Let $S$ be a smooth $\IZ_{(p)}$-scheme which admits a morphism $\rho : S_\IQ \to \Sh_\sfK(G)$. If $\rho^* \sfV_\ell$ extends to a local system $\sfW_\ell$ over $S$ for every prime $\ell \neq p$, then $\rho$ extends to a morphism $S \to \shS_\sfK(G)$, through which $\sfW_\ell$ is identified with the pullback of $\bL_\ell$. 
\end{theorem}
\begin{proof}
    To simplify notation, let $\sfV_\et^{(p)}\coloneqq \prod_{\ell \neq p} \sfV_\ell$ and $\sfW_\et^{(p)} := \prod_{\ell \neq p} \sfW_\ell$. Note that $\sfW_\et^{(p)}|_{S_\IQ} = \rho^*(\sfV_\et^{(p)})$. Now the prime-to-$p$ part of the tautological level structure $[\eta_V]$ gives us a section $[\eta_V^{(p)}] : \Sh_\sfK(G) \to \sfK^p \backslash \Isom(V \tensor \IA^p_f, \sfV_\et^{(p)})$, where the target is viewed as a pro-\'etale cover of $\Sh_\sfK(G)$. By \cite[0BQM]{stacks-project}, $\rho^\ast[\eta^{(p)}_V]$ extends to a morphism $S \to \sfK^p \backslash \Isom(V \otimes \IA_f^p, \sfW_\et^{(p)})$, where the target is a pro-\'etale cover of $S$. We define $\wt{S}$ as the pull-back
    \[\begin{tikzcd}
	{\wt{S}} & {\Isom(V \otimes \IA_f, \sfW_\et^{(p)})} \\
	S & {\sfK \backslash \Isom(V \otimes \IA_f, \sfW_\et^{(p)})}
	\arrow[from=1-1, to=1-2]
	\arrow[from=1-1, to=2-1]
	\arrow[from=2-1, to=2-2]
	\arrow[from=1-2, to=2-2]
	\arrow["\lrcorner"{anchor=center, pos=0.125}, draw=none, from=1-1, to=2-2]
    \end{tikzcd}\]
    Note that since $\wt{S} \to S$ is a $\sfK^p$-torsor for the pro\'etale topology, $\wt{S}$ is representable by a scheme. As $\Sh_{\sfK_p}(G) \to \Sh_{\sfK}(G)$ can be defined by an analogous construction, $\rho$ lifts to a $\sfK^p$-equivariant morphism $\wt\rho\colon \wt{S}_\IQ \to \Sh_{\sfK_p}(G)$. The canonical extension property allows us to extend it to a (necessarily $\sfK^p$-equivariant) morphism $\wt{S} \to \shS_{\sfK_p}(G)$. Now the $\sfK^p$-action defines an \'etale descend datum which yields the desired morphism $S \to \shS_\sfK(G)$.
\end{proof}

\section{Period morphisms}
\subsection{The basic set-up}
We first state a few basic definitions and results which will be needed to construct the period morphism. 

\begin{definition}
\label{def: \'etale locally isomorphic}
    Let $T$ be a connected noetherian normal scheme with geometric point $t$, $\ell$ be a prime with $\ell \in \sO_T^\times$ and $\sfW_\ell$ be an \'etale $\IQ_\ell$-local system over $T$. We denote by $\Mon(\sfW_\ell, t)$ the Zariski closure of the image of $\pi_1^\et(T, t)$ in $\GL(\sfW_{\ell, t})$, and denote by $\Mon^\circ(\sfW_\ell, t)$ the identity component of $\Mon(\sfW_\ell, t)$.

    If $\sfW_\ell'$ is another $\IQ_\ell$-local system, then we say that $\sfW_\ell$ and $\sfW_\ell'$ are \textit{\'etale locally isomorphic} if they are isomorphic over some finite connected \'etale cover $T'$ of $T$; or equivalently, there is an isomorphism $\sfW_{\ell, t} \sto \sfW'_{\ell, t}$ which is equivariant under an open subgroup of $\pi^\et_1(T, t)$. 
\end{definition}

\begin{lemma}
\label{lem: extend LB on generic}
Let $T$ be a noetherian integral normal scheme with generic point $\eta$. Let $f : \sY \to T$ be a smooth proper morphism.
\begin{enumerate}[label=\upshape{(\alph*)}]
    \item The natural map $\Pic(\sY) \to \Pic(\sY_\eta)$ is surjective with kernel $\mathrm{im}(\Pic(T) \to \Pic(\sY))$.
    \item If for some geometric point $b$ over $\eta$ and prime $\ell \in \sO_T^\times$, $\Mon(R^2 f_{*} \IQ_\ell, b)$ is connected, then the natural map $\NS(\sY_\eta)_\IQ \to \NS(\sY_b)_\IQ$ is an isomorphism. 
\end{enumerate} 
\end{lemma}
\begin{proof}
(a) We may always extend a line bundle on $\sY_\eta$ to $\sY_U$ for some open dense subscheme $U \subseteq T$, and then to a line bundle on $\sY$ (use e.g., \cite[Prop.~II.6.5]{Hartshorne}). But any two extensions to $\sY$ differ by an element of $\Pic(T)$ by \cite[Err\textsubscript{IV} Cor.~21.4.13]{EGAIV4}. 

(b) Let $\bar{\eta}$ be the geometric point over $\eta$ obtained by taking the separable closure of $k(\eta)$ in $k(b)$. Then every class of $\NS(\sY_b)_\IQ$ descends to $\bar{\eta}$. As the natural morphism $\Gal(\bar{\eta}/\eta) = \pi^\et_1(\eta, b) \to \pi^\et_1(T, b)$ is surjective, and $\Gal(\bar{\eta}/\eta)$ acts on $\NS(\sY_{\bar{\eta}})_\IQ$ through a finite quotient, the connectedness assumption on $\Mon(R^2 f_{*} \IQ_\ell, b)$ implies that $\Gal(\bar{\eta}/\eta)$ in fact acts trivially. This implies that every class in $\NS(\sY_{\bar{\eta}})_\IQ$ descends to $\eta$. 
\end{proof}




Now we state the set-up we will work with for the entire section. 

\begin{set-up}
\label{set-up: char 0 base}
Let $F \subseteq \IC$ be a subfield finitely generated over $\IQ$ and let $S$ be a connected smooth $F$-variety with generic point $\eta$. Let $f\colon \sX \to S$ be a $\heartsuit$-family with a relatively ample line bundle $\bxi$.  We fix a subspace $\Lambda \subseteq \NS(\sX_\eta)_\IQ$  containing the class of $\bxi_\eta$. Now choose an $F$-linear embedding $k(\eta) \into \IC$ and let $b \in S(\IC)$ be the resulting point, which we also view as a closed point on $S_\IC$. Let $S^\circ \subset S_\IC$ denote the connected component containing $b$. \textit{We assume that for some (and hence every, see \ref{lem: Larsen-Pink} below) prime $\ell$, $\Mon(R^2 f_* \IQ_\ell, b)$ is connected.} 

Define a pairing on $\H^2_\dR(\sX/S) := R^2 f_* \Ohm^\bullet_{\sX/S}$ by $(x, y) \mapsto x \cup y \cup c_1(\bxi)^{d - 2}$ for local sections $x, y$ and $d = \dim \sX/S$. By \ref{lem: extend LB on generic}, every line bundle on $\sX_\eta$ extends to a relative line bundle on $\sX/S$. By taking Chern classes, we obtain a well defined embedding $\underline{\Lambda} \into \H^2_\dR(\sX/S)$, where $\underline{\Lambda}$ is the constant sheaf with fiber $\Lambda$. Define $\sfP_\dR$ to be the orthogonal complement of $\underline{\Lambda}$ in $\H^2_\dR(\sX/S)(1)$. We define the primitive Betti cohomology $\sfP_B$ over $S_\IC$ and \'etale cohomology $\sfP_\et$ over $S$ analogously. Then $\sfP := (\sfP_B, \sfP_\dR, \sfP_\et) \in \sfR(S)$ is a system of realizations over $S$ in the sense of \ref{def: system of realizations}. 
Note that we applied a Tate twist so that the VHS $\sfP|_{S_\IC}$ has weight $0$. Since we assumed $\Mon(R^2 f_* \IQ_\ell, b)$ is connected for some $\ell$, $\NS(\sX_\eta)_\IQ = \NS(\sX_b)_\IQ$ by \ref{lem: extend LB on generic}. This implies that $\sfP$ orthogonally decomposes into $(\Lambda^\perp \tensor \mathbf{1}_S) \oplus \sfP_0$ where $\Lambda^\perp$ is the orthogonal complement of $\Lambda$ in $\NS(\sX_\eta)_\IQ$ and $\sfP_0 = (\sfP_{0, B}, \sfP_{0, \dR}, \sfP_{0, \et})$ is another object in $\sfR(S)$. Moreover, there are no nonzero $(0,0)$-classes in $\sfP_{0, B, b}$. Note that $\sfP_0$ is nothing but $\sfP$ when $\Lambda = \NS(\sX_\eta)_\IQ$. 
\end{set-up}

\begin{lemma}
\label{lem: Larsen-Pink}
    In the above set-up, $\Mon(R^2 f_* \IQ_\ell, b)$ is connected for every prime $\ell$, and $b$ is a Hodge-generic point for the VHS $\sfP|_{S^\circ}$. 
\end{lemma}
\begin{proof}
    The first statement follows from the assumption that $F$ is finitely generated over $\IQ$ and \cite[Prop.~6.14]{Larsen-Pink}, which implies that the \'etale group scheme of connected components of $\Mon(R^2 f_* \IQ_\ell, b)$ is independent of $\ell$. The second statement follows from the main theorem of \cite{Moonen}.\footnote{The main theorem of \cite{Moonen} is an overkill for the purpose here and only used to avoid a case by case discussion of Mumford-Tate groups.} Since the statement only concerns $S^\circ$, we may replace $S$ by $S^\circ$ and $F$ by its field of definition and thus assume that $S$ is geometrically connected. Then we may make use of the notion of Galois-generic points (\cite[Def.~4.2.1]{Moonen-Fom}). As the Mumford-Tate conjecture is known for $H^2$ of the fibers of $\sX/S$ and $b$ lies above $\eta$, the fact that $\eta$ is Galois-generic implies that $b$ is Hodge-generic.
\end{proof}

\subsubsection{} \label{sec: define G for period} Let $V$ be a quadratic form over $\IQ$ which is isomorphic to $\sfP_{B, b}$ and fix an isometry $\mu_b : V \sto \sfP_{B, b}$. Let $G := \SO(V)$. Let $V_0 := \mu_b^{-1}(\sfP_{0, B, b})$.
Note that via $\mu_b$, the monodromy representation of $\pi^\et_1(S, b)$ takes values in $G(\IA_f)$. We say that $\sfK$ is \textit{admissible} if the image of $\pi_1^\et(S, b)$ in $\GL(\sfP_{B, b} \tensor \IA_f)$ lies in $\sfK$ via $\mu_b$, or equivalently, $\mu_b$ extends to a $(G, V, \sfK)$-level structure $[\mu]$ on $\sfP_\et$ with $[\mu]_b = \sfK \cdot (\mu_b \tensor \IA_f)$. Define $\Ohm := \{ v \in \IP(V \tensor \IC) \mid \< v, v \> = 0, \< v, \bar{v} \> > 0 \}$ as in \ref{sec: preparations}. Then $(G, \Ohm)$ is a Shimura datum of abelian type with reflex field $\IQ$. Again let $\sfV = (\sfV_B, \sfV_\dR, \sfV_\et)$ be the automorphic system on $\Sh_\sfK(G)$ attached to $V \in \mathrm{Rep}(G)$, and let $[\eta_V]$ be the tautological $\sfK$-level structure on $\sfV_\et$ (see \ref{def: automorphic systems}).

\subsubsection{} \label{def: condition sharp} Suppose that $(\sX/S, \bxi)|_{S^\circ}$ belongs to case (R+) or (R2') described in \ref{sec: the 4 cases}, i.e., the endomorphism field $E$ of the Hodge structure on $\sfP_{0, B, b}$ is totally real. By \ref{thm: Zarhin}, the $E$-action on $\sfP_{0, B, b}$ is self-adjoint. Let $E$ act on $V_0$ through $\mu_b$. Recall the notations in \ref{not: totally real}. Let $\sV$ be the $E$-bilinear lift of $V_0$, i.e., $V_0 = \sV_{(\IQ)}$ and set $\sG := \Res_{E/\IQ} \SO(\sV)$, $\sH := \sG/ Z_E^1$. In addition, set $\sV^\sharp := \sV \oplus \sE$ and $\sG^\sharp := \Res_{E/\IQ} \SO(\sV^\sharp)$. Embed $\SO(\sV)$ into $\SO(\sV^\sharp)$ by acting trivially on $\sE$. This induces an embedding $\sG \into \sG^\sharp$.

We say that a neat compact open subgroup $\sfK \subset G(\IA_f)$ satisfies condition $(\sharp)$ depending on the particular case ($\sK := \sfK \cap \sG(\IA_f)$ below):
\begin{itemize}
    \item[(R1)] Always.
    \item[(R2)] If the image of $\sK$ in $\sG^\sharp(\IA_f)$ lies in some neat compact open subgroup.
    \item[(R2')] If the image of $\sK$ in $\sH(\IA_f)$ lies in some neat compact open subgroup.
\end{itemize}
In case (R2'), we say that for a prime $\ell_0$, $\sfK_{\ell_0}$ is sufficiently small if $\sK_{\ell_0} \cap Z_E^1(\IQ_{\ell_0}) = 1$. Here $\sfK_{\ell_0}$ is the image of $\sfK$ under the projection $G(\IA_f) \to G(\IQ_{\ell_0})$ and $\sK_{\ell_0}$ is defined similarly.

\subsubsection{} \label{eqn: fp taut diagram} For any subfield $F' \subseteq \IC$ which contains $F$ and $s \in S(F')$, we denote by $\fp_s \in \Mot_\AH(F')$ the submotive of $\fh^2(\sX_s)(1)$ such that $\w_?(\fp_s) = \sfP_{?, s}$ for $? = B$ (when $F' = \IC$), or $\dR, \et$.  Note that for any $\sigma \in \Aut(\IC/F)$ and $s \in S(\IC)$, there is a natural isomorphism $(\fp_{s})^\sigma \iso \fp_{\sigma(s)}$ in $\Mot_\AH(\IC)$, which we shall label as $\sigma_{\fp, s}$; moreover, the following diagrams tautologically commute (recall the notations in \ref{sec: twist tensors} and \ref{not: galois descent}): 
\begin{equation*}
\hspace*{-0.7cm}
    \begin{tikzcd}
	{\w_\et(\fp_s)} & {\sfP_{\et, s}} & {\sfP_{\dR, s}} & {\sfP_{\dR, s}} \\
	{\w_\et((\fp_s)^\sigma)} & {\w_\et(\fp_{\sigma(s)}) = \sfP_{\et, \sigma(s)}} & {\w_\dR((\fp_s)^\sigma)} & {\w_\dR(\fp_{\sigma(s)}) = \sfP_{\dR, \sigma(s)}}
	\arrow[from=1-1, to=1-2, equal]
	\arrow["{\iso_\bc}"', from=1-1, to=2-1]
	\arrow["{\sigma_{\fp, s}}"', from=2-1, to=2-2]
	\arrow["{\sigma_{\sfP_\et, s}}", from=1-2, to=2-2]
	\arrow["{\sigma_{\fp, s}}"', from=2-3, to=2-4]
	\arrow["{\iso_\bc}"', from=1-3, to=2-3]
	\arrow[from=1-3, to=1-4, equal]
	\arrow["{\sigma_{\sfP_\dR, s}}", from=1-4, to=2-4]
\end{tikzcd}
\end{equation*}

\subsubsection{} 
\label{sec: define alpha circ} \textit{For the rest of section 4, we will put a standing assumption that $S$ is \textbf{geometrically connected} as an $F$-variety, so that $S_\IC = S^\circ$.}


Let $\sfK \subseteq G(\IA_f)$ be a neat compact open subgroup and assume that it is also admissible, so that $\mu_b$ extends to a $\sfK$-level structure $[\mu]$ on $\sfP$ with $[\mu]_b = \sfK (\mu_b \tensor \IA_f)$. Now \ref{lem: situation in practice} guarantees that the restriction of $[\mu]$ to $S_\IC$ is $(\sfP|_{S_\IC})$-rational and is of type $\Ohm$. Therefore, \ref{thm: moduli over C} gives us a unique morphism $\rho_\IC : S_\IC \to \Sh_\sfK(G)_\IC$ such that 
\begin{equation}
    \label{eqn: construct rho circ}
    (\rho_\IC)^* (\sfV, [\eta_V]) \iso (\sfP, [\mu])|_{S_\IC}. 
\end{equation}

Note that as $\sfK$ is neat, the above isomorphism is unique; we denote its \'etale component by $\alpha^\circ_\et : \rho_\IC^*(\sfV_\et) \sto \sfP_\et|_{S_\IC}$ and let $\alpha^\circ_\ell$ be its $\ell$-adic component for a prime $\ell$. 

\begin{theorem}
\label{thm: period char 0}
In the notations above, assume either
\begin{enumerate}[label=\upshape{(\alph*)}]
    \item $\fp_s \in \Mot_\Ab(\IC)$ for every $s \in S_\IC$ (e.g., when the family $(\sX/S, \bxi)|_{S_\IC}$ belongs to case (CM)), or
    \item the family $(\sX/S, \bxi)|_{S_\IC}$ belongs to case (R+) = (R1) + (R2) and $\sfK$ satisfies condition $(\sharp)$ as defined in \ref{def: condition sharp}. 
\end{enumerate}
Then $\rho_\IC$ descends to a morphism $\rho : S \to \Sh_\sfK(G)_F$ over $F$. Moreover, $\alpha^\circ_\et$ descends to an isomorphism $\alpha_\et : \rho^* (\sfV_\et) \sto \sfP_\et$ over $S$. 
\end{theorem}

Case (a) is easy: The diagrams in \ref{eqn: fp taut diagram} readily imply that $\sfP$ over $S$ is weakly AM, so that the conclusion follows from \ref{thm: shimura as moduli}. In fact, we have that the de Rham component of the isomorphism (\ref{eqn: construct rho circ}) also descends to $S$. Note that case (a) in particular covers the case when $(\sX/S, \bxi)|_{S_\IC}$ belongs to case (CM) by \ref{thm: Moonen}(d). The proof of case (b) is the content of \S\ref{sec: proof of (R+)} below. If $(\sX/S, \bxi)|_{S_\IC}$ belongs to case (R2') and it is not known that $\fp_s \in \Mot_\Ab(\IC)$ for every $s \in S_\IC$, a slightly weaker version of the above theorem holds:
\begin{theorem}
\label{thm: period in R2' case}
    Assume that the family $(\sX/S, \bxi)|_{S^\circ}$ belongs to case (R2'), $\Lambda = \Lambda_0$, $\sfK$ satisfies condition $(\sharp)$, and for a prime $\ell_0$, $\sfK_{\ell_0}$ is sufficiently small as defined in \ref{def: condition sharp}. 

    Then $\rho_\IC : S_\IC \to \Sh_\sfK(G)_\IC$ descends to a morphism $\rho : S_{F'} \to \Sh_\sfK(G)_{F'}$ for a finite extension $F' / F$ in $\IC$. Moreover, $\alpha^\circ_{\ell_0}$ descends to an isomorphism $\alpha_{\ell_0} : \rho^* \sfV_{\ell_0} \sto \sfP_{\ell_0}|_{S_{F'}}$, and $\rho^* \sfV_{\ell}$ is \'etale locally isomorphic to $\sfP_{\ell}|_{S_{F'}}$ for every other $\ell$ in the sense of \ref{def: \'etale locally isomorphic}. 
\end{theorem}

\begin{remark}
In literature, to define period morphisms to orthogonal Shimura varieties, one usually keeps track of a trivialization of the determinants (cf. \cite[Prop.~4.3]{MPTate}). Such a trivialization (i.e., an isometry $\underline{\det(V)} \sto \det(\sfP_B)$ in our notations) is implicit in the statement that (the restriction to $S_\IC$ of) $[\mu]$ is $(\sfP|_{S_\IC})$-rational as a $(G, V, \sfK)$-level structure, because $\det(V)$ is $G$-invariant. More concretely, one obtains this trivialization by globalizing $\det(\mu_b)$, using that $\pi_1(S_\IC, b)$ fixes $\det(\sfP_{B, b})$ (cf. the proof of \ref{lem: situation in practice}). However, we remark that if $S$ were not geometrically connected, we would not be able to show that $[\mu]$ is $(\sfP|_{S_\IC})$-rational over the connected components of $S_\IC$ other than $S^\circ$, unless we know $\det(\sfP_{B, b})$ is spanned by an absolute Hodge class (e.g., in case (a) of \ref{thm: period char 0}). 

On the other hand, for the purpose of putting level structures, the lack of absolute-Hodgeness of $\det(\sfP_{B, b})$ is partially remedied by independence-of-$\ell$ type results on algebraic monodromy (e.g., \cite[Lem.~3.2]{Saitodisc}, cf. \cite[Cor.~5.9]{Taelman2}) which implies that if for some $\ell$, $\det(\sfP_{B, b}) \tensor \IQ_\ell$ is $\pi_1^\et(S, b)$-invariant, then the same is true for all $\ell$. In \ref{lem: Larsen-Pink} we used Larsen-Pink' result \cite[Prop.~6.14]{Larsen-Pink} to achieve a similar effect. Later in \ref{sec: preparation for period over E 2} this is used to overcome a similar difficulty: We do not know that the tensors which cut out $\Res_{E/\IQ}\SO(\sV)$ from $\Res_{E/\IQ} \O(\sV)$ are given by absolute-Hodge tensors on $\sfP_{B, b}$ via $\mu_b$. As far as we are aware of, this cannot be deduced from \ref{thm: Moonen}. However, we can put $\sK$-level structures on $\sfP_\et$ in question and proceed. 
\end{remark}

\subsection{Case (R+): maximal monodromy}
\label{sec: proof of (R+)}


\begin{lemma}
\label{sec: throw away (0,0) classes} 
    It suffices to prove \ref{thm: period char 0} when $\Lambda = \NS(\sX_\eta)_\IQ$. 
\end{lemma}
\begin{proof}
Let $\Lambda^\perp$ be the orthogonal complement of $\Lambda$ in $\Lambda_0 := \NS(\sX_\eta)_\IQ$ and set $M = \mu_b^{-1}(\Lambda^\perp)$. Then $V = V_0 \oplus M$ and we view $G_0$ as the stabilizer of $M$ in $G$. The image of $\pi_1^\et(S, b)$ in $G(\IA_f)$ via $\mu_b$ actually lies in $G_0(\IA_f)$. Therefore,  $\sfK_0 = \sfK \cap G_0(\IA_f)$ satisfies condition $(\sharp)$ for $\Lambda_0$.

Let $\sfV_0$ be the automorphic system on $\Sh_0 := \Sh_{\sfK_0}(G_0)$ given by $V_0 \in \mathrm{Rep}(G_0)$, $[\eta_{V_0}]$ be the $\sfK_0$-level structure on $\sfV_{0}$, and $[\mu_0]$ be the $\sfK_0$-level structure on $\sfP_0$ defined by $\mu_{0, b} := \mu_b |_{V_0}$. As in the paragraph above \ref{thm: period char 0}, we obtain a morphism $\rho_{0, \IC} : S_\IC \to \Sh_{0, \IC}$ such that $(\sfV_0, [\eta_{V_0}])|_{S_\IC} \iso (\sfP_0, [\mu_0])|_{S_\IC}$. Define $\alpha^\circ_{0, \et}$ accordingly. 

Consider the Shimura morphism $j : \Sh_0 \to \Sh := \Sh_{\sfK}(G)$. Then we have natural identifications 
\begin{equation}
\label{eqn: j^*}
    j^*(\sfV) = \sfV_0 \oplus (M \tensor \mathbf{1}_{\Sh_0}), \text{ and } \sfP = \sfP_0 \oplus (\Lambda^\perp \tensor \mathbf{1}_{S})
\end{equation}
Moreover, the level structure $[\eta_{V_0}]$ (resp. $[\mu_0]$) refines $j^*([\eta_V])$ (resp. $[\mu]$) in the sense of \ref{def: level structures}(d). Therefore, one easily checks that $$(j_\IC \circ \rho_{0, \IC})^* (\sfV, [\eta_V]) \iso (\sfP, [\mu])|_{S_\IC}.$$
By the uniqueness statement in \ref{thm: moduli over C}, this implies that $\rho_\IC = j_\IC \circ \rho_{0, \IC}$.

Assume now that $\rho_{0, \IC}$ descends to $\rho_0$ over $F$, and $\alpha^\circ_{0, \et}$ descends to $\alpha_{0, \et}$ over $S$. Then $\rho_\IC$ descends to $j_F \circ \rho_0$, and $\alpha^\circ_\et$ descends to an isomorphism $\alpha_{0,\et} \oplus (\mathrm{id}_{\Lambda^\perp} \tensor \underline{\IA}_f)$ over $S$. 
\end{proof}

\subsubsection{} \label{sec: preparation for period over E}
In this section we prove \ref{thm: period char 0}. By \ref{sec: throw away (0,0) classes}, we may assume that $\Lambda = \Lambda_0$, so that $V = V_0$ and $\sfP = \sfP_0$. Let $E, \sG, \sV, \sG^\sharp, \sfK, \sK$ be as introduced in \ref{def: condition sharp}. In particular, $E$ is the endomorphism field of the Hodge structure $\sfP_b$, and $\sV$ is the $E$-bilinear lift of $V$, which carries an $E$-action via the isometry $\mu_b : V \sto \sfP_{B, b}$ fixed in \ref{sec: define G for period}.  

Note that by \ref{thm: Moonen} each $e \in E$, viewed as an element of $\End(\sfP_{B, b}) = \End(\w_B(\fp_b))$, is absolute Hodge, so its image in $\End(\sfP_{\et, b})$ is stabilized by an open subgroup of $\pi_1^\et(S, b)$. Since we assumed that $\Mon(\sfP_\ell, b)$ is connected for every $\ell$, the $E$-action on $\sfP_{\et,b}$ must already be $\pi_1^\et(S, b)$-equivariant. This has the following consequence: 

\begin{lemma}
    The action of $E$ on $\sfP_{B, b}$ extends to an action on $\sfP$. If $\tau : E \into \IC$ is the embedding induced by the action of $E$ on $\Fil^1 \sfP_{\dR, b}$, or equivalently the unique indefinite real place of $\sV$, then $\tau(E) \subseteq F$. 
\end{lemma}
\begin{proof}
    For the first statement, it is clear that the $E$-action on $\sfP_{B, b}$ (resp. $\sfP_{\et, b}$) extends (necessarily uniquely) to $\sfP_B$ and (resp. $\sfP_\et$) because it is $\pi_1(S_\IC, b)$ (resp. $\pi_1^\et(S, b)$)-equivariant. It remains to show that the $E$-action on $\sfP_{\dR}|_{S_\IC}$, obtained via the Riemann-Hilbert correspondence, descends to an action on $\sfP_\dR$. This follows from the fact that the de Rham realization of every $e \in E$, as an element of $\sfP_{\dR, b}^\tensor$, descends to $\sfP_{\dR, \eta}^\tensor$: Since $e$ is absolute Hodge and its \'etale component is $\pi^\et_1(\eta, b)$-invariant, its de Rham component descends to $\eta$ (cf. the argument for \cite[(2.2.2)]{KisinInt}). 

    We remind the reader that by $S(\IC)$ we mean the set of \textbf{$F$-linear} morphisms $\Spec(\IC) \to S$. The first statement implies that for every $s \in S(\IC)$, $\sfP_{B, s}$ carries an action of $E$, which is self-adjoint by \ref{thm: Zarhin}; moreover, for every $\sigma \in \Aut(\IC/F)$, the $\sigma$-linear isomorphism $\sigma_{\sfP_\dR, s} : \sfP_{\dR, s} \iso \sfP_{\dR, \sigma(s)}$ of filtered vector spaces is $E$-equivariant (see \ref{not: galois descent} for this notation). Therefore, if we let $\tau_s : E \to \IR$ be the place through which $E$ acts on $\Fil^1 \sfP_{\dR, s}$, then $\tau_{\sigma(s)} = \sigma \circ \tau_s$. Now we use that $\tau_s$ can also be characterized as the unique real place of $E$ such that $\sfP_{B, s} \tensor_{\tau_s} \IR$ is indefinite. Parallel transport implies that $\tau_s$ is constant on $S(\IC)$, which by assumption is connected. As $\tau = \tau_b$ and $\sigma(b) \in S(\IC)$, $\sigma \circ \tau = \tau$ for every $\sigma \in \Aut(\IC/F)$. This implies that $\tau(E) \subseteq F$. 
\end{proof}

\subsubsection{} \label{sec: preparation for period over E 2} Below we shall view $E$ as a subfield of $F$ (and hence of $\IC$) via $\tau$ as above and drop $\tau$ from the notation. Now we recall the discussion in \ref{sec: N commutes with pullback}: Let $\Ohm_\sV \subseteq \Ohm$ be the Hermitian symmetric subdomain $\{ w \in \IP(\sV \tensor_E \IC) \mid \< w, \bar{w} \> > 0, \< w, w \> = 0\}$. Then $(\sG, \Ohm_\sV)$ is a Shimura subdatum of $(G, \Ohm)$ with reflex field $E$. Therefore, there is a Shimura morphism $i_\IC : \Sh_\sK(\sG)_\IC \to \Sh_\sfK(G)_\IC$ which descends to $i : \Sh_\sK(\sG) \to \Sh_\sfK(G)_E$ over $E$. Let $\wt{\sfV}$ be the automorphic system on $\Sh_\sK(\sG)$ defined by $V$ and let $[\eta_\sV]$ be its tautological $\sK$-structure. Recall that $\wt{\sfV}$ is identified with $i^* \sfV$, and $[\eta_\sV]$ refines $i^*([\eta_V])$.  

Consider the situation in \ref{sec: define alpha circ}. Recall that we assumed that $\Mon(\sfP_\ell, b)$ is connected for every $\ell$. Since the centralizer of the $E$-action in $\O(V)$ can be identified with $\Res_{E/\IQ} \O(\sV)$, which contains $\sG$ as the identity component, the monodromy action of $\pi_1^\et(S, b)$ must take values in $\sK := \sfK \cap \sG(\IA_f)$ via $\mu_b$. Therefore, there exists a $(\sG, V, \sK)$-level structure $[\wt{\mu}]$ on $\sfP_\et$ such that $[\wt{\mu}]_b = \sK \cdot (\mu_b \tensor \IA_f)$. As the Hodge structure on $\sfP_{B, b}$ is defined by a point on $\Ohm_\sV$ via $\mu_b$, by \ref{lem: situation in practice} the restriction of $[\wt{\mu}]$ to $S_\IC$ is $(\sfP|_{S_\IC})$-rational and of type $\Ohm_\sV$. One easily checks that: 

\begin{lemma}
\label{lem: define varrho over C}
    Let $\varrho_\IC : S_\IC \to \Sh_\sK(\sG)_\IC$ be the unique morphism such that 
    \begin{equation}
    \label{eqn: define wt rho}
        (\varrho_\IC)^* (\ssfV, [\eta_\sV]) \iso (\sfP, [\wt{\mu}])|_{S_\IC}
    \end{equation}
    given by \ref{thm: moduli over C}. Then $\rho_\IC = i_\IC \circ \varrho_\IC$.
\end{lemma}

Hence we reduce \ref{thm: period char 0} to: 
\begin{theorem}
    \label{thm: period char 0 F} 
    Under the hypothesis of \ref{thm: period char 0} and notations above, the morphism $\varrho_\IC$ descends to an $F$-morphism $\varrho : S \to \Sh_\sK(\sG)_F$; moreover, the \'etale component $\alpha^\circ_\et : \varrho_\IC^* \wt{\sfV}_\et \sto \sfP_\et|_{S_\IC}$ of (\ref{eqn: define wt rho}) descends to an isomorphism of $\alpha_\et : \varrho^*\ssfV_\et \sto \sfP_\et$ over $S$. 
\end{theorem}
Note that $\alpha_\et^\circ$ above agrees with the one in \ref{thm: period char 0} because $\ssfV_\et = i^* \sfV_\et$. \\\\
Below for any $\IQ$-linear Tannakian category $\shC$, we write $\rN(-)$ for the functor $\Mod_E(\shC) \to \shC$ which sends every $M \in \Mod_E(\shC)$ to $\Nm_{E/\IQ}(M) \tensor_\IQ \det(M_{(\IQ)})$ (cf. \ref{sec: norm functors}). Recall that in \ref{sec: apply N to level} we explained how to apply $\rN(-)$ to a $\sK$-level structure on a system of realizations.


\subsubsection{} \label{sec: proof in odd case} We first treat the case when $m := \dim_E \sV$ is odd. In this case, $N := \rN(\sV) \in \mathsf{Rep}(\sG)$ is \textit{faithful}. Let $\sfN$ be the automorphic system on $\Sh_\sfK(\sG)$ associated to $N$. Let $[\eta_N]$ be the tautological $\sK$-level structure on $\sfN$. Then by \ref{sec: N commutes with pullback}, there is a unique isomorphism 
\begin{equation}
    \label{eqn: N of V is N}
    (\rN(\ssfV), \rN(\eta_\sV)) \iso (\sfN, [\eta_N]). 
\end{equation}
As remarked in \ref{rmk: odd case sH = sG}, one should read \ref{sec: N commutes with pullback} with $(\sH, \sC, \Ohm_\sH) = (\sG, \sK, \Ohm_\sV)$. 
\begin{lemma}
    \label{lem: norm of P is weakly AM}
    $\rN(\sfP) \in \sfR(S)$ is weakly AM in the sense of \ref{def: compatible with AM}. 
\end{lemma}
\begin{proof}
    As the formation of $\rN(-)$ on $\Mot_\AH(\IC)_{(E)}$ commutes with cohomological realizations, for every $s \in S(\IC)$, $\rN((\sfP_{B, s}, \sfP_{\dR, s})) = \w_\Hdg(\rN(\fp_s))$. By \ref{thm: Moonen}(a), $\rN(\fp_s) \in \Mot_\Ab(\IC)$. Therefore, in \ref{def: compatible with AM} we may take $M = \rN(\fp_s)$, so that $M^\sigma = (\rN(\fp_s))^\sigma = \rN((\fp_s)^\sigma)$. By appling $\rN(-)$ to the objects in the diagram \ref{eqn: fp taut diagram}, one checks that the diagrams in \ref{def: compatible with AM} commute for $\fN(\sfP)$. 
\end{proof} 

\begin{lemma}
    There exists a unique morphism $\rho_N : S \to \Sh_\sK(\sG)_F$ such that 
    \begin{equation}
    \label{eqn: iso for rho N}
    \rho_N^* (\sfN, [\eta_N]) \iso (\rN(\sfP), \rN([\wt{\mu}])).
    \end{equation}
    Moreover, $\varrho_\IC = \rho_N |_{S_\IC}$. 
\end{lemma}
\begin{proof}
    One easily checks from \ref{sec: apply N to level} that $\rN([\wt{\mu}]) = \sK \cdot (\rN(\mu_b) \tensor \IA_f)$, so by \ref{lem: situation in practice} $\rN([\wt{\mu}])$ is $\rN(\sfP)$-rational. Moreover, as $[\wt{\mu}]$ is of type $\Ohm_\sV$, so is $\rN([\wt{\mu}])$. As $\rN(\sfP)$ is weakly AM, \ref{thm: shimura as moduli} gives the $\rho_N$ for which (\ref{eqn: iso for rho N}) holds. On the other hand, by appling $\rN(-)$ to (\ref{eqn: define wt rho}), we see that 
    \begin{equation}
        \label{eqn: rho pullback N}
        (\varrho_\IC)^* (\sfN, [\eta_\rN]) \iso (\rN(\sfP), \rN([\wt{\mu}]))|_{S_\IC}.
    \end{equation}
By the uniquness statement in \ref{thm: moduli over C}, this implies that $\varrho_\IC = \rho_N |_{S_\IC}$.
\end{proof}

\noindent \textit{Proof of \ref{thm: period char 0 F} for $m$ odd}: To affirm the first statement, set $\varrho := \rho_N$. The second statement now follows from \ref{lem: pi_1 equivariance}. Indeed, comparing (\ref{eqn: N of V is N}), (\ref{eqn: iso for rho N}) and (\ref{eqn: rho pullback N}), we see that $\rN(\alpha^\circ_\et)$ descends to $S$. However, as $\dim_E \sV$ is odd, $\sZ(\sV) = 1$ (see \ref{not: totally real}). By \ref{lem: pi_1 equivariance}, $\alpha^\circ_\et$ descends to $S$. \qed

\subsubsection{} Recall that in \ref{def: condition sharp} $\sV^\sharp = \sV \oplus \sE$ for $\sE$ defined in \ref{not: totally real}. If $m = \dim_E \sV$ is even, we let $\sV^\sharp$ play the role of $\sV$ in the above proof. Recall that in \ref{thm: period char 0} we assumed that $\sfK$ satisfies condition $(\sharp)$ and we are currently in situation $V = V_0$, so that $\sK = \sfK \cap \sG(\IA_f) \subseteq \sU$ for some neat compact open $\sU \subseteq \sG^\sharp(\IA_f)$. Define the Hermitian symmetric domain $\Ohm_{\sV^\sharp}$ with $\sV$ replaced by $\sV^\sharp$ in $\Ohm_\sV$. Then we obtain an embedding of Shimura data $(\sG, \Ohm_\sV) \into (\sG^\sharp, \Ohm_{\sV^\sharp})$. By \cite[(1.15)]{DeligneTdShimura}, for some $\sU' \supseteq \sK$, the Shimura morphism $\Sh_\sK(\sG) \to \Sh_{\sU'}(\sG^\sharp)$ is an embedding. Replacing $\sU$ by $\sU \cap \sU'$ if necessary, we may assume that the Shimura morphism $\Sh_\sK(\sG) \to \Sh_{\sU}(\sG^\sharp)$ is also an embedding. Below we write this embedding simply as $j : \Sh_\sK \to \Sh_{\sU}$. Let $\sfW$ be the automorphic system of realizations on $\Sh_{\sU}$ given by $W := (\sV^\sharp)_{(\IQ)} \in \mathsf{Rep}(\sG^\sharp)$ and let $[\eta_W]$ be the tautological $\sU$-level structure on $\sfW$. 

The reader should now apply the discussion in \ref{sec: N commutes with pullback} with $(\sG, \Ohm_\sV)$ replaced by $(\sG^\sharp, \Ohm_{V^\sharp})$, $\wt{\sfV}$ replaced by $\sfW$, $(\sH, \sC, \Ohm_\sH) = (\sG^\sharp, \sU, \Ohm_{\sV^\sharp})$ and $\pi = \mathrm{id}$ (cf. \ref{rmk: odd case sH = sG}). In particular, $\sfW$ is equipped with a natural $E$-action. This time we set $N = \rN(\sV^\sharp) \in \mathrm{Rep}(\sG^\sharp)$. Let $\sfN$ be the automorphic system on $\Sh_\sU$ given by $N$ and let $[\eta_N]$ be the tautological $\sU$-level structure. Note that $N \in \mathrm{Rep}(\sG^\sharp)$ is \textit{faithful} as $\dim_E \sV^\sharp$ is odd. Now \ref{sec: N commutes with pullback} tells us that 
\begin{equation}
    \label{eqn: N on W}
    (\sfN, [\eta_N]) = (\rN(\sfW), \rN([\eta_W]))
\end{equation}
It is not hard to see that the restriction of $\sfW$ to $\Sh_\sK$ is naturally identified with $\ssfV \oplus (\sE \tensor \mathbf{1}_{\Sh_\sK})$. Correspondingly, we set $\sfQ := \sfP \oplus (\sE \tensor \mathbf{1}_S)$ and define $\nu_b : W \sto \sfQ_{B, b}$ by $\mu_b \oplus \mathrm{id}_\sE$. Then $\nu_b$ defines a $\sU$-level structure $[\nu]$ with $[\nu]_b = \sU \cdot (\nu_b \tensor \IA_f)$. Define $\beta_\IC := j_\IC \circ \varrho_\IC$. Then we have:
\begin{equation}
\label{eqn: beta on W}
(\beta_\IC)^*(\sfW, [\eta_W]) \iso (\sfQ, [\nu])|_{S_\IC}    
\end{equation}
\begin{lemma}
    There exists a unique morphism $\beta_N : S \to (\Sh_\sU)_F$ such that 
    \begin{equation}
    \label{eqn: N for beta N}
    \beta_N^* (\sfN, [\eta_N]) \iso (\rN(\sfQ), \rN([\nu])).
    \end{equation}
    Moreover, $\beta_\IC = \beta_N |_{S_\IC}$. 
\end{lemma}
\begin{proof}
    By \ref{thm: Moonen}(b) and a slight variant of the argument for \ref{lem: norm of P is weakly AM}, $\rN(\sfQ)$ is weakly AM. As $\rN([\nu])_b = \sU \cdot (\rN(\mu_b) \tensor \IA_f)$, \ref{lem: situation in practice} says that the $(\sG^\sharp, N, \sU)$-level structure $\rN([\nu])$ is $\rN(\sfQ)$-rational. As $[\wt{\mu}]$ is of type $\Ohm_{\sV}$, $[\nu]$ is of type $\Ohm_{\sV^\sharp}$, so that $\rN([\nu])$ is also of type $\Ohm_{\sV^\sharp}$. Applying \ref{thm: shimura as moduli} to the faithful representation $N$ of $\sG^\sharp$, we obtain the desired map $\beta_N$ such that (\ref{eqn: N for beta N}) holds. By the uniqueness statement in \ref{thm: moduli over C}, to show $\beta_\IC = \beta_N |_{S_\IC}$ it suffices to observe that 
    $$ \beta_\IC^*(\rN(\sfW), \rN([\eta_W])) \iso (\rN(\sfQ), \rN([\nu]))|_{S_\IC}. $$
    One checks this by applying $\rN(-)$ to (\ref{eqn: beta on W}). 
\end{proof}

\noindent \textit{Proof of \ref{thm: period char 0 F} for $m$ even:} The above implies that $\beta_\IC$ descends to $\beta_N$ over $F$. Recall that $\beta_\IC = j_\IC \circ \varrho_\IC$ and $j_\IC$ is an embedding. As the actions on $\IC$-points of both $\beta_\IC$ and $j_\IC$ are $\Aut(\IC/F)$-equivariant, the same is true for $\varrho_\IC$, so that $\varrho_\IC$ descends to a morphism $\varrho$ over $F$ with $\beta_N = j_F \circ \varrho$. 

Let $\lambda^\circ_\et :  (\sfW_\et, [\eta_W])|_{S_\IC} \iso (\sfQ_\et, [\nu])|_{S_\IC}$ be the \'etale component of (\ref{eqn: beta on W}). Then $\rN(\lambda^\circ_\et)$ is the \'etale component of (\ref{eqn: N for beta N}) restricted to $S_\IC$. Therefore, $\rN(\lambda^\circ_\et)$ descends to $S$. Applying \ref{lem: pi_1 equivariance} to $\sV^\sharp$, we have that $\lambda^\circ_\et$ descends to an isomorphism $\lambda_\et : \beta_N^* \sfW_\et \iso \sfQ_\et$ over $S$. Note the decompositions $\beta_N^* \sfW = \varrho^* \ssfV \oplus (\sE \tensor \mathbf{1}_S)$ and $\sfQ = \sfP \oplus (\sE \tensor \mathbf{1}_S)$. Since $\lambda_\et^\circ$ respects these decompositions over $S_\IC$ by construction, its descent $\lambda_\et$ over $S$ must also respect these decompositions, which are defined over $S$. Hence $\lambda_\et$ restricts to the sought after $\alpha_\et : \varrho^* \ssfV_\et \sto \sfP_\et$ in \ref{thm: period char 0 F}. \qed

\subsection{Case (R2'): non-maximal monodromy}
\label{sec: non-maximal monodromy}
By \ref{prop: dim = 1 for R2'}, we expect this case to be rare in practice. Readers who are not particularly interested in this case might skip to the next section.

\begin{lemma}
\label{lem: finite descent}
    Let $k \subseteq \IC$ be a subfield and $A, B$ be $k$-varieties with $A$ being geometrically connected. Suppose that there is a morphism $f : A_\IC \to B_\IC$ over $\IC$, and an \'etale morphism $g: B \to C$ over $k$ such that for some $h : A \to C$, $g_\IC \circ f = h_\IC$. Then $f$ descends to a subfield $k'$ of $\IC$ which is finite over $k$ such that $g_{k'} \circ f = h_{k'}$. 
\end{lemma}
\begin{proof}
    We assume without loss of generality that $k$ is algebraically closed. 
    The graph $\Gamma_f \colon A_\IC \to (A \times_C B)_\IC$ of $f$ (as a morphism between $C_\IC$-schemes) defines a section of the \'etale morphism $(g \times_C h)_\IC \colon (A \times_C B)_\IC \to A_\IC$. Hence $\Gamma_f$ maps $A_\IC$ isomorphically onto a connected component $D_\IC$ of $(A\times_C B)_\IC$. Since $k$ is algebraically closed, $D_\IC$ comes from an extension of scalars of a connected component $D \subset (A \times_C B)$. As the natural projection $D \to A$ is defined over $k$ and its base change to $\IC$ is the inverse of $\Gamma_f$, we must have that $\Gamma_f$ is also defined over $k$, and hence so is $f$. 
\end{proof}

Below for any $\IQ$-linear Tannakian category $\shC$, we write $\rN(-)$ for the functor $\Nm_{E/\IQ} : \Mod_E(\shC) \to \shC$ (cf. \ref{sec: norm functors}).\\
\textit{Proof of \ref{thm: period in R2' case}}. As before, set $\sK := \sfK \cap \sG(\IA_f)$ and let $i : \Sh_\sK(\sG) \to \Sh_\sfK(G)_E$ be the Shimura morphism. Our discussions in \ref{sec: preparation for period over E} up to \ref{lem: define varrho over C} apply without any change in the (R2') case, so that $\rho_\IC$ factors through a morphism $\varrho_\IC : S_\IC \to \Sh_\sK(\sG)_\IC$ such that 
\begin{equation}
    \label{eqn: varrho iso R2'}
    \varrho_\IC^*(\ssfV, [\eta_{\sV}]) \iso (\sfP, [\wt{\mu}])|_{S_\IC},
\end{equation}
and we still have $E \subseteq F$. We first show that $\varrho_\IC$ descends to a morphism $\varrho : S_{F'} \to \Sh_\sK(\sG)_{F'}$ for some finite extension $F'/F$ in $\IC$. 

Set $N := \rN(\sV)$. Then $N$ is a faithful representation of $\sH = \sH(\sV)$ \ref{not: totally real}. Since in \ref{thm: period in R2' case} we assumed that $\sfK$ satisfies condition $(\sharp)$, there exists a neat $\sC \subseteq \sH(\IA_f)$ such that the image of $\sK$ lies in $\sC$. Now recall our discussion in \ref{sec: N commutes with pullback} and the notations therein. Let $\pi : \Sh_\sK(\sG) \to \Sh_\sC(\sH)$ be the natural Shimura morphism over $E$. Let $\rN([\wt{\mu}])$ denote the $\sC$-level structure on $\rN(\sfP_\et)$ such that $\rN([\wt{\mu}])_b = \sC \cdot (\rN(\wt{\mu}_b) \tensor \IA_f)$. By applying $\rN(-)$ to the diagrams in \ref{eqn: fp taut diagram}, \ref{thm: Moonen}(c) implies that $\rN(\sfP)$ is weakly AM. One checks using \ref{lem: situation in practice} that $\rN(\wt{\mu})$ is $\rN(\sfP)$-rational, and is of type $\Ohm_\sH$. Then by \ref{thm: shimura as moduli}, we obtain a morphism $\rho_N : S \to \Sh_\sC(\sH)_F$ such that 
    \begin{equation}
    \label{eqn: N in appendix}
    (\rho_N)^* (\sfN, [\eta_N]) \iso (\rN(\sfP), \rN([\wt{\mu}])).
    \end{equation}
    By applying $\rN(-)$ to (\ref{eqn: varrho iso R2'}) and comparing with (\ref{eqn: N commutes with pullback}) in \ref{sec: N commutes with pullback}, for $\beta_\IC := \pi_\IC \circ \varrho_\IC$ we obtain 
    $$\beta_\IC^*(\sfN, [\eta_N]) = \varrho_\IC^* \pi_\IC^* (\sfN, [\eta_N]) \iso \varrho_\IC^* (\rN(\ssfV), \rN([\eta_\sV])) \iso (\rN(\sfP), \rN([\wt{\mu}]))|_{S_\IC}.  $$ 
    Therefore, by the uniqueness statement in \ref{thm: moduli over C}, $\beta_\IC = \rho_N |_{S_\IC}$, i.e., $\beta_\IC$ is defined over $F$. As $\pi$ is \'etale and is defined over $E \subseteq F$. By \ref{lem: finite descent}, $\varrho_\IC$ descends to a morphism $\varrho$ over some finite extension $F'/F$ in $\IC$ such that $(\rho_N)_{F'} = \pi_{F'} \circ \varrho$. 
    
    The above gives the first statement of \ref{thm: period in R2' case} and we now turn to the second. Recall that we defined $\alpha^\circ_\et : \rho_\IC^* (\sfV_\et, [\eta_V]) \sto (\sfP, [\mu])|_{S_\IC}$ in \ref{sec: define alpha circ}. Since $\wt{\sfV}_\et = i^* \sfV_\et$ and $\rho_\IC = i_\IC \circ \varrho_\IC$, we may alternatively view $\alpha^\circ_\et$ as the \'etale component of (\ref{eqn: varrho iso R2'}), i.e., an isomorphism $\varrho_\IC^* \wt{\sfV}_\et \sto \sfP_\et|_{S_\IC}$, which sends $\varrho_\IC^* [\eta_\sV]$ to $[\wt{\mu}]$. As $(\rho_N)_{F'} = \pi_{F'} \circ \varrho$,  (\ref{eqn: N in appendix}) gives us an isomorphism
    $$ \varrho^*(\rN(\wt{\sfV}_\et), \rN([\eta_\sV])) = \varrho^* \pi^* (\sfN_\et, [\eta_N]) \iso (\rN(\sfP_\et), \rN([\wt{\mu}]))|_{S_{F'}} $$
    whose restriction to $S_\IC$ is $\rN(\alpha^\circ_\et)$. This implies that $\rN(\alpha^\circ_\et)$ descends to $S_{F'}$. As we assumed that $\sK_{\ell_0} \cap \sZ(\IQ_{\ell_0}) = 1$, \ref{lem: pi_1 equivariance} tells us that the $\ell_0$-adic component $\alpha^\circ_{\ell_0} : \ssfV_{\ell_0}|_{S_\IC} \sto \sfP_{\ell_0}|_{S_\IC}$ descends to $S_{F'}$. For every other $\ell$, $\sK_\ell$ still contains an open subgroup $\sK'_\ell$ such that $\sK'_\ell \cap \sZ(\IQ_\ell) = 1$. Hence \ref{lem: pi_1 equivariance} implies that $\varrho^* \ssfV_\ell$ is \'etale-locally isomorphic to $\sfP_\ell|_{S_{F'}}$. 
    \qed

\begin{remark}
    We remark that \ref{thm: period in R2' case} is slightly weaker than \ref{thm: period char 0} (e.g., one cannot descend $\rho_\IC$ to $F$ but only to some finite extension) fundamentally because the representation $\sG \to \GL(N)$ is not faithful, but has a finite kernel. In the (R2) case, this was avoided because we worked with $\sV^\sharp$ instead.
\end{remark}


\section{Proof of Theorem B}
\subsection{A specialization lemma for monodromy} 

\begin{definition}
\label{def: lambda number}
    Let $S$ be a noetherian integral normal scheme. Let $\ell \in \sO_S^\times$ be a prime and $\sfW_\ell$ be an \'etale $\IQ_\ell$-local system. We denote by $\lambda(\sfW_\ell)$ the dimension $\dim \varinjlim_U \sfW_{\ell, s}^{U}$ where $s$ is a geometric point on $S$ and $U$ runs through open subgroups of $\pi_1^\et(S, s)$.
\end{definition}

It is clear that the definition is independent of the choice of $s$. 

\begin{definition}
\label{def: good compactification}
    Let $T$ be a noetherian base scheme and $S$ be a smooth $T$-scheme of finite type. 
    \begin{enumerate}[label=\upshape{(\alph*)}]
        \item Let $\ell \in \sO_T^\times$ be a prime and $\sfW_\ell$ be an \'etale $\IQ_\ell$-local system. We say that $\sfW_\ell$ \textbf{has constant $\lambda^\geo$} if there exists a number $\lambda$ such that for every geometric point $t \to T$ and every connected component $S^\circ$ of $S_t$, $\lambda(\sfW_\ell|_{S^\circ}) = \lambda$. When this condition is satisfied, write $\lambda^\geo(\sfW_\ell)$ for $\lambda$. 
        \item We say that $\overline{S}$ is a \textbf{good relative compactification} of $S$ if $\overline{S}$ is a smooth proper $T$-scheme and there exists a relative normal crossing divisor $D$ of $\overline{S}$ such that $S = \overline{S} - D$. 
    \end{enumerate}
\end{definition}

\begin{lemma}
\label{lem: constant lambda geo}
    Let $T$ be a DVR with special point $t$ and generic point $\eta$. Assume $\mathrm{char\,} k(\eta) = 0$ and $\ell \in \sO_T^\times$. Let $S \to T$ be a smooth morphism of finite type with $S$ being connected. Let $\sfW_\ell$ be a $\IQ_\ell$-local system over $S$. If $S$ admits a good relative compactification $\overline{S}$ over $T$, then $\sfW_\ell$ has constant $\lambda^\geo$ over $T$. 
\end{lemma}
\begin{proof}
Let $\wt{T}$ be the strict Henselianization of $T$ and let $\wt{t}$ and $\wt{\eta}$ be the special and generic point of $\wt{T}$. Let $S'$ be a connected component of $S_{\wt{T}}$. Then \cite[055J]{stacks-project} tells us that $S'_{\wt{\eta}}$ is connected. As $S'_{\wt{\eta}}$ necessarily contains a $k(\wt{\eta})$-rational point, $S'_{\wt{\eta}}$ is geometrically connected. Now by applying \cite[0E0N]{stacks-project} to $\overline{S}$, $S_{\wt{\eta}}$ and $S_{\wt{t}}$ have the same number of geometric connected components, so $S'_{\wt{t}}$ must be connected. To prove the lemma we may replace $T$ by $\wt{T}$ and $S$ by $S'$, so that $S_t$ and $S_{\eta}$ are both geometrically connected. Let $\bar{\eta}$ be the geometric point over $\eta$ defined by a chosen algebraic closure of $k(\eta)$. 

Choose a section $\sigma : T \to S$ and set $a = \sigma(t), b = \sigma(\bar{\eta})$. Note that $\sigma$ provides an \'etale path between $a$ and $b$, through which we identify $\sfW_{\ell, a}$ with $\sfW_{\ell, b}$ and $\pi^\et_1(S, a)$ with $\pi^\et_1(S, b)$. Let $\rho : \pi^\et_1(S, a) \to \GL(\sfW_{\ell, a})$ be the monodromy representation. For any group $G$ with a morphism $G \to \pi^\et_1(S, a)$ implicitly understood, write $\rho(G)^\circ$ for the identity component of the Zariski closure of the image of $G$ in $\GL(\sfW_{\ell, a})$. Clearly, $\rho(G)^\circ$ remains unchanged if we replace $G$ by a finite index subgroup. It suffices to show that $\rho(\pi_1^\et(S_t, a))^\circ = \rho(\pi_1^\et(S_{\bar{\eta}}, b))^\circ$, i.e., $\Mon^\circ(\sfW_\ell|_{S_t}, a) = \Mon^\circ(\sfW_\ell|_{S_{\bar{\eta}}}, b)$ in the notation in \ref{def: \'etale locally isomorphic}. 

Take a sequence $\sF_n$ of locally constant free $\IZ / \ell^n \IZ$-modules over $S$ such that $\sfW_\ell \iso (\varprojlim_n \sF_n) \tensor \IQ_\ell$. As $\mathrm{char\,}k(\eta) = 0$, each $\sF_n$ is tamely ramified over $\overline{S}$ by Abhyankar's lemma \cite[XIII~App.~Prop.~5.5]{SGA1}\footnote{There is a typo in the statement: $Y$ should be $\mathrm{Supp}(D)$, not $X - \mathrm{Supp}(D)$.}. Let us use a superscript ``$t$'' to indicate tame fundamental group. Then we know that $\rho(\pi^\et_1(S_?, a))^\circ = \rho(\pi_1^\et(S_?, a)^t)^\circ$ for $? = \emptyset, t$. By \cite[Exp.~XIII~2.10]{SGA1}, the natural map $\pi_1^\et(S_t, a)^t \to \pi_1^\et(S, a)^t$ is an isomorphism, so it remains to show that $\rho(\pi_1^\et(S_{\bar{\eta}}, b))^\circ = \rho(\pi_1^\et(S, b)^t)^\circ$. 

The section $\sigma(\eta)$ induces an isomorphism $\pi^\et_1(S, b) = \pi^\et_1(S_{\bar{\eta}}, b) \rtimes \Gal_{k(\eta)}$. As $\sigma^*(\sfW_\ell)$ is necessarily trivial, the subgroup $\Gal_{k(\eta)} \subseteq \pi^\et_1(S, b)$ acts trivially on $\sfW_{\ell, b}$. Therefore, $\rho(\pi^\et_1(S_{\bar{\eta}}, b))^\circ = \rho(\pi^\et_1(S_{\eta}, b))^\circ = \rho(\pi^\et_1(S, b)^t)$ as desired. Note that the second equality follows from the simple fact that $\pi_1^\et(S_\eta, b)$ maps surjectively to $\pi_1^\et(S, b)$, and $\pi^\et_1(S, b)^t$ is a quotient of $\pi_1^\et(S, b)$. 
\end{proof}

\begin{proposition}
\label{prop: thm of fixed part}
    Let $S$ be a connected smooth $\IC$-variety and $f : \sX \to S$ be a $\heartsuit$-family. Let $s$ be any Hodge-generic point on $S$. Then for any prime $\ell$ and $\sfV_\ell := R^2 f_* \IQ_\ell(1)$, $\dim \NS(\sX_s)_\IQ = \lambda(\sfV_\ell)$. 
\end{proposition}
\begin{proof}
    Up to replacing $S$ by a connected \'etale cover, assume that $M_\ell := \Mon(\sfV_\ell, s)$ is connected (see \ref{def: \'etale locally isomorphic}). Set $\rho := \NS(\sX_s)_\IQ$. By \ref{lem: extend LB on generic}, $\rho = \NS(\sX_\eta)_\IQ$. As $\sfM_\ell$ is unchanged if we replace $S$ by a further connected \'etale cover, $\dim \sfV_{\ell, s}^{M_\ell} = \lambda(\sfV_\ell)$. Let $\sfV = (\sfV_B, \sfV_\dR)$ be the VHS $R^2 f_* \IQ(1)$. Then $\sfV$ splits into $\mathbf{1}_S^{\oplus \rho} \oplus \sfP$ for some VHS $\sfP$ such that $\sfP_{B, s}^{(0, 0)} = 0$. It suffices to argue that $\sfP_{\ell, s}^{M_\ell} = 0$, where $\sfP_\ell = \sfP_B \tensor \IQ_\ell$.  

    Set $M:= \Mon(\sfV_B, s)$ (see \ref{def: maximal monodromy}). Then we have $M_\ell = M \tensor \IQ_\ell$. This implies that $V_{\ell, s}^{M_\ell} = V_{B, s}^{M} \tensor \IQ_\ell$, so we reduce to showing that $\sfP_{B, s}^{M} = 0$. We recall that Deligne's theorem of the fixed part \cite[(4.1.2)]{DelHdg} says that the subspace $\sfP_{B, s}^{M}$ has a Hodge structure which is respected by the embedding $\sfP_{B, s}^{M} \into \sfP_{B, s}$. Since the Hodge structure on $\sfP_{B, s}$ is irreducible (\cite[\S3~Lem.~2.7]{HuyK3Book}), $\sfP_{B, s}^{M}$ is either $0$ or $\sfP_{B, s}$. But it cannot be $\sfP_{B, s}$ because $M \neq 1$ by our assumption that $\sX/S$ is a $\heartsuit$-family. 
\end{proof}

\subsection{An effective theorem}

First, we extend the set-ups in \ref{set-up: char 0 base} and \ref{sec: define G for period} to a family over $\IZ_{(p)}$ when $F = \IQ$.

\begin{set-up} \label{sec: basic set-up}
Let $\sfM$ be a connected separated scheme over $\IZ_{(p)}$ which is smooth and of finite type for some prime $p > 2$. Let $(f \colon \sX \to \sfM)$ be a smooth projective morphism of relative dimension $d$ such that $\sX|_{\sfM_\IQ}$ is a $\heartsuit$-family. Let $\eta$ be the generic point of $\sfM$. Let $\bxi$ be a relatively ample line bundle on $\sX/\sfM$, which endows $R^2 f_* \IA^p_f(1)$ and $R^2 f_{\IQ *} \IZ_p(1)$ a symmetric bilinear pairing. Let $\Lambda \subseteq \Lambda_0 := \NS(\sX_\eta)_\IQ$ be a subspace which contains the class of $\bxi_\eta$. Recall that by \ref{lem: extend LB on generic}, $\Pic(\sX_\eta) = \Pic(\sX)/\Pic(\sfM)$. By choosing a section $\Lambda_0 \into \Pic(\sX_\eta)_\IQ$, we obtain an embedding $\underline{\Lambda}_0 \to R^2 f_* \IA^p_f(1)$, and for every field $k$ and $s \in \sfM(k)$, $\Lambda_0$ (and hence $\Lambda$) is naturally a subspace of $\NS(\sX_s)_\IQ$. We write $\PNS(\sX_s)_\IQ$ for the orthogonal complement of $\Lambda$ in $\NS(\sX_s)_\IQ$. Note that these definitions are independent of the section $\Lambda_0 \into \Pic(\sX_\eta)_\IQ$ chosen. Choose a base point $b \in \sfM(\IC)$ lying above $\eta$ and let the connected component of $\sfM_\IC$ which contains $b$ be $\sfM^\circ$. 


We assume that $\Mon(R^2 f_{*} \IQ_2, b)$ is connected. Now apply the set-ups in \ref{set-up: char 0 base} and \ref{sec: define G for period} with $S = \sfM_\IQ$ and $S^\circ = \sfM^\circ$ and define a system of realizations $(\sfP_B, \sfP_\dR, \sfP_\et) \in \sfR(\sfM_\IQ)$; moreover, we fix an isometry $\mu_b : V \sto \sfP_{B, b}$ and define the Shimura datum $(G, \Ohm)$. Let $(R^2 f_{\IQ *} \IZ_p)_{\mathrm{tf}}$ be the image of $R^2 f_{\IQ *} \IZ_p$ in $R^2 f_{\IQ *} \IQ_p$ and define $(R^2 f_{\IC *} \IZ_{(p)})_{\mathrm{tf}}$ similarly ($\mathrm{tf}$ is short for ``torsion-free''). Let $\bP_B := \sfP_B \cap (R^2 f_{\IC *} \IZ_{(p)}(1))_{\mathrm{tf}}$ and $\bP_p := \sfP_p \cap (R^2 f_{\IQ *} \IZ_p(1))_{\mathrm{tf}}$. For every $\ell \neq p$, let $\bP_\ell$ be the orthogonal complement of $\Lambda$ in $R^2 f_* \IQ_\ell(1)$, so that $\sfP_\ell$ over $\sfM_\IQ$ extends to $\bP_\ell$ over $\sfM$. If $k$ is a perfect field of characteristic $p$ and $W := W(k)$, for every point $t \in \sfM(k)$, $\bxi_t$ defines a pairing on the F-isocrystal $\H^2_\cris(\sX_t/W)[1/p]$ and we write $\bP_{\cris, t}[1/p]$ for the orthogonal complement of the classes in $\Lambda \subseteq \NS(\sX_t)_\IQ$. Assume that the $\IZ_{(p)}$-pairing on $L_{(p)} := \mu_b^{-1}(\bP_{B, b})$ is self-dual. We abusively write the reductive $\IZ_{(p)}$-group $\SO(L_{(p)})$ also as $G$. 
\end{set-up}

Under the above set-up, we define: 
\begin{definition}
\label{def: admissible period morphism}
    Let $\sfK \subseteq G(\IA_f)$ be a neat compact open subgroup of the form $\sfK_p \sfK^p$ for $\sfK_p = G(\IZ_p)$ and $\sfK^p \subseteq G(\IA^p_f)$. Let $\shS_\sfK(G)$ denote the integral model of $\Sh_{\sfK}(G)$ over $\IZ_{(p)}$. Let $\ell_0 \neq p$ be a prime. We say that a morphism $\rho : \sfM \to \shS_\sfK(G)$ is an \textbf{$\ell_0$-admissible period morphism} if (recall the notations in \ref{sec: the L sheaves})
    \begin{enumerate}[label=\upshape{(\alph*)}]
        \item there exists an isometry $\alpha_B^\circ : \rho_\IC^* \bL_B |_{\sfM^\circ} \sto \bP_B |_{\sfM^\circ}$ compatible with the Hodge filtrations (i.e., induces an isomorphism $(\rho_\IC^* \sfV) |_{\sfM^\circ} \sto \sfP |_{\sfM^\circ}$ of VHS over $\sfM^\circ$); 
        \item there is an isometry $\alpha_{\ell_0} : \rho^* \bL_{\ell_0} \sto \bP_{\ell_0}$ whose restriction to $\sfM^\circ$ agrees with $\alpha_B^\circ \tensor \IQ_{\ell_0}$; 
        \item for every $\ell \neq p$, $\rho^* \bL_\ell$ is \'etale-locally isomorphic to $\bP_\ell$ over $\sfM$;
        \item $\rho_\IQ^* \bL_{p}$ is \'etale locally isomorphic to $\bP_p$ over $\sfM_\IQ$. 
    \end{enumerate}
    Note that the isomorphism $\alpha_{\ell_0}$ is unique if it exists. If (b) is satisfied for every prime $\ell \neq p$, then we simply say that $\rho$ is \textbf{admissible}.  
\end{definition}

\begin{theorem}
\label{thm: TateStrong}
Consider the set-up in \ref{sec: basic set-up}. Assume that for some prime $\ell_0 \neq p$
\begin{enumerate}[label=\upshape{(\alph*)}]
    \item $\bP_{\ell_0}$ has constant $\lambda^\geo$ over $\IZ_{(p)}$ as defined in \ref{def: good compactification}, and 
    \item there is an $\ell_0$-admissible period morphism $\rho : \sfM \to \shS_\sfK(G)$ for some $\sfK$ as in \ref{def: admissible period morphism}.
\end{enumerate}
Then for every $k$ which is finitely generated over $\IF_p$ and $t : \Spec(k) \to \sfM$, the fiber $\sX_t$ satisfies the Tate conjecture for divisors. 
\end{theorem}
\begin{proof}
Define $\wt{G} := \CSpin(L)$ and $\IK_p := \wt{G}(\IZ_p)$ as in \S\ref{sec: integral model}. Choose a compact open subgroup $\IK^p \subseteq \wt{G}(\IA^p_f)$ whose image is contained in $\sfK^p$ such that $\IK := \IK_p \IK^p$ is neat. Up to replacing $\sfM$ by a further connected \'etale cover, let us assume that $\rho$ can be lifted to a morphism $\sfM \to \shS := \shS_\IK(\wt{G})$. Below we shall use $\rho$ to denote this lift. Recall the definition of special endomorphisms in \ref{def: special endomorphisms}. Under these preparations, we have the following proposition: 
\begin{proposition}
\label{prop: key diagram exists}
    For every algebraically closed field $\kappa$ and geometric point $s : \Spec(\kappa) \to \sfM$, there is an isomorphism $\t_s : \LEnd(\shA_{\rho(s)})_\IQ \sto \PNS(\sX_s)_\IQ$ such that the diagram  
    \begin{equation}
    \label{eqn: key diagram}
    \begin{tikzcd}
	{\LEnd(\shA_{\rho(s)})_\IQ} & {\PNS(\sX_s)_\IQ} \\
	{\bL_{\ell_0, {\rho(s)}}} & {\bP_{\ell_0, s}}
	\arrow["{\t_s}", from=1-1, to=1-2]
	\arrow[from=1-2, to=2-2]
	\arrow[from=1-1, to=2-1]
	\arrow["{\alpha_{\ell_0, s}}"', from=2-1, to=2-2]
    \end{tikzcd}
    \end{equation}
    commutes, where the vertical arrows are cycle class maps. 
\end{proposition}

Now we prove \ref{thm: TateStrong} assuming the proposition above. For any $k$ (not necessarily finite) and $t$ in \ref{thm: TateStrong}, we claim that 
\begin{equation}
\label{eqn: LEnd = lambda for ell}
    \LEnd(\shA_{t})_\IQ \tensor \IQ_\ell = \bL_{\ell, \bar{t}}^{\Gal(\bar{t}/t)}
\end{equation}
If $k$ is finite, the claim is given by \cite[Thm~6.4]{MPTate}. We remark that the $\ell$-independence assumption in (6.2) of \textit{loc. cit.} can now be deduced from \cite[Cor.~2.3.1]{KisinModp}. When $k$ is not finite, the claim follows from the proof of \cite[Cor.~6.11]{MPTate}. Indeed, we may assume that $k$ is the fraction field of some smooth and geometrically connected variety $T$ over a finite extension of $\IF_p$, and $t$ extends to a morphism $T \to \sfM$. Let us choose a closed point $t_0$ on $T$, which we also view as a $k(t_0)$-valued point on $\sfM$. Choose a geometric point $\bar{t}_0$ over $t_0$, and let $\gamma$ be an \'etale path connecting $\bar{t}$ and $\bar{t}_0$. By \cite[I~Prop.~2.7]{FC90}, $\End(\shA_T) = \End(\shA_t)$. This gives us a specialization morphism $\LEnd(\shA_t) \to \LEnd(\shA_{t_0})$ which fits into a commutative diagram below 
\[\begin{tikzcd}[column sep=0.3em, row sep=0.4em]
	{\LEnd(\shA_t)} && {\bL_{\ell, \bar{t}}} \\
	& {\LEnd(\shA_{t_0})} && {\bL_{\ell, \bar{t}_0}} \\
	{\End(\shA_t)} && {\End(\bH_{\ell, \bar{t}})} \\
	& {\End(\shA_{t_0})} && {\End(\bH_{\ell, \bar{t}_0})}
	\arrow[from=2-2, to=4-2]
	\arrow[from=3-1, to=3-3]
	\arrow[from=2-2, to=2-4]
	\arrow[from=2-4, to=4-4]
	\arrow[from=4-2, to=4-4]
	\arrow[from=1-1, to=2-2]
	\arrow[from=1-1, to=3-1]
	\arrow[from=3-1, to=4-2]
	\arrow[from=1-1, to=1-3]
	\arrow["\gamma", from=1-3, to=2-4]
	\arrow[from=1-3, to=3-3]
	\arrow["\gamma", from=3-3, to=4-4]
\end{tikzcd}\]
 (recall the notations in \S\ref{sec: integral model}). 
We claim that all vertical squares are Cartesian. For the squares on the left and right, this is clear by the definition \ref{def: special endomorphisms}. For the square at the front, this is obtained by applying (\ref{eqn: LEnd = lambda for ell}) to $t_0$. Hence the remaining diagram at the back must also be Cartesian. Therefore, the surjectivity of $\End(\shA_{t}) \tensor \IQ_\ell \to \End(\bH_{\ell, \bar{t}})^{\Gal(\bar{t}/t)}$ (\cite{Zarhin2}) implies the surjectivity of $\LEnd(\shA_t) \tensor \IQ_\ell \to \bL_{\ell, \bar{t}}^{\Gal(\bar{t}/t)}$. Hence we have affirmed (\ref{eqn: LEnd = lambda for ell}) for $t$. 

Finally, combining (\ref{eqn: key diagram}), (\ref{eqn: LEnd = lambda for ell}), and assumption \ref{def: admissible period morphism}(c), we have 
$$ \dim \PNS(\sX_{\bar{t}})_\IQ = \lambda(t^* \bP_{\ell_0}) = \lambda(t^*\bL_{\ell_0}) = \lambda(t^*\bL_{\ell}) = \lambda(t^*\bP_{\ell}) $$
This implies the Tate conjecture in codimension $1$ for $\sX_t$. \end{proof}

Now we prove \ref{prop: key diagram exists}, which is the key geometric input to \ref{thm: TateStrong}. 
\begin{proof}
    We first prove the statement when $\mathrm{char\,} \kappa = 0$. Without loss of generality, we may assume that $\kappa$ can be embedded to $\IC$; moreover, as $\Aut(\IC)$ acts transitively on the set of connected components of $\sfM_\IC$, we may choose an embedding such that the resulting $\IC$-point lies on the distinguished component $\sfM^\circ$ (defined in \ref{sec: basic set-up}) of $\rho$. To prove the statement we may replace $s$ by this $\IC$-point. Then the statement follows from Hodge theory. Indeed, we obtain a commutative diagram\footnote{Note that we are considering every $\sfM$-scheme also as an $\shS$-scheme via $\rho$, so $\shA_s$ (resp. $\bL_{B, s}$) is the same as $\shA_{\rho(s)}$ (resp. $\bL_{B, \rho(s)}$). Similar conventions apply below.}:
    \begin{equation}
    \label{eqn: key diagram char 0}
    \begin{tikzcd}
	{\LEnd(\shA_{s})_\IQ} & {\PNS(\sX_s)_\IQ} \\
	{\bL_{B, s}[1/p]^{(0, 0)}} & {\bP_{B, s}[1/p]^{(0, 0)}}
	\arrow["{\t_s}", from=1-1, to=1-2]
	\arrow["\iso", from=1-2, to=2-2]
	\arrow["\iso"', from=1-1, to=2-1]
	\arrow["{\alpha^\circ_{B, s}}", from=2-1, to=2-2]
    \end{tikzcd}
    \end{equation}
    The vertical maps are again cycle class maps, but this time they are isomorphisms. For the arrow on the right, we are applying the Lefschetz $(1, 1)$-theorem. Since $\alpha_{\ell_0, s} = \alpha^\circ_{B, s} \tensor \IQ_{\ell_0}$ by assumption, we obtain (\ref{eqn: key diagram}). 
    
    

    Now we assume that $\mathrm{char\,} \kappa = p$. Set $W = W(\kappa)$. We shall construct $\t_s$ by considering characteristic $0$ liftings of $s$. Let $\what{\shS}_s$ be the formal completion of $\shS_\IK(\wt{G})_{W}$ at $\rho(s)$. For a special endomorphism $\zeta \in \LEnd(\shA_s)$ consider the following functor: 
\begin{equation}
    \label{eqn: local def space of LEnd}
     \Def_{\shS}(\zeta, s) : R \mapsto \{ \wt{s} \in \what{\shS}_s(R) \mid \zeta \textit{ deforms to } \LEnd(\shA_{\wt{s}}) \}
\end{equation}
where $R$ runs through all Artin $W$-algebras. By \cite[\S5.14]{CSpin}, $\Def_\shS(\zeta, s)$ is represented by a closed formal subscheme of $\what{\shS}_s$ cut out by a single formal power series $f_\zeta \in \sO_{\what{\shS}_s}$. Similarly, let $\what{\sfM}_s$ be the formal completion of $\sfM_{W}$ at $s$. Then $\rho$ restricts to a morphism $\what{\sfM}_s \to \what{\shS}_s$. Consider the pullback of $\shA$ to $\sfM$ and define $\Def_\sfM(\zeta, s)$ to be the functor defined by (\ref{eqn: local def space of LEnd}) with $\what{\shS}_s$ replaced by $\what{\sfM}_s$. Then we have a fiber diagram:
\begin{equation}
    \begin{tikzcd}
	\Def_\sfM(\zeta, s) & \Def_\shS(\zeta, s) \\
	\what{\sfM}_s & \what{\shS}_s
	\arrow[from=1-1, to=1-2]
	\arrow[from=1-2, to=2-2]
	\arrow[from=1-1, to=2-1]
	\arrow["\rho",from=2-1, to=2-2]
	\arrow["\lrcorner"{anchor=center, pos=0.125}, draw=none, from=1-1, to=2-2]
    \end{tikzcd}
\end{equation}
In particular, $\Def_{\sfM}(\zeta, s)$ is a closed formal subscheme of $\what{\sfM}_s$ cut out by the pullback $\rho^*(f_\zeta)$. Now we prove the key intermediate lemma: 
\begin{lemma}
\label{lem: key flatness lemma}
    Up to replacing $\zeta$ by a power, $\Def_\sfM(\zeta, s)$ is flat over $W$. 
\end{lemma}
\begin{proof}
    It suffices show that if $\zeta$ does not lie in the image of $\LEnd(\shA_\sfM)_\IQ \to \LEnd(\shA_s)_\IQ$, then $\Def_\sfM(\zeta, s)$ is flat over $W$, because otherwise up to replacing $\zeta$ by a power, $\Def_\sfM(\zeta, s) = \what{\sfM}_s$. Suppose by way of contradiction that $\Def_\sfM(\zeta, s)$ is not flat over $W$, which is equivalent to saying that $\rho^*(f_\zeta)$ vanishes on the entire mod $p$ disk $\what{\sfM}_{\kappa, s} := \what{\sfM}_s \tensor_W \kappa$, i.e., the formal completion of $\sfM_\kappa$ at $s$. Let $\bzeta$ be the deformation of $\zeta$ over $\what{\sfM}_{\kappa, s}$. Let use write $u$ for the generic point of $\what{\sfM}_{\kappa, s}$. Let $S$ be the connected component of $\sfM_\kappa$ which contains $s$ and let $v$ be its generic point. Since $\what{\sfM}_{\kappa, s}$ is also the completion of $S$ at $s$, there is natural embedding $k(v) \into k(u)$ of residue fields. Let $\bar{u}$ be the geometric point over $u$ defined by a chosen algebraic closure of $k(u)$. We view it also as a geometric point over $v$.. 

    Recall that we assumed that $\Mon(R^2f_{\et *} \IQ_2, b)$ is connected in \ref{sec: basic set-up}, so that $\NS(\sX_\eta)_\IQ \into \NS(\sX_b)_\IQ$ is an isomorphism by \ref{lem: extend LB on generic}. Therefore, $\pi_1^\et(\sfM, b)$ acts trivially on $\NS(\sX_b) \tensor \IQ_{\ell_0}$. Since $\alpha_{\ell_0, b}$ is $\pi_1^\et(\sfM, b)$-equivariant by assumption \ref{def: admissible period morphism}(b), $\pi_1^\et(\sfM, b)$ acts trivially on $\LEnd(\sX_b) \tensor \IQ_{\ell_0}$ as well. Hence every element of $\LEnd(\shA_b)_\IQ$ is defined over $\eta$. It follows from \cite[I~Prop.~2.7]{FC90} that $\LEnd(\shA_\eta) = \LEnd(\shA_\sfM)$. By \ref{prop: thm of fixed part}, $\dim \PNS(\sX_b)_\IQ = \lambda^\geo(\bP_{\ell_0})$. Assumption \ref{def: admissible period morphism}(b) implies that $\bL_{\ell_0}|_{\sfM}$ also has constant $\lambda^\geo$ and $\lambda^\geo(\bL_{\ell_0}|_{\sfM}) = \lambda^\geo(\bP_{\ell_0})$. Combining these observations, we must have $\dim \LEnd(\shA_\sfM)_\IQ = \lambda^\geo(\bL_{\ell_0}|_{\sfM})$. 
    
    Now consider the subspace $I := \varinjlim_U \bL_{\ell_0, \bar{u}}^U$, as $U$ runs through open subgroups of $\pi^\et_1(v, \bar{u})$, so that $\dim I = \lambda(\bL_{\ell_0}|_{S}) = \lambda^\geo(\bL_{\ell_0}|_{\sfM})$ (recall definition \ref{def: good compactification}). As the endomorphism scheme of an abelian scheme is representable and unramified over the base, every element in $\LEnd(\shA_{\bar{u}})$ is defined over some finite \textit{separable} extension of $k(v)$ inside $k(\bar{u})$. Therefore, we obtain a well defined map $\LEnd(\shA_{\bar{u}}) \tensor \IQ_{\ell_0} \into I$. However, we note that the composite 
    $$ \LEnd(\shA_\sfM) \tensor \IQ_{\ell_0} \to \LEnd(\shA_{\bar{u}}) \tensor \IQ_{\ell_0} \into I $$
    has to be an isomorphism because $\dim \LEnd(\shA_\sfM)_\IQ = \dim I$. 
    This forces the natural map $\LEnd(\shA_\sfM)_\IQ \to \LEnd(\shA_{\bar{u}})_\IQ$ to be an isomorphism. Now note that $\bzeta_{\bar{u}} \in \LEnd(\shA_{\bar{u}})_\IQ$. But as we assumed that $\zeta$ does not come from $\LEnd(\shA_\sfM)_\IQ$, the same has to be true for $\bzeta_{\bar{u}}$. This gives the desired contradiction.
\end{proof}
Suppose we have replaced $\zeta$ by a power so that the above lemma holds. Then there exists a DVR $V$ finite flat over $W$ such that $\Def_\sfM(\zeta, s)$ admits a $V$-valued point (cf. \cite[Cor.~1.7]{Del02}). Let $\wt{s} \in \sfM_W(V)$ denote the corresponding morphism $R \to V$. Choose an algebraic clsoure $\bar{K}$. As $V$ is a strictly Henselian DVR, it defines an \'etale path $\gamma_{\wt{s}}$ connecting $w := \wt{s}_{\bar{K}}$ and $s$. There is a compatible specialization map $\NS(\sX_w) \to \NS(\sX_s)$ along $V$ (cf. \cite[Prop.~3.6]{MPJumping} and its proof), which we denote by $\sp(\wt{s})$. There is a similar specialization map of special endomorphisms. Now we have the following diagram: 

\begin{equation}
    \begin{tikzcd}[column sep=0.3em]
	{\LEnd(\shA_{w})_\IQ} && {\PNS(\sX_w)_\IQ} \\
	& {\LEnd(\shA_{s})_\IQ} && {\PNS(\sX_s)_\IQ} \\
	{\bL_{\ell_0, w}} && {\bP_{\ell_0, w}} \\
	& {\bL_{\ell_0, s}} && {\bL_{\ell_0, s}}
	\arrow["{\sp(\wt{s})}"', from=1-1, to=2-2]
	\arrow["{\t_{w}}", from=1-1, to=1-3]
	\arrow["{\sp(\wt{s})}"', from=1-3, to=2-4]
	\arrow[from=2-2, to=4-2]
	\arrow["{\alpha_{\ell_0, s}}", from=4-2, to=4-4]
	\arrow[from=1-1, to=3-1]
	\arrow["{\gamma_{\wt{s}}}"', from=3-1, to=4-2]
	\arrow[from=1-3, to=3-3]
	\arrow["{\alpha_{\ell_0, w}}"{description, pos=0.4}, from=3-1, to=3-3]
	\arrow[from=2-4, to=4-4]
	\arrow["{\gamma_{\wt{s}}}"'{pos=0.4}, from=3-3, to=4-4]
	\arrow["{\t_s}"{pos=0.4}, dashed, from=2-2, to=2-4]
\end{tikzcd}
\end{equation}
Note that $\mathrm{char\,}\bar{K} = 0$, so we have shown that $\t_w$ exists. The map $\t_s$ does not exist yet. But by construction of $\wt{s}$, $\zeta$ lifts to some (necessarily unique) $\wt{\zeta}$ over $w$, and we can define $\t_s(\zeta)$ to be $\sp(\wt{s})(\t_w(\wt{\zeta}))$. Note that $\t_s(\zeta)$ does not depend on the choice of $\wt{s}$ because its class in $\bP_{\ell_0, s}$ is completely determined by the class of $\zeta$ in $\bL_{\ell_0, s}$ via $\alpha_{\ell_0, s}$. Repeating this construction for every $\zeta \in \LEnd(\shA_{s})$, we obtain the desired map $\t_s$. 
\end{proof}

\subsection{Proof of Theorem B}

\subsubsection{} \label{sec: Grothendieck restriction} The reader may have noticed that \ref{thm: period char 0} and \ref{thm: period in R2' case} apply to geometrically connected bases. To make use of these results, we consider the following simple-minded functor: If $k \subseteq k'$ is a field extension, and $T$ is a $k'$-scheme, we write $T_{(k)}$ for the $k$-scheme obtained by composing the structure morphism of $T$ with the projection $\Spec(k') \to \Spec(k)$. Then $T \mapsto T_{(k)}$ is the left adjoint to the base change functor from $k$-schemes to $k'$-schemes. 

We will only apply this functor when $k = \IQ$. Let $S$ be a connected smooth $\IQ$-variety. Then its scheme of connected components $\pi_0(S)$ is the spectrum of some number field $F$. The natural morphism $S \to \pi_0(S)$ endows $S$ with the structure of an $F$-variety. We denote this $F$-variety by ${}^F S$. Note that now $S = ({}^F S)_{(\IQ)}$ and ${}^F S$ is geometrically connected as an $F$-variety. 

\subsubsection{}
\label{rmk: identifying sheaves}
Since $S$ and ${}^F S$ have the same underlying scheme, \'etale sheaves and filtered flat vector bundles\footnote{Technically, the statement for flat vector bundles uses also that $F$ is \'etale over $\IQ$.} on $S$ are naturally identified with the corresponding structures on ${}^F S$ and vice versa and we do not distinguish them notationally. If $Y$ is a $\IQ$-variety, then there is a canonical bijection between the sets of morphisms $\epsilon : \mathrm{Mor}_F({}^F S, Y_F) \sto \mathrm{Mor}_\IQ(S, Y)$. Suppose now that $\rho : {}^F S \to Y_F$ is a morphism, and $M$ is an \'etale sheaf or a filtered flat vector bundle on $Y$, then $\rho^* (M|_{Y_F})$ is canonically identified with $\epsilon(\rho)^* M$ because $\epsilon(\rho) = (Y_F \to Y) \circ \rho$ as morphisms of schemes.

 


\begin{proposition}
\label{prop: TateStrong}
Let $U \subseteq \mathrm{Spec}(\IZ[2^{-1}])$ be an open subscheme. Let $\sfM$ be a connected separated smooth $U$-scheme of finite type with generic point $\eta$. Let $\sX \to \sfM$ be a smooth projective morphism such that $\sX|_{\sfM_\IQ}$ is a $\heartsuit$-family. 

Let $\bxi$ be a relatively ample line bundle of $\sX/\sfM$ and let $\Lambda \subseteq \NS(\sX_\eta)_\IQ$ be a subspace containing the class of $\bxi_\eta$. Let $b \in \sfM(\IC)$ be a point lying above $\eta$ and $\sfM^\circ$ be the connected component of $\sfM_\IC$ containing $b$. Assume that $R^2 f_{*} \IQ_2(1)$ has constant $\lambda^\geo$, and for every $p \in U$, the pairing on $\PH^2(\sX_b, \IZ_{(p)})_{\mathrm{tf}}$ is self-dual. Then we have the following.
\begin{enumerate}[label=\upshape{(\alph*)}]
    \item If $(\sX/\sfM, \bxi)|_{\sfM^\circ}$ has maximal monodromy (see \ref{def: maximal monodromy}), then for every $s \in \sfM$, $\sX_s$ satisfies the Tate conjecture for divisors. 
    \item If $(\sX/\sfM, \bxi)|_{\sfM^\circ}$ belongs to case (R2') (see \ref{sec: the 4 cases}), then up to replacing $U$ by a nonempty open subscheme, the above is true. 
\end{enumerate}
\end{proposition}
Here $\PH^2(\sX_b, \IZ_{(p)})_{\mathrm{tf}}$ denotes the submodule of $\H^2(\sX_b, \IZ_{(p)})_{\mathrm{tf}}$ consisting of elements orthogonal to the image of $\Lambda$ in $\NS(\sX_b)_\IQ$ under the pairing on $\H^2(\sX_b, \IQ)$ induced by $\bxi_b$. Recall that the big monodromy case contains (R+) = (R1) + (R2) and (CM). 

\begin{proof}
    The hypothesis remains unchanged if we replace $\sfM$ by a connected \'etale cover (and $b$ by a lift). Therefore, we may assume that $\Mon(R^2 f_* \IQ_2, b)$ is connected. Moreover, to prove (b), we may assume that $\Lambda = \NS(\sX_b)_\IQ$. Indeed, replacing $\Lambda$ by $\NS(\sX_b)_\IQ$ might make $\PH^2(\sX_b, \IZ_{(p)})_{\mathrm{tf}}$ no longer self-dual only for finitely many $p$, but we are allowed to shrink $U$. 
       
    Apply the set-ups in \ref{set-up: char 0 base} to $S = \sfM_\IQ$ and $S^\circ = \sfM^\circ$. Let $V, G = \SO(V)$ and $\mu_b : V \sto \sfP_{B, b}$ be as defined in \ref{def: condition sharp}. Let $L \subseteq V$ be the $\IZ$-lattice defined by $\mu_b^{-1} (\sfP_{B, b} \cap \H^2(\sX_b, \IZ)_{\mathrm{tf}})$ and choose a $N \gg 0$ such that $\sfK := G(\IA_f) \cap \ker(\GL(L \tensor \what{\IZ}) \to \GL(L/ 2^N L))$ satisfies condition $(\sharp)$; in case (b), we additionally require $\sfK_{\ell_0 = 2}$ to be sufficiently small (see \ref{def: condition sharp} for these notions). Note that $\sfK$ is of the form $\prod_{q} \sfK_q$, where $q$ runs through all primes and $\sfK_q$ is a compact open subgroup of $G(\IQ_q)$. By \ref{lem: Larsen-Pink}, up to replacing $\sfM$ by a further finite connected \'etale cover, we assume that $\sfK$ is admissible, i.e., the image of $\pi_1^\et(\sfM, b)$ in $\GL(\sfP_{\et, b})$ lies in $\sfK$ via the chosen isometry $\mu_b$.  Now we set $F$ to be the field such that $\pi_0(S) = \Spec(F)$ and recall that the base point $b$ induces an embedding $F \into \IC$ through which $S^\circ = ({}^F S)_\IC$. 

    For (a), we may now apply \ref{thm: period char 0} to ${}^F S$ to conclude that there is a period morphism $\rho : {}^F S \to \Sh_\sfK(G)_F$ together with an isomorphism $\alpha^\circ_B : \rho_\IC^* (\sfV_B)|_{S^\circ} \to \sfP_B|_{S^\circ}$ which preserves the Hodge filtrations such that $\alpha^\circ_\et := \alpha^\circ_B \tensor \IA_f$ descends to an isomorphism $\rho^* \sfV_\et \sto \sfP_\et$ over ${}^F S$. By \ref{rmk: identifying sheaves} we may view $\rho$ as a morphism $S \to \Sh_\sfK(G)$ over $\IQ$ and $\alpha_\et$ as an isomorphism of \'etale sheaves over $S$. Let $p \neq 2$ be a prime in $U$. Restrict $\sfM$ to $\IZ_{(p)}$ and apply the set-ups in \ref{sec: basic set-up}. Note that by construction, $\sfK_p = \SO(L \tensor \IZ_p)$ is hyperspecial. Let $\shS$ be the integral model of $\Sh_\sfK(G)$ over $\IZ_{(p)}$. Recall that $\sfV_\et$ (resp. $\sfP_\et$) is the restriction of $\prod_{\ell \neq p}\bL_{\ell}\times \bL_p[1/p]$ to $\Sh_\sfK(G)$ (resp. $\prod_{\ell \neq p} \bP_\ell \times \bP_p[1/p]$ to $S = \sfM_\IQ$) (see \ref{sec: the L sheaves} and \ref{sec: basic set-up}). The existence of $\alpha_\et$ allows us to apply \ref{thm: extension property} to extend $\rho$ to a morphism $\sfM_{\IZ_{(p)}} \to \shS$, which by construction is admissible in the sense of \ref{def: admissible period morphism}. Now we conclude by \ref{thm: TateStrong}. 

    For (b) a minor adaptation is needed. Take $\ell_0 = 2$. By \ref{thm: period in R2' case} we still have a period morphism $\rho_\IC : S^\circ \to \Sh_\sfK(G)_\IC$ equipped with an isomorphism $\alpha^\circ_B : \rho_\IC^* \sfV_B \sto \sfP_B |_{S^\circ}$ which preserves the Hodge filtrations, but $\rho_\IC$ does not descend all the way to $F$. Instead, we only have that for some finite extension $F'/F$ in $\IC$, and $\wt{S} := ({}^F S) \tensor_F F'$, $\rho_\IC$ descends to a morphism $\rho : \wt{S} \to \Sh_\sfK(G)_{F'}$ such that 
    \begin{enumerate}[label=\upshape{(\roman*)}]
        \item $\alpha^\circ_{\ell_0} := \alpha^\circ_B \tensor \IQ_{\ell_0}$ descends to an isomorphism $\rho^* \sfV_{\ell_0} \sto \sfP_{\ell_0}|_{\wt{S}}$; 
        \item and $\rho^* \sfV_\ell$ is \'etale locally isomorphic to $\sfP_{\ell}|_{\wt{S}}$ for every other $\ell$.
    \end{enumerate}
    Set $S' := \wt{S}_{(\IQ)}$. Then by \ref{rmk: identifying sheaves} again, we may view $\rho$ as a morphism $S' \to \Sh_\sfK(G)$ and (i), (ii) above as statements about \'etale sheaves over $S'$. Note that $S'$ is naturally a connected \'etale over of $S$. As $L$ is self-dual at every $p \in U$, there exists a $U$-scheme $\shS$ such that $\shS \tensor \IZ_{(p)}$ is the integral canonical model of $\Sh_\sfK(G)$ over $\IZ_{(p)}$ (cf. the first theorem of \cite{Lovering}). By a standard spreading out argument, up to further shrinking $U$, we may assume that $S'$ is the $\IQ$-fiber of a smooth $U$-scheme $\sfM'$; moreover, $S' \to S$ extends to an \'etale covering map $\sfM' \to \sfM$ and $\rho : S' \to \Sh_\sfK(G)$ extends to a morphism $\rho_U : \sfM' \to \shS$. The reader readily checks using (i) and (ii) above that for each $p \in U$ the localization $\rho_U \tensor \IZ_{(p)}$ is an $\ell_0$-admissible period morphism in the sense of \ref{def: admissible period morphism}. Therefore, the conclusion follows from \ref{thm: TateStrong}. 
\end{proof}

We are now ready to prove Theorem~B. 
\begin{proof}
Note that the conclusion of Theorem~B is for $p \gg 0$. By \ref{prop: TateStrong} it suffices to show that there are exists an open subscheme $U \subseteq \Spec(\IZ[2^{-1}])$ such that after we restrict $\sfM$ to $U$, $R^2 f_{*} \IQ_2$ has constant $\lambda^\geo$. By combining Nagata's compactification and Hironaka's resolution of singularities in characteristic zero, we can find a compactification $\bar{\sfM}$ of the generic fiber $\sfM_\IQ$ such that the boundary $\fD \coloneqq \bar{\sfM} - \sfM_\IQ$, equipped with the reduced scheme structure, is a normal crossing divisor. For some open subscheme $U$ of $\mathrm{Spec}(\IZ[2^{-1}])$, $\bar{\sfM}$ and $\fD$ are defined over $U$, and $\fD$ becomes a relative normal crossing divisor over $U$. This implies that for $(p) \in U$, $\sfM_{\IZ_{(p)}}$ admits a good compactification relative to $\IZ_{(p)}$ in the sense of \ref{def: good compactification}. Now we conclude by \ref{lem: constant lambda geo}.
\end{proof}

To effectively apply \ref{prop: TateStrong}, the crux of the matter is to verify the hypothesis that $R^2 f_* \IQ_2(1)$ has constant $\lambda^{\mathrm{geo}}$. In practice, we achieve this by finding suitable proper curves on some partial compactification of $\sfM$. This is encapsulated in the following lemma, which is not stated in full generality but is tailored for direct application to the concrete situations we will consider.

\begin{lemma}
\label{lem: Step 1}
    Let $p$ be a prime and $\ell$ be a different prime. Set $W := W(\bar{\IF}_p)$, $K := W[1/p]$ and choose an isomorphism $\bar{K} \iso \IC$. Let $\sfM$ be a connected separated smooth $\IZ_{(p)}$-scheme of finite type such that $\sfM_{\IF_p}$ and $\sfM_\IQ$ are both geometrically irreducible. Let $\sX \to \sfM$ be a smooth projective morphism such that $\sX|_{\sfM_\IQ}$ is a $\heartsuit$-family. Let $\sfV$ be the VHS defined on $R^2 f_{\IC *} \IQ(1)$. 

    Suppose that there exists a smooth connected $W$-curve $C$ with a morphism to $\sfM$ such that
    \begin{enumerate}[label=\upshape{(\roman*)}]
        \item the image of $C_\IC$ is not contained in the Noether-Lefschetz loci of $\sfV$ and the restriction $\sfV|_{C_\IC}$ is non-isotrivial, and
        \item $C$ has a good compactification over $W$.
    \end{enumerate}
    Then $\IL := R^2 f_* \IQ_\ell(1)$ has constant $\lambda^{\geo}$. 
\end{lemma}
\begin{proof}
Note that condition (i) implies that $\lambda(\IL|_{C_\IC}) = \lambda(\IL|_{\sfM_\IC})$, due to \ref{prop: thm of fixed part}. 
Choose a base point $s \in \sfM(\IC)$ which lies over the generic point $\eta$ of $\sfM$. Let $s_p$ be a geometric point over the generic point $\eta_p$ of $\sfM_{\bar{\IF}_p}$. Since $s$ specializes to $s_p$, $\dim \NS(\sX_{s_p})_\IQ \ge \dim \NS(\sX_s)_\IQ$; moreover, as every element in $\NS(\sX_{s_p})_\IQ$ is stabilized by an open subgroup of $\pi_1^\et(\sfM_{\bar{\IF}_p}, s_p)$, $\lambda(\IL |_{\sfM_{\bar{\IF}_p}}) \ge \dim \NS(\sX_{s_p})_\IQ$. On the other hand, as $\pi_1$ is a functor, we have $\lambda(\IL |_{\sfM_{\bar{\IF}_p}}) \le \lambda(\IL |_{C_{\bar{\IF}_p}})$ by default. But the curve $\IL|_C$ has constant $\lambda^\geo$ because $C$ has a good relative compactification \ref{lem: constant lambda geo}, so we have
$$ \dim \NS(\sX_s)_\IQ \stackrel{\ref{prop: thm of fixed part}}{=} \lambda(\IL|_{\sfM_\IC}) = \lambda(\IL|_{C_\IC}) =  \lambda(\IL|_{C_{\bar{\IF}_p}}) \ge \lambda(\IL |_{\sfM_{\bar{\IF}_p}}). $$ This implies that $\lambda(\IL|_{\sfM_\IC}) = \lambda(\IL|_{\sfM_{\bar{\IF}_p}})$ as desired. 
\end{proof}

\section{Deforming curves on parameter spaces}
\label{sec: deform curves}
In this section, we prepare some preliminary lemmas which will be used to construct the curves as in \ref{lem: Step 1} on appropriate moduli spaces.

\subsection{Families of curves which homogeneously dominate a variety}

Let $k$ be an algebraically closed field. For a morphism $f \colon X \to Y$ between $k$-varieties, we say that $f$ has \textit{equi-dimensional} fibers if for every two points $y, y' \in Y$, $\dim X_y = \dim X_{y'}$. If $f : X \to Y$ is a smooth morphism between $k$-varieties, denote by $\sT(X/Y)$ the relative tangent bundle, i.e., the dual of $\Ohm^1_{X/Y}$. If $Y = \Spec(k)$, then we simply write $\sT X$ for $\sT(X/Y)$, and for a $k$-point $x \in X$, we write $\sT_x X$ for the tangent space $T_x X$ to emphasize that it is a fiber of $\sT X$. 

\begin{definition}
\label{def: homogeneously dominates}
Let $S$ and $T$ be two smooth irreducible $k$-varieties, $g \colon \shC \to T$ be a smooth family of connected curves, and $\varphi\colon \shC \to S$ be a morphism with equi-dimensional fibers (i.e., there exists some integer $d \ge 0$ such that every geometric fiber of $\varphi$ is non-empty and equi-dimensional of dimension $d$). Let $U$ be the maximal open subvariety of $\shC$ on which the composition $\sT(\shC/T) \into \sT \shC \to \varphi^*(\sT S)$ does not vanish. 
Suppose that 
\begin{enumerate}[label=\upshape{(\alph*)}]
    \item the induced morphism $\shC \to T \times S$ is quasi-finite, 
    \item for every $k$-point $s \in S$, $U_s := U \cap \varphi^{-1}(s)$ is dense in $\varphi^{-1}(s)$; 
    \item the morphism $U \to \IP(\sT S)$ has equi-dimensional fibers. 
\end{enumerate}
Then we say that the family of curves $\shC/T$ \textit{homogeneously dominates} $S$ (via the morphism $\varphi$). If there exists an open dense subvariety $T' \subseteq T$ such that the restriction $\shC|_{T'}$ homogeneously dominates $S$, then we say that $\shC/T$ \textit{strongly} dominates $S$. 
\end{definition}


The natural morphism $U \to \IP(\sT S)$ is induced by the identification $\shC \iso \IP(\sT(\shC/T))$. Roughly speaking, the family $\shC / T$ homogeneously dominates $S$ if there are curves parameterized by $T$ passing through every given point on $S$ in any given direction, and the sub-family of such curves has a fixed dimension. 

The notion ``$\shC/T$ strongly dominates $S$'' is only defined for convenience, as sometimes the natural families of curves have some bad locus on $T$ of smaller dimension which does not affect applications.

\begin{lemma}
\label{lem: open in proj tangent bundle}
Let $\shP \to S$ be a smooth morphism between smooth $k$-varieties. Let $\sX \subseteq \shP$ be a relative effective Cartier divisor whose total space is smooth. If $s \in S(k)$ is a point such that $\sX_s$ has isolated singularities, then there exists an open dense subvariety $U \subseteq \IP(\sT_s S)$ with the following property: For every unramified morphism $\varphi \colon C \to S$ from a smooth curve $C$ which sends a point $c \in C$ to $s$, the total space of the pullback family $\sX|_C$ has no singularity on $\sX_s$ if $d \varphi(\IP(\sT_c C)) \in U$. 
\end{lemma}
\begin{proof}
Since the question is \'etale-local in nature, we might as well assume that $S = \IA^m_k$ for $m = \dim S$, $s = 0$, $\sX_s$ has a single isolated singularity at a $k$-point $P \in \sX_s \subseteq \shP_s$, and $\shP$ is isomorphic to $\IA^n_S = S \times \IA^n_k$ near $P$. Let $x_i$'s and $s_j$'s be the coordinates on $\IA^n_k$ and $S_k$ respectively. Suppose that $\sX$ is locally cut out by an equation $F(x_1, \cdots, x_n, s_1, \cdots, s_m)$ near $P$. That $P$ is a singularity of the fiber $\sX_s$ but not of the total space $\sX$ implies that $\p F / \p s_j \neq 0$ at $P$ for some $j$. One may simply take $U$ to be the open subscheme of $\IP(\sT_s S) \iso \IP^{m - 1}$ where the coordinate of $\p_{s_j}$ is nonzero. 
\end{proof}

\begin{definition}
\label{def: generalized disc}
Let $f \colon X \to S$ be a finite type morphism between schemes. 
\begin{enumerate}[label=\upshape{(\alph*)}]
    \item Let the \textit{singular locus} $\mathrm{Sing}(f)$ be the reduced closed subscheme of $X$ whose support consists of all points where $f$ fails to be smooth.\footnote{Note that this definition is different from the one given in \cite[Tag~0C3H]{stacks-project}.}
    \item  If $f$ is in addition proper and flat, we say that the scheme-theoretic image of $\mathrm{Sing}(f)$ is the (generalized) discriminant locus of $f$, and denote it by $\mathrm{Disc}(f)$. 
    \item In the above situation, we say that $\mathrm{Disc}(f)$ is \textit{mild} if it has codimension at least $1$ in $S$ and there exists a dense open subscheme $V \subseteq \mathrm{Disc}(f)$ such that for every geometric point $s$ on $V$, the fiber $X_s$ has only isolated singularities. 
\end{enumerate}
\end{definition}

\begin{remark}
\label{rmk: formation of disc}
    Note that the properness assumption on $f$ implies that $\mathrm{Disc}(f)$ is closed in $S$. Moreover, since it is defined to be the scheme-theoretic image of a reduced scheme, it is also reduced. Its formation commutes with flat base change but not arbitrary base change: For any morphism $T \to S$, $\Disc(f_T)$ is always the reduced subscheme of $\Disc(f)_T$, so they are equal if and only if the latter is reduced.
\end{remark}

\begin{remark}
\label{rmk: codimension 1 disc}
    In many natural settings, the discriminant locus $\Disc(f)$ is expected to have codimension $1$. However, no general theorem currently characterizes precisely when this is true. We provide an ad hoc criterion which suffices for our purposes: Suppose that the base $S$ is an open subvariety of $\IP^n_\IC$ for some $n \ge 1$ and the complement has codimension $\ge 2$. Let $f : X \to S$ be a proper and generically smooth morphism. If over $U := S \smallsetminus \Disc(f)$, the VHS on $R^i (f_U)_* \IQ$ for some $i \in \IN$ is non-isotrivial, then $\Disc(f)$ must have codimension $1$. Indeed, otherwise $U$ is simply connected and cannot carry a non-isotrivial VHS. 
\end{remark}

\begin{proposition}
\label{prop: strongly dominates smooth total family}
Let $\sX$ and $S$ be as in \ref{lem: open in proj tangent bundle}. Suppose that 
\begin{itemize}
    \item the generalized discriminant variety $D \coloneqq \mathrm{Disc}(f) \subseteq S$ is mild; 
    \item there is a family of smooth curves $g : \shC \to T$ which homogeneously dominates $S$ through a morphism $\varphi : \shC \to S$.
\end{itemize}
Then for a general $k$-point $t \in T_k$, the total space of the pullback family $\sX|_{\shC_t}$ is smooth. 
\end{proposition}
\begin{proof}
By assumption we have a diagram 
\[\begin{tikzcd}
	& \shC && \sX \\
	T && S
	\arrow["g", from=1-2, to=2-1]
	\arrow["\varphi"', from=1-2, to=2-3]
	\arrow["f"', from=1-4, to=2-3]
\end{tikzcd}\]

Let $\sZ \subseteq S \times T$ be the subset of points $(s, t)$ such that $\shC_t$ passes through $s$ and the total space of $\sX|_{\shC_t}$ has a singularity lying above $\varphi^{-1}_t(s)$. It is easy to see that $\sZ$ is constructible: Let $\sX|_\shC$ be the pullback of $\sX$ along $\varphi$. Then we have a natural morphism $\sX|_\shC \to T \times_k S$ and $\sZ$ is the set-theoretic image of $\mathrm{Sing}(\sX|_\shC \to T)$. Endow $\sZ$ with the structure of a reduced scheme. 

It suffices to show that $\dim \sZ < \dim T$. Let $V$ be as in \ref{def: generalized disc}(c) and for $s \in V(k)$, let $U$ and $U_s = U \cap \varphi^{-1}(s)$ be as introduced in \ref{def: homogeneously dominates}). Let $t = g(s)$. By \ref{lem: open in proj tangent bundle}, there exists a proper closed subvariety $U_{s, \mathrm{bad}} \subseteq U$ such that the total space of $\sX|_{\shC_t}$ is not smooth near the fiber $\sX_u$ only if $u \in U_{s, \mathrm{bad}}$. Let $\wt{\sZ}_s \subseteq \varphi^{-1}(s)$ be the union of the complement of $U_s$ and the Zariski closure of $U_{s, \mathrm{bad}}$. Then $\wt{\sZ}_s$ is a proper closed subvariety of $\varphi^{-1}(s)$. Since the morphism $\varphi^{-1}(s) \to T$ is quasi-finite and the fiber $\sZ_s \subseteq T$ is contained in the image of $\wt{\sZ}_s$, we have $$ \dim \sZ_s <  \dim \varphi^{-1}(s) = \dim \shC - \dim S. $$
Since the image of $\sZ$ in $S$ is contained in $D$, and $s$ runs over an open dense subvariety of $D$, we have that $\dim \sZ \le \dim \shC - 2 = \dim T - 1$ as desired. 
\end{proof}



\subsection{Applications of the Baire category theorem}

Let $k$ be an algebraically closed field of characteristic $p > 0$. Set $W := W(k)$ and $K := W[1/p]$. Choose an algebraic closure $\bar{K}$ of $K$. 
\begin{lemma}
\label{lem: countability trick}
Suppose that $S \to B$ is a flat morphism between irreducible smooth $W$-schemes of finite type. Let $N$ be a countable union of closed proper subschemes of $S_{\bar{K}}$. Let $b \in B(k)$ be any point and $\what{B}_b$ be the formal completion of $B$ at $b$. Then the subset of points $\wt{b} \in \what{B}_b(W)$ such that $\mathrm{supp}(S_{\wt{b}} \tensor \bar{K})$ is not contained in $N$ is analytically dense. 
\end{lemma}
\begin{proof}
Let $U := \what{B}_b(W)$. By taking the union of $N$ with all its Galois conjugates, we may assume that $N$ is defined over $K$. Let $N_1, N_2, \cdots$ be the irreducible components of $N$. By flatness the morphism $S \to B$ is open, so for each $i$ there exists a proper closed subscheme $Z_i \subseteq B$ such that every $z \in B(K)$ satisfies $\supp(S_z) \subseteq N_i$ only if $z \in Z_i$. Indeed, one may simply take $Z_i$ to be the complement of the image of $S - N_i$. Since each $U - Z_i(W)$ is open dense in analytic topology, we may conclude by the Baire category theorem for complete metric spaces that $U - \cup_{i = 1}^\infty Z_i(W)$ is analytically dense. 
\end{proof}

\begin{lemma}
\label{lem: FAP for curves}
Let $S$ and $T$ be smooth irreducible $W$-schemes of finite type and $N$ be a countable union of closed proper subschemes of $S_{\bar{K}}$. Let $t \in T_k$ be a closed point and $\what{T}_t$ be the formal completion of $T$ at $t$. 

Suppose that $\shC$ is a smooth family of geometrically connected curves over $T$, and there is a morphism $\varphi \colon \shC \to S$ such that the family $(\shC/T)_{\bar{K}}$ over $\bar{K}$ strongly dominates $S_{\bar{K}}$. Then the subset of points $\wt{t} \in \what{T}_t(W)$ such that $\varphi(\supp(\shC_{\wt{t}} \tensor \bar{K}))$ is not contained $N$ is analytically dense. 
\end{lemma}
\begin{proof}
Again by Galois descent we may assume that $N = \cup_{i = 1}^\infty N_i$ for irreducible closed subschemes $N_i$ of $S_K$. Let $M_i := \shC_K \times_\varphi N_i$ and $M := \cup_{i = 1}^\infty M_i$. The assumption that $(\shC/T)_{\bar{K}}$ strongly dominates $S_{\bar{K}}$ implies that each $M_i$ is a proper closed subscheme. Now we apply the above lemma with $(S \to B, N)$ replaced by $(\shC \to T, M)$. 
\end{proof}


\section{Elliptic surfaces with $p_g = q = 1$}
\label{sec: Thm A}
\subsection{Generalities on elliptic surfaces}
\label{sec: generalities of ES}
In this section we recall some basic facts about elliptic surfaces and describe their moduli. The goal is to establish \ref{lem: thread the needle}, which will be used in conjunction with the lemmas in the preceeding section, so that eventually we can apply \ref{thm: TateStrong} and \ref{lem: Step 1} together.

Let $k$ be an algebraically closed field of characteristic $\neq 2, 3$. Let $C$ be a smooth projective curve over $k$ and $\pi \colon X \to C$ be an elliptic surface over $C$ with a zero section $\sigma \colon C \to X$ through which we also view $C$ as a curve on $X$. The \textit{fundamental line bundle} $L$ of $X/C$ is defined to be the dual of the normal bundle $N_{C/X}$, or equivalently that of $R^1 \pi_* \sO_X$. The degree of $L$ is defined to be the height of $X/C$ (or rather its generic fiber), which we denote by $\htt(X)$. Set $V_r = \H^0(L^r)$. There exists a pair $(a_4, a_6) \in V_4 \times V_6 - \{ 0 \}$, which is unique up to the action of $\lambda \in k^\times$ by $\lambda \cdot (a_4, a_6) = (\lambda^4 a_4, \lambda^6 a_6)$, such that $X$ is the minimal resolution of the hypersurface $X' \subseteq \IP(L^2 \oplus L^3 \oplus \sO_C)$ defined by the Weierstrass equation (\cite[Thm~1]{Kas})
\begin{equation}
    \label{eqn: Weierstrass for elliptic surfaces}
    y^2 z = x^3 - a_4 xz^2 - a_6 z^3,
\end{equation}
where $x, y, z$ as homogenenous coordinates on $L^2, L^3, \sO_C$ respectively. The hypersurface $X'$ has at most rational double point singularities and is called the \textit{Weierstrass normal form} of the original surface $X$. If $X'$ is smooth, then of course $X = X'$. In this paper, we only consider $X$ with $\htt(X) > 0$.

Next, we recall that Kodaira classified all the possible singular fibers in the elliptic fibration $\pi\colon X \to C$ when $k = \IC$ in \cite{Kodaira}, and his classification is well known to hold verbatim in characteristic $\neq 2, 3$ as well. We refer the reader to \cite[\S4]{SchSh} for a summary. Set $\Delta \coloneqq 4 a_4^3 - 27 a_6^2$. Let $c \in C$ be a point and denote by $\val_c$ the valuation defined by a uniformizer at $c$. The only facts we shall need from \textit{loc. cit.} are the following:

\begin{proposition}
\label{prop: Kodaira classification}
\begin{enumerate}[label=\upshape{(\alph*)}]
    \item The fiber $X_c$ is singular if and only if $\Delta$ vanishes at $c$, i.e., $\val_c(\Delta) \ge 1$. 
    \item $X_c$ is of $\mathrm{I}_n$ type $(n > 0)$ if and only if $\val_c(a_4) = \val_c(a_6) = 0$, and $n = \val_c(\Delta)$. 
    \item $X_c$ is of $\mathrm{II}$-type if and only if $\val_c(a_4) \ge 1$ and $\val_c(a_6) = 1$. 
    \item If $X_c$ is a singular fiber of any other type, $\val_c(\Delta) \ge 3$.
    \item If $X_c$ is of $\mathrm{I}_1$-type or $\mathrm{II}$-type, then $X_c = X'_c$. In other words, the singularity on $X'_c$ is not a surface singularity. 
    \item If $X_c$ is of $\mathrm{I}_2$-type, then $X'_c$ has a unique ODP singularity given by contracting the irreducible component not meeting the zero section. 
\end{enumerate}
\end{proposition}

The degree of the discriminant $\Delta$ is $12 \chi(\sO_X) = e(X)$. Recall that the genus $g(C)$ is equal to the irregularity $q(X)$ and we have $p_g(X) = \chi(\sO_X) - 1 + g(C)$. Therefore, elliptic surfaces with $p_g = 1$ fall into two types: 

\begin{itemize}
    \item $\chi(\sO_X) = 2$ and $g(C) = 0$. These are elliptic K3 surfaces. 
    \item $\chi(\sO_X) = g(C) = 1$. These surfaces have Kodaira dimension $1$. 
\end{itemize}

We are interested in the latter class. Note that although these surfaces are elliptic fibrations over genus $1$ curves, one should not confuse them with \textit{bielliptic surfaces}, which are of Kodaira dimension $0$.

\subsubsection{} \label{not: weighted projective}
For future reference we introduce some notation. Let $B$ be a base scheme and $\sV$ be a vector bundle over $B$. We denote by $\IA(\sV)$ the relative affine space over $B$ defined by $\sV$ and $\IA(\sV)^*$ the open part of $\IA(\sV)$ minus the zero section. Given a sequence of numbers $q= (q_0, \cdots, q_m)$ such that $q_i$'s are invertible in $\sO_B$ and vector bundles $\sV_0, \cdots, \sV_m$ such that $\sV = \oplus_{i = 0}^m \sV_i$, we denote by $\sP_{q}(\sV)$ the resulting \textit{weighted projective stack}, i.e., the quotient stack of $\IG_m$-action on $\sV$ given by 
$$ \lambda \colon (v_0, \cdots, v_m) \mapsto (\lambda^{q_0} v_0, \cdots, \lambda^{q_m} v_m) \text{ for } \lambda \in \IG_m, $$
and by $\IP_{q}(\sV)$ the coarse moduli space of $\sP_q(\sV)$. It is well known that this coarse moduli space can be constructed explicitly by applying the relative Proj functor to a sheaf of graded algebras over $B$. We omit the details. If $q$ is not specified, then it is assumed to be $(1, \cdots, 1)$.

\begin{set-up}
\label{constr: family of ES over S}
Let $B$ be a Noetherian $\IZ[1/6]$-scheme, $\varpi \colon \sC \to B$ be a family of smooth projective curves over $B$ of genus $g$ and $\sL$ be a relative line bundle on $\sC$ of degree $h$. Assume that $4h \ge 2g - 1$ and $h \ge 1$. Let $\sV_r$ denote the vector bundle $\varpi_* \sL^r$ for $r \ge 4$. Let $\wt{\sX}$ be the subscheme of $\IA(\sV_4 \oplus \sV_6)^* \times_B \IP(\sL^2 \oplus \sL^3 \oplus \sO_\sC)$ defined by the Weierstrass equation (\ref{eqn: Weierstrass for elliptic surfaces}) in the obvious way. Let $\mu$ be the $\IG_m$-action on $\IA(\sV_4 \oplus \sV_6)^* \times_B \IP(\sL^2 \oplus \sL^3 \oplus \sO_\sC)$ defined by 
\begin{equation}
\label{eqn: mu action stack}
    \lambda \cdot ((a_4, a_6), [x : y : z]) = ((\lambda^4 a_4, \lambda^6 a_6), [\lambda^2 x : \lambda^3 y : z]), \text{ for } \lambda \in \IG_m.
\end{equation}
Let $\sQ(\mu)$ denote the quotient stack of the $\IG_m$-action $\mu$. Then $\wt{\sX}$ descends to an algebraic substack $\sX$ of $\sQ(\mu)$. Note that $\sQ(\mu)$, and hence $\sX$, admit natural morphisms to $\sP_{(4, 6)}(\sV_4 \oplus \sV_6)$. Set $\fD := \Disc(\wt{\sX}/ \IA(\sV_4 \oplus \sV_6)^*)$ (see \ref{def: generalized disc}). It defines a closed substack $\overline{\fD}$ in $\sP_{(4,6)}(\sV_4 \oplus \sV_6)$, because it is invariant under the $\IG_m$-action on $\IA(\sV_4 \oplus \sV_6)^*$ through weight $(4, 6)$. Let ${\fU}$ denote the open complement of $\fD$ in $\IA(\sV_4 \oplus \sV_6)^*$. 
\end{set-up}  

\begin{remark}
    Note that with fixed $(\sC, \sL)$, the formation of $\sV_4, \sV_6, \wt{\sX}, \sX$ and ${\fU}$ naturally commutes with base change among $B$-schemes. We will implicitly use this for the rest of Sec.~7. However, a priori $\fD$ and $\overline{\fD}$ might not commute with non-flat base change as they can become non-reduced (cf. \ref{rmk: formation of disc}). Much of Sec.~7 is devoted to giving conditions to exclude this possibility. The key intermediate result is \ref{lem: thread the needle}, which will play an important role in the proof of Theorem A. 
\end{remark}

We remark that $\sX$ is only ``stacky'' because of the base. 

\begin{lemma}
Let $T$ be a Noetherian $B$-scheme and $\nu \colon T \to \sP_{(4, 6)}(\sV_4 \oplus \sV_6)$ be a morphism. Then the pullback $\nu^* \sQ(\mu)$, and hence $\nu^* \sX$, are flat projective schemes over $T$. 
\end{lemma}
\begin{proof}
The reader can check that $\sQ(\mu)$ is in fact a $\IP^2$-bundle over $\sP_{(4, 6)}(\sV_4 \oplus \sV_6) \times_B \sC$. Therefore, the pullback $\nu^* \sQ(\mu)$ is a $\IP^2$-bundle over the scheme $T \times_B \sC$. Being a closed substack of the scheme $\nu^* \sQ(\mu)$, $\nu^* \sX$ has to be a projective scheme. The flatness of $\nu^* \sQ(\mu)$ is clear, and one deduces the flatness of $\nu^* \sX$ using that it is locally cut out by a single equation, and its geometric fibers are of codimension $1$ (cf. \cite[\href{https://stacks.math.columbia.edu/tag/00MF}{Tag 00MF}]{stacks-project}). 
\end{proof}





\begin{proposition}
\label{prop: smoothness of total space elliptic}
The morphism $\wt{\sX} \to B$ is smooth. 
\end{proposition}
\begin{proof}
As $\wt{\sX}$ is flat over $B$, it suffices to check smoothness of geometric fibers. Hence we may assume that $B = \mathrm{Spec}(k)$, where $k$ is an algebraically closed field of characteristic $\neq 2, 3$. Let us simply write $\IA^*$ for $\IA(\sV_4 \oplus \sV_6)^*$. Choose a point $u \in \IA^*(k)$ and $c \in \sC(k)$. Choose a uniformizer $t$ of $C$ at $c$ and bases $\{ \sigma_i \}, \{ \t_j \}$ for $\sV_4, \sV_6$ respectively. Then the formal completion of $\IA^* \times C$ at $(u, c)$ can be identified with $\mathrm{Spf}(R)$, where $R = k[\![t, \alpha_0, \cdots, \alpha_{4h - g}, \beta_0, \cdots, \beta_{6h - g}]\!]$.

By choosing a local $\sO_C$-generator of $L$ at $c$, we turn $\sigma_i$'s and $\t_j$'s into elements in $k[\![t]\!]$. Let $(\{a_i\}, \{b_j\}) \in k^{10h - 2g + 2}$ be the affine coordinates of $u$ in $\IA(\sV_4 \oplus \sV_6)$. Then the restriction of $\wt{\sX}$ to $\mathrm{Spf}(R)$ can be identified with the subscheme of $\mathrm{Proj\,} R[x, y, z]$ defined by the equation 
\begin{equation}
    \label{eqn: local Weierstrass}
    W \coloneqq y^2z - x^3 + (\sum_{i = 0}^{4h - g} (a_i + \alpha_i) \sigma_i(t)) xz^2 + (\sum_{j = 0}^{6h - g} (b_j + \beta_j) \t_j(t)) z^3 = 0.
\end{equation}
Let $r$ be the special point of $\mathrm{Spf}(R)$. The singularity of the (generalized) elliptic curve $\wt{\sX}_t$ defined by the above equation when $t$, $\alpha_i$'s and $\beta_j$'s all vanish cannot appear on the $z = 0$ chart. So we may set $z = 1$ in the above equation and consider the resulting scheme in $\mathrm{Spec} R[x, y]$. As $\theta_j$'s form a basis of $\sV_6$, $\t_j(0) \neq 0$ for some $j$. Then for this $j$ the partial derivative $\p W / \p \beta_j$ remains nonzero on the special fiber. This implies that the total space of the restriction of $\wt{\sX}$ to $\mathrm{Spf}(R)$ is smooth. But the choice of $(u, c)$ is arbitrary, so $\wt{\sX}$ is smooth. 
\end{proof}

\subsection{Nonlinear Bertini theorems for families of elliptic surfaces}
\label{Sec: Nonlinear Bertini}

In this section, $k$ remains an algebraically closed field of characteristic $\neq 2, 3$. 
\begin{proposition}
\label{prop: nonlinear Bertini}
Let $C$ be a smooth projective curve of genus $g$ over $k$ and let $L$ be a line bundle on $C$ with degree $h$. Set $V_r \coloneqq \H^0(L^r)$ for every $r \in \IN$. For every $d \in \IN$, consider the closed subset of $\IA(V_4 \oplus V_6) \times C$ defined by ($\Delta \in V_{12}$ is defined by $4 a_4^3 - 27 a_6^2$ as before)
$$ \sK_d \coloneqq \{ (a_4, a_6, c) \in V_4 \times V_6 \times C \mid \val_c(\Delta) \ge d \} $$
and endow it with the reduced subscheme structure. Likewise, let $D \subseteq C \times C$ be the diagonal and define a closed subscheme in $\IA(V_4 \oplus V_6) \times (C \times C - D)$ by $$ \sK_2^+ \coloneqq \{ (a_4, a_6, c , c') \in V_4 \times V_6 \times (C \times C - D) \mid \val_c(\Delta) \text{ and } \val_{c'}(\Delta) \text{ are both }\ge 2 \}. $$
If $2h \ge g + 1$, then we have the following: 
\begin{enumerate}[label=\upshape{(\alph*)}]
    \item $\sK_d$ has codimension $d$ for $d \le 3$.
    \item $\sK_2$ has two irreducible components $\sK_2(\mathrm{I}_2)$ and $\sK_2(\mathrm{II})$ characterized by conditions $\val_c(a_6) = 0$ and $\val_c(a_6) \ge 1$ respectively.
    \item $\sK_2^+$ has codimension $4$. 
\end{enumerate}
\end{proposition}
\begin{proof}
Recall that by the Riemann-Roch theorem, for any line bundle $M$ on $C$, if $\deg(M) \ge 2g - 1$, then $h^0(M) = \deg(M) - g + 1$; if $\deg(M) = 2g - 2$, then $h^0(M) = \deg(M) - g + 1$ unless $M \iso \w_C$. 

Fix any point $c \in C$ and consider the projection $\sK_d \to C$. It suffices to show that the fiber $\sK_{d, c}$ over $c$, viewed naturally as a closed subscheme of $\IA(V_4 \oplus V_6)$, has codimension $d$. We identify the completion of $C$ along $c$ with $\mathrm{Spf}(k[\![t]\!])$ by choosing a uniformizer $t$. After choosing a local generator of $L$, we may consider the Taylor series of any $\sigma \in \H^0(L^r)$, which is a power series $\sigma(t) \in k[\![t]\!]$. By the first paragraph, for $r \ge 4$ and $d \le 3$, we have 
$$ h^0(L^r((1-d)c)) = h^0(L^r(-dc)) + 1. $$
Therefore, we may choose a basis $\sigma_0, \cdots, \sigma_{4h - g}$ for $V_4$ such that $\val_c(\sigma_i) = i$ for $i = 0, 1, 2$ and $\{ \sigma_i\}_{3 \le i \le 4h - g}$ forms a basis for $\H^0(L^4(-3c))$. We may assume that $\sigma_0(t) \equiv 1, \sigma_1(t) \equiv t$ and $\sigma_2(t) \equiv t^2$ modulo $t^3$. We choose a basis $\{ \t_0, \cdots, \t_{6h - g} \}$ in an entirely similar way. 

With the given choices of bases, we use $\{ \alpha_i \}$ and $\{ \beta_j \}$ for the coordinates of $V_4$ and $V_6$ respectively, so that $\Delta$ can be expressed as 
\begin{equation}
\label{eqn: express Delta}
    \Delta = 4 \left(\sum_{i = 0}^{4h - g} \alpha_i \sigma_i \right)^3 - 27 \left(\sum_{j = 0}^{6h - g} \beta_j \t_j \right)^2.
\end{equation}
Then the fiber $\sK_{d, c}$ ($d \le 3$) is supported on the subset of $\IA(V_4 \oplus V_6)$ cut out by the first $d$ equations from below: 
\begin{align}
\begin{cases}
    \Delta(0) &= 4 \alpha_0^3 - 27 \beta_0^2 \\
    \Delta'(0) &= 3(4 \alpha_0^2 \alpha_1 - 18 \beta_0 \beta_1) \\
    \Delta''(0) &= 24(\alpha_0 \alpha_1^2 + \alpha_0^2 \alpha_2) - 54(\beta_1^2 + 2 \beta_0 \beta_2)
\end{cases}
\end{align}

The statement (a) is clear for $d = 0, 1$. For $d = 2$, it is clear that $\sK_2$ contains the following subscheme 
$$ \sK_2(\mathrm{II}) \coloneqq \{ (a_4, a_6, c) \in V_4 \times V_6 \times C \mid \val_c(a_4) \ge 1, \val_c(a_6) \ge 1 \}, $$
such that the support of the fiber of $\sK_{2}(\mathrm{II})$ over $c$ is cut out by $\alpha_0 = \beta_0 = 0$. Let $\IA^2 \coloneqq \IA(k\< \sigma_0, \t_0 \>)$ be the affine space with coordinates $(\alpha_0, \beta_0)$ and $C' \subseteq \IA^2$ be the cuspidal curve defined by $\Delta(0) = 0$. Then the fiber of $\sK_2 - \sK_2(\mathrm{II})$ over a point in $C' - \{(0, 0)\}$ is given by a codimension $1$ hyperplane in $\IA(k\< \sigma_i, \beta_j \>_{i, j \ge 1})$. This implies that $\sK_2 - \sK_2(\mathrm{II})$ is irreducible of codimension $2$ in $\IA(V_4 \oplus V_6)$, and we denote this component by $\sK_2(\mathrm{I}_2)$. Note that this implies (b). To see the $d = 3$ case for (a), just note that $\Delta''(0)$ does not vanish identically on both $\sK_2(\mathrm{I}_2)$ and $\sK_2(\mathrm{II})$. 

Finally we treat (c). We consider the projection $\Phi \colon \sK_2^+ \to (C \times C - D)$ and take a point $(c, c') \in (C \times C - D)$. Denote the fiber of $\Phi$ over $(c, c')$ by $\Phi_{(c, c')}$. We assume first that $ L^4 \not\cong \w_C(2c + 2c')$. This condition is automatically satisfied when $2h > g + 1$ and ensures that $h^0(L^r(-2c - 2c')) = rh + g - 5$ for $ r \ge 4$. Then we may choose $\sigma_0, \cdots, \sigma_3 \in V_4$ with the following vanishing orders: 

\begin{equation}
\label{eqn: order of poles}
\begin{tabular}{ |c|c|c|c|c| } 
 \hline
  & $\sigma_0$ & $\sigma_1$ & $\sigma_2$ & $\sigma_3$ \\ \hline
 $\val_c$ & $0$ & $1$ & $\ge 2$ & $\ge 2$\\ \hline
 $\val_{c'}$ & $\ge 2$ & $\ge 2$ & $0$ & $1$ \\ \hline
\end{tabular}
\end{equation}
Then we complete $\{ \sigma_0, \cdots, \sigma_3 \}$ to a basis $\{ \sigma_i \}$ of $V_4$ by adjoining a basis for $\H^0(L^4(-2c - 2c'))$. Let $t, s$ be uniformizers of the completions of $C$ along $c$ and $c'$ respectively. After choosing local generators of $L$, we may consider Taylor series $\sigma_i(t) \in k[\![t]\!]$ and $\sigma_i(s) \in k[\![s]\!]$, and assume that $\sigma_0(t) \equiv 1, \sigma_1(t) \equiv t \mod t^2$ and $\sigma_2(s) \equiv 1, \sigma_3(s) \equiv s \mod s^2$. Choose an entirely similar basis $\{ \t_j \}$ for $V_6$ and express $\Delta$ again as in (\ref{eqn: express Delta}). Then the defining equations for $\Phi_{(c, c')}$ in $\IA(V_4 \oplus V_6)$ are 
\begin{align}
    \begin{cases}
    & 4 \alpha_0^3 - 27 \beta_0^2 = 4 \alpha_2^3 - 27 \beta_2^2 = 0 \\
    & 3(4 \alpha_0^2 \alpha_1 - 18 \beta_0 \beta_1) = 3(4 \alpha_2^2 \alpha_3 - 18 \beta_2 \beta_3) = 0 
    \end{cases}
\end{align}
By the same argument for the $d = 2$ case in (a), the above equations define a codimension $4$ subscheme. The point is that the variables with indices $0, 1$ do not interfere with those with $2, 3$.  

It remains to deal with the case when $2h = g + 1$ and $L^4 \iso \w_C(2c + 2c')$. Note that in this case $g \ge 1$, so the condition $L^4 \iso \w_C(2c + 2c')$ defines a closed subscheme of $(C \times C - D)$ of codimension at least $1$. Therefore, it is enough to show that the codimension of $\Phi_{(c, c')}$ is at least $3$. Note that we are able to choose a basis $\{ \t_j \}$ just as before, but this time choose $\{\sigma_0, \sigma_1, \sigma_2\}$ with the following vanishing orders: 
\begin{equation}
\label{eqn: order of poles}
\begin{tabular}{ |c|c|c|c| } 
 \hline
  & $\sigma_0$ & $\sigma_1$ & $\sigma_2$ \\ \hline
 $\val_c$ & $0$ & $1$ & $2$ \\ \hline
 $\val_{c'}$ & $2$ & $1$ & $0$ \\ \hline
\end{tabular}
\end{equation}
and complete it to a basis of $V_4$ by adjoining a basis of $H^0(L^4(-2c - 2c'))$. Assume that $\sigma_0(t) = 1, \sigma_1(t) = t \mod t^3$ and $\sigma_2(s) = 1 \mod s$. Then the conditions $\val_c(\Delta) \ge 2$ and $\val_{c'}(\Delta) \ge 1$ give us $3$ equations which are necessarily satisfied by $\Phi_{(c, c')}$: 
\begin{align}
    \begin{cases}
    & 4 \alpha_0^3 - 27 \beta_0^2 = 4 \alpha_2^3 - 27 \beta_2^2 = 0 \\
    & 3(4 \alpha_0^2 \alpha_1 - 18 \beta_0 \beta_1) = 0 
    \end{cases}
\end{align}
It is clear that these indeed cut out a subscheme of codimension $3$. 
\end{proof}

\begin{proposition}
\label{prop:irreducible_ell_surf}
Assume $2h \ge g + 1$. Apply set-up\ref{constr: family of ES over S} to $B \coloneqq \Spec(k)$. The resulting discriminant $\fD$ is a proper subvariety of $\IA(\sV_4 \oplus \sV_6)^*$. If $\mathrm{codim\,} \fD = 1$, then $\fD$ has a unique irreducible component $\fD_0$ of maximal dimension; moreover, for a general point $a$ on $\fD_0$, $\wt{\sX}_a$ is smooth away from a single ODP. 
\end{proposition} 
\begin{proof}
As $B = \Spec(k)$, $(\sV_4, \sV_6, \sC)$ above is the same as $(V_4, V_6, C)$ in \ref{prop: nonlinear Bertini}, and we use the notations from \ref{prop: nonlinear Bertini} and the results in \ref{prop: Kodaira classification} throughout the proof below. 

Note that for $a = (a_4, a_6) \in \IA(V_4 \oplus V_6)^*$ such that $\Delta_a = 4a_4^3 - 27 a_6^3 \in \H^0(L^{12})$ does not vanish identically on $C$, $\wt{\sX}_a$ is singular if and only if its elliptic fibration has a reducible singular fiber. It is clear that $\fD$ is contained in the image of $\sK_2$ in $\IA(V_4 \oplus V_6)$, and hence has codimension at least $1$. If $\mathrm{codim\,} \fD = 1$, then by \ref{prop: nonlinear Bertini}(c), there exists an open dense subset $U \subseteq \fD$ such that if $a \in U$, $\wt{\sX}_a$ has at most one singular fiber not of $\mathrm{I}_1$-type. If moreover this singular fiber is of $\mathrm{II}$-type, then $\wt{\sX}_a$ is smooth and $a \not\in \fD$. Therefore, the only possible irreducible component of maximal dimension in $\fD$ is the Zariski closure of the image of $\sK_2(\mathrm{I}_2)$. This implies the second statement in the proposition. 
\end{proof}

\subsection{Mod $p$ behavior of discriminants}
\label{sec: mod p behavior of disc}

\begin{set-up}
\label{constr: curves dominating moduli of ES}
Suppose that in \ref{constr: family of ES over S} $B = \Spec(\sO_B)$ for a local ring $\sO_B$, so that the vector bundles $\sV_4$ and $\sV_6$ are trivial $\sO_B$-modules. By choosing $\sO_B$-generators for $\sV_4$ and $\sV_6$, we identify $\IA(\sV_4 \oplus \sV_6)$ with $\IA^{d_1} \oplus \IA^{d_2}$, where $d_1 = 4h - g + 1$ and $d_2 = 6h - g + 1$. Assume that $2h \ge g + 1$ as in the results in \S\ref{Sec: Nonlinear Bertini}. Consider $\IP^1 = \Proj \sO_B [u, v]$. Let $\IA^1 = \Spec(\sO_B [u])$ be the $v = 1$ chart on $\IP^1$, and let $\infty : B \to \IP^1$ denote the section defined by $v = 0$. Let $\sW_r$ be the $\sO_B$-module of degree $r$ homogenous polynomials in $\sO_B [u, v]$ or equivalently the module of degree $\le r$ polynomials in $\sO_B [u]$. Consider the open subscheme $T \subseteq \IA(\sW_4^{d_1} \oplus \sW_6^{d_2})$ consisting of the points of the form 
$$ \{ (f_1, \cdots, f_{d_1}, g_1, \cdots, g_{d_2}) \mid \text{the common vanishing locus\,} V(\{f_i, g_j \}) = \emptyset \}.  $$
Then it is clear that there is a natural morphism $\IP^1_B \times_B T = \IP^1_T \to \sP_{(4, 6)}(\sV_4 \oplus \sV_6)$. By setting $v = 1$ in the polynomials $f_i$'s and $g_j$'s, we also obtain an $B$-morphism $\IA^1_B \times_B T = \IA^1_T \to \IA(\sV_4 \oplus \sV_6)^*$. Recall that $\IA(\sV_4 \oplus \sV_6)^*$ denotes the affine space $\IA(\sV_4 \oplus \sV_6)$ minus the zero section. This morphism fits into a commutative diagram
\begin{equation}
    \label{eqn: dominate moduli of ES}
    \begin{tikzcd}
    \IA^1_T \arrow{r}{\varphi} \arrow[hook]{d}{} & \IA(\sV_4 \oplus \sV_6)^* \arrow{d}{} \\ 
    \IP^1_T \arrow{r}{\overline{\varphi}} & \sP_{(4, 6)}(\sV_4 \oplus \sV_6)
\end{tikzcd}.
\end{equation}
\end{set-up}

For the content below, recall \ref{def: homogeneously dominates} and definition of $\wt{\sX}, \sX, \fD, \overline{\fD}$ and ${\fU}$ in \ref{constr: family of ES over S}. Assume that $k$ is an algebraically closed field of characteristic $\neq 2, 3$. 
\begin{proposition}
\label{prop: (4, 6) curve strongly dominates}
Suppose that $B = \Spec(k)$. Then the family $\IA^1_T / T$ strongly dominates $\IA(\sV_4 \oplus \sV_6)^*$ via $\varphi$. Moreover, for a general point $t \in T$, $\overline{\varphi}_t(\infty) \not\in \overline{\fD}$.

\end{proposition}
\begin{proof}

The first statement is an exercise of dimension counting, so we omit the details. For the second statement, it suffices to exhibit a single such $t$ as the condition is open. Note that the automorphism group of $\IP^1$ naturally acts on $T$. We start with any point $t' \in T$ be such that for some point $w \in \IA^1$, $\varphi_t(w) \not\in \fD$. Then we can always apply an automorphism of $\IP^1$ to switch $w$ and $\infty$. This gives us a point $t \in T$ and by construction $\overline{\varphi}_{t}(\infty) \not\in \overline{\fD}$. 
\end{proof}

\begin{lemma}
\label{lem: thread the needle}
Suppose that $\mathrm{char\,}k = p \ge 5$ and in set-up\ref{constr: curves dominating moduli of ES} $B$ is taken to be $\Spec(W)$ for $W := W(k)$. Assume that $\fD_K$ has codimension $1$ in $\IA(\sV_4 \oplus \sV_6)^*_K$ over $K := W[1/p]$. Then for a general $k$-point $t \in T_k$, $\overline{\varphi}_t$ has the following properties: 
\begin{enumerate}[label=\upshape{(\alph*)}]
    \item $\bvarphi_t(\infty) \not\in \overline{\fD}_k$, the total space of the family $\bvarphi^*_t(\sX) \to \IP^1_k$ is smooth and every fiber has at most a single ODP singularity. 
    \item $\bvarphi^*_t(\overline{\fD}_k)$, or equivalently $\varphi_t^*(\fD_k)$, is reduced. 
    \item For every $\wt{t} \in T(W)$ lifting $t$, $\varphi_{\wt{t}}^*(\fD)$ is \'etale over $W$, so that the open subcurve $\varphi_{\wt{t}}^*({\fU}) \subseteq \IA^1_{\wt{t}}$ has a good compactification relative to $W$.
\end{enumerate}
\end{lemma}

\begin{proof}
Using that $\overline{\fD}$ is proper over $W$ and $\fD_K$ has codimension $1$ in $\IA(\sV_4 \oplus \sV_6)^*_K$, it is not hard to see that $\fD_k$ is also of codimension exactly $1$ in $\IA(\sV_4 \oplus \sV_6)^*_k$. We break into 3 steps. \\
\indent \textbf{Step 1: }By \ref{prop:irreducible_ell_surf} and \ref{rmk: formation of disc}, the irreducible components of maximal dimension of $\fD_{K}$ and $(\fD_k)_{\mathrm{red}}$ are unique. Let us denote them by $\fD_{K}^\circ$ and $\fD_k^\circ$ respectively. Moreover, by \ref{prop:irreducible_ell_surf}, there exists an open dense subscheme $V \subseteq \fD^\circ_k$, such that for every $v \in V(k)$, $\sX_v$ has at most a single ODP singularity. In particular, by \ref{prop: (4, 6) curve strongly dominates}, as well as \ref{prop: strongly dominates smooth total family} and its proof, for a general point $t \in T(k)$, $\overline{\varphi}_t(\infty) \not\in \overline{\fD}_k$, the total space of the pullback family $\varphi_t^*(\overline{\sX})$ is smooth, and the image of $\varphi_t$ only intersects $\fD^\circ_k$ on $V$, so that every singular fiber of $\varphi_t^*(\overline{\sX})$ over $\IA^1_t$ has a single ODP; moreover, as $\overline{\varphi}_t(\infty) \not\in \overline{\fD}_k$, the total space of $\overline{\varphi}_t^*(\sX)$ is smooth. Hence we may conclude (a). \\
\indent \textbf{Step 2: } Next, we show the following claim $(\star)$: \textit{If $t$ is a general point, for any $\wt{t} \in T(W)$ lifting $t$, $Z := \varphi_{\wt{t}}^*(\fD)$ is flat over $W$.} Let $\fD_0, \cdots, \fD_r$ be the irreducible components of $\fD$ such that $\fD_0$ is the component which contains $\fD_{K}^\circ$. We claim that $\fD_0$ contains $\fD^\circ_k$ as well. To simplify notation, let us write $\IA(\sV_4 \oplus \sV_6)$ as $\shA$. Then $\fD_{0, K} \subseteq \shA_K^*$ is cut out by a single polynomial $F$ in the coordinates of the affine space $\shA_K$. By minimally clearing denominators, we may assume that the coefficients of $F$ are defined in $W$ and generate $W$. Using the fact that $\shA$ is affine and $\sO_{\shA}$ is a UFD, one checks that the Zariski closure of $\fD_{0, K}$ in $\shA$ contains the vanishing locus $V(F)$ of $F \in \sO_{\shA}$. Note that $F$ is weighted-homogeneous, so $V(F)_k \subseteq \shA_k$ at least contains the origin. In paricular, $V(F) \to B$ is surjective. By \cite[Tag~0B2J]{stacks-project}, $\dim V(F)_k = \dim V(F)_K$. This implies that $\dim \fD_{0, k} = \dim \fD_{0, K}$. By the uniqueness of $\fD_k^\circ$ as an irreducible component of maximal dimension, we conclude that $\fD_k^\circ \subseteq \fD_0$. By applying \cite[Tag~0B2J]{stacks-project} again, we also conclude that for any $i > 0$, $\fD_{i, k}$ has codimension $\ge 2$ in $\shA^*_k$.

Set $U \subseteq \fD_0$ be the complement of the closed subscheme $\cup_{i > 1} (\fD_0 \cap \fD_i)$. Then $(U_k)_{\mathrm{red}}$ is dense in $\fD_{k}^\circ$. As $t$ is general and the family $\IA^1_{T_k}$ strongly dominates $\IA(\sV_4 \oplus \sV_6)_k$, we may assume that the intersection $\mathrm{im}(\varphi_t) \cap (\fD_k)_{\mathrm{red}}$ is transverse and lies in $U_k$. Now we can prove the claim $(\star)$. Indeed, note that $\fD_0$ is a Weil divisor, and hence also a Cartier divisor of $\shA^*$, as $\shA^*$ is regular. This implies that $Z = \varphi^*_{\wt{t}}(\fD_0)$ is everywhere locally cut out in $\IP^1_W$ by a single equation. Since $Z_k \subseteq \IP^1_k$ is of codimension $1$, $Z$ is flat over $W$ by \cite[Tag~00MF]{stacks-project}.\\
\indent \textbf{Step 3: } Finally, we show (b) and (c) simultaneously. Note that if we show (c) for some $\wt{t}$, then we can already conclude (b), which conversely implies (c) for all $\wt{t}$. Indeed, $\sO_Z$ is isomorphic to $W[x]/(f(x))$ for some $f(x) \not \in pW[x]$. If $Z_k = \varphi^*(\fD_k)$ is reduced, then by Hensel's lemma $Z$ has to be a disjoint union of several copies of $W$. Hence it suffices to show (c) for some $\wt{t}$, for which we may assume that the generic fiber $\varphi_{\wt{t}} \tensor K$ intersects $\fD_K$ transversely. What we are using here is that $\IA^1_{T} \tensor \bar{K}$ strongly dominates $\IA(\sV_4 \oplus \sV_6)^*_{\bar{K}}$ and the points on $T(W)$ which lift $t$ are Zariski dense on $T_K$. 

Since $Z$ is finite and flat, $\deg_K \,Z_K = \deg_k\, Z_k$. As $Z_K$ is reduced, $\deg Z_K$ is the number of closed points on $Z_{\bar{K}}$. Therefore, to show that $Z_k$ is reduced, it suffices to show that $Z_k$ has the same number of closed points, i.e., $|Z_k(k)| = |Z_{\bar{K}}(\bar{K})|$. Now we compute $|Z_k(k)|$ by topology. Choose a prime $\ell \neq p$. Since we assumed that $\overline{\varphi}_t^* (\sX)$ is smooth over $k$, we may apply the Grothendieck-Ogg-Shafarevich formula in the form of \cite[XVI, Prop.~2.1]{SGA7II} to the morphism $g : \bvarphi_t^*(\sX) \to \IP^1_k$ and obtain
\begin{equation*}
    \chi(\bvarphi_t^*(\sX)) = \chi(\sX_{\overline{\eta}_t}) \chi(\IP^1_k) - \sum_{x \in \mathrm{Sing}(g)} \mu(g, x) 
\end{equation*}
where $\mu(g, x)$ denotes the Milnor number at the singularity $x$. By our assumption that for every $z \in Z_k$, the fiber $\sX_z$ is singular at exactly an ODP, we have that $|\Sing(g)(k)| = |Z_k(k)|$ and for each $x \in \Sing(g)$, $\mu(g, x) = 1$ (see case $2$ in the proof of Prop.~1.13 in \textit{loc. cit}.). Therefore, we have that
\begin{equation}
    \label{eqn: number of singular fibers}
    |Z_k(k)| = \chi(\sX_{\overline{\eta}_t}) \chi(\IP^1_k) - \chi(\bvarphi_t^*(\sX)).
\end{equation}

By considering how singularities might degenerate, we easily see that the fiber of $\sX$ over each point in $Z_{\bar{K}}(\bar{K})$ has at most an ODP singularity. Therefore, by the same computation as above, for $u := \wt{t}_{\bar{K}}$ and a geometric generic point $\overline{\eta}_u$ of $\IP^1_u$, we have 
\begin{equation}
    \label{eqn: number of singular fibers char 0}
    |Z_{\bar{K}}(\bar{K})| = \chi(\sX_{\overline{\eta}_u}) \chi(\IP^1_{\bar{K}}) - \chi(\bvarphi_u^*(\sX)).
\end{equation}
Using the smooth and proper base change theorem for \'etale cohomology, it is not hard to see that $\chi(\bvarphi_u^*(\sX)) = \chi(\bvarphi_t^*(\sX))$ and $ \chi(\sX_{\overline{\eta}_u}) =  \chi(\sX_{\overline{\eta}_t})$. Hence we conclude that $|Z_k(k)| = |Z_{\bar{K}}(\bar{K})|$ as desired. 
\end{proof}

Note that the fact that Milnor number for ODPs on a pencil of surfaces is equal to $1$ fundamentally uses the $p \neq 2$ assumption. Of course it is irrelevant here because we are working with the $p \ge 5$, but we remark that in \cite[Thm~4.2]{Saitodisc} the non-reducedness of the relevant discriminant scheme modulo $2$ can indeed be explained by the fact that an ordinary quadratic singularity has an even Milnor number in characteristic $2$ (\cite[Prop.~3.24]{Daichi}).\footnote{This was explained in the appendix in the earlier arXiv version of the paper, which will be published separately.} We also remark that for results in \S\ref{Sec: Nonlinear Bertini} crucially uses the $2h \ge g + 1$ assumption, which is indeed satisfied by the case we care about ($g = h = 1$). 

\subsection{Proof of Theorem A}
\label{sec: BSD proof}
In this section we work with the following set-up: 
\begin{set-up}
\label{set-up: moduli of elliptic surfaces}
Let $\shM_{1, 1}$ be the moduli stack of the pair of a genus $1$ curve together with a degree $1$ line bundle (i.e., an elliptic curve). Let $B$ be the $\IZ[1/6]$-scheme defined by 
$$ \{ (a, b) \in \mathrm{Spec}(\IZ[\frac{1}{6}][a, b]) \mid 4 a^3 - 27 b^2 \neq 0 \}. $$
Then the Weierstrass equation equips $B$ with a surjective morphism $B \to \shM_{1, 1}$. Let $(\varpi : \sC \to B, \sL)$ be the restriction of the universal family over $\shM_{1,1}$ to $B$. Apply the constructions in \ref{constr: family of ES over S} with this triple $(B, \sC, \sL)$ and define the objects $\sV_r = \varpi_* \sL^r$ ($r = 4,6$), $\wt{\sX} \to \IA(\sV_4 \oplus \sV_6)^*$ and $\fD = \Disc(\wt{\sX}/\IA(\sV_4 \oplus \sV_6)^*)$ accordingly. Below we write $\IA(\sV_4 \oplus \sV_6)^*$ as $\wt{\sfM}$, the open subscheme $\wt{\sfM} - \fD$ as $\sfM$, and the restriction of $\wt{\sX}$ to $\sfM$ as $\sX^\circ$. 
\end{set-up}

\begin{remark}
\label{rmk: all surfaces can be found}
    For any algebraically closed field $k$ of characteristic $p \neq 2, 3$ and elliptic surface $X$ over $k$ with $p_g = q = 1$, there exists a point $z \in \bsfM(k)$ such that $X$ is the minimal model of $\wt{\sX}_z$. Moreover, there are no reducible fibers in the elliptic fibration of $X$ if and only if $z \in \sfM(k)$. The choice of $z$ is unique up to the $\IG_m$-action on $\sfM$ given by $\lambda \cdot (a_4, a_6) = (\lambda^4 a_4, \lambda^6 a_6)$ where $(a_4, a_6)$ is the relative coordinate on $\IA(\sV_4 \oplus \sV_6)$. 
\end{remark}


We need a lower bound on the rank of the Kodaira-Spencer map, or equivalently the image of the period morphism over $\IC$.  

\begin{lemma}
\label{lem: KS for elliptic}
For a general $z \in \sfM_\IC$ and $X \coloneqq \sX_z$, the Kodaira-Spencer map $$ T_z \sfM_\IC \to \Hom(\H^1(\Ohm^1_{X}), \H^2(\sO_{X}))$$ has rank at least $3$.
\end{lemma}
\begin{proof}
This follows a construction of Ikeda and the Artin-Brieskorn resolution. In \cite{Ikeda}, Ikeda constructed a subfamily of elliptic surfaces over $\IC$ with the given invariants $p_g = q = 1$ using bielliptic curves of genus $3$. Let $\wt{C}$ be a bielliptic curve of genus $3$, equipped with an involution $\sigma$ such that $\wt{C}/\sigma$ is a smooth genus $1$ curve $C$. On the symmetric square $\wt{C}^{(2)}$ of $\wt{C}$, $\sigma$ also lifts to an involution $\sigma^{(2)}$. Consider the surface $Y'={\wt{C}}^{(2)}/\sigma^{(2)}$, which is shown to be a projective surface of Kodaira dimension $1$ with $6$ ODPs. Its minimal resolution $Y$ is an elliptic surface with $p_g = q = 1$. By Prop.~2.9 in \textit{loc. cit.}, the morphism $\wt{C} \to C$ can be recovered from $Y$. Note however the Weierstrass model of $Y$ is singular, so $Y$ is not given by a point on $\sfM$.

Note that for every $\IC$-point $s \in \wt{\sfM}_\IC$ with image $t$ in $B_\IC$, $s$ is given by a pair $(a_4, a_6) \in \H^0(\sL_t^4) \times \H^0(\sL_t^6)$. Let $\sfM'_\IC \subseteq \wt{\sfM}_\IC$ be the open subscheme which consists of those $s$ such that $\Delta_s = 4 a_4^3 - 27 a_6^2$ does not vanish identically on the base curve $\sC_t$. Then $\wt{\sX}_s$ is the Weierstrass normal form of an elliptic surface and by \cite[Thm~1]{Kas} has at most rational double point singularities. By applying the Artin-Brieskorn resolution \cite{Artin-Res} to $\wt{\sX}|_{\sfM'_\IC}$, we obtain a smooth and proper algebraic space $\sX_\IC^\sharp \to \sfM_\IC^\sharp$, where $\sfM_\IC^\sharp$ is an algebraic space which admits a morphism to $\sfM'_\IC$ bijective on geometric points. Moreover, \cite[Thm~2]{Artin-Res} tells us that for any $\IC$-point $z$ of $\sfM^\sharp_\IC$ which maps to a point $s$ of $\sfM'_\IC$, the Henselianization of $\sfM^\sharp_\IC$ at $z$ maps surjectively to that of $\sfM'_\IC$ at $s$. Let $\sfM^+_\IC$ be a resolution of singularities of $\sfM^\sharp_\IC$ and pullback the family $\sX^\sharp|_{\sfM'_\IC}$ to $\sfM^+_\IC$.

Note that all elliptic surfaces which can be constructed as in the first paragraph can be found as fibers of this family over $\sfM^+_\IC$. Let $\Ohm$ be the period domain parametrizing Hodge structures of K3-type on the integral lattice $\Lambda$ given by the Betti cohomology of any complex elliptic surface with $p_g = q = 1$. Let $\bsfM^+_\IC$ be the universal cover of $\sfM^+_\IC$. Then up to an action of $\O(\Lambda)$ there is a well defined period map $\bsfM^+_\IC \to \Ohm$. The moduli space of bielliptic curves of genus $3$ over $\IC$ is $4$-dimensional, so \cite[Thm~1.1(1)]{Ikeda} implies that the period image of $\bsfM^+_\IC$ is of dimension $\ge 3$. Since $\sfM_\IC \subseteq \sfM'_\IC$ is open and dense, the discussion on the Henselization of $\sfM^\sharp_\IC$ at $\IC$-points in the preceeding paragraph implies that the preimage of $\sfM_\IC \subseteq \sfM'_\IC$ in $\sfM^+_\IC$ is also open and dense. This implies that the preimage of $\sfM_\IC$ in $\bsfM^+_\IC$ also has period image of $\dim \ge 3$ by a continuity argument. 
\end{proof}




We are now ready to prove (a more general form of) Theorem~A: 

\begin{theorem}
    \label{thm: BSD as Tate}
    Assume that $k$ is a field finitely generated over $\IF_p$ for $p \ge 5$ and let $\bar{k}$ be a separable closure. Let $X$ be an elliptic surface over $k$ with $p_g = q = 1$. If all fibers in the elliptic fibration of $X_{\bar{k}}$ are irreducible, then the Tate conjecture holds for $X$. 
\end{theorem}

\begin{proof} Let $\eta$ be the generic point of $\sfM$. Let $\ff$ and $\fs$ be the classes in $\NS(\sX^\circ_\eta)$ such that $\ff$ (resp. $\fs$) is given by a smooth fiber (resp. zero section) in the elliptic fibration of $\sX^\circ_\eta$. We may polarize the family $\sX^\circ/\sfM$ with the subspace $\IQ\< \ff, \fs\>$, as it contains the class of $\bxi_\eta$ for some relatively ample line bundle $\bxi$. Since $\< \ff, \ff \> = 0, \< \ff, \fs \> = 1$ and $\< \fs, \fs \> = -1$, one deduces that for any $s \in \sfM(\IC)$, the lattice $\PH^2(\sX^\circ_s, \IZ_{(p)})_{\mathrm{tf}}$ is self-dual for every $p \ge 5$. Note that $\sfM_\IC$ is clearly connected, and by \ref{lem: KS for elliptic} and \ref{prop: dim = 1 for R2'}, the family $(\sX^\circ/\sfM, \bxi)|_{\sfM_\IC}$ has maximal monodromy, as defined in \ref{def: maximal monodromy}. By \ref{rmk: all surfaces can be found}, it suffices to apply \ref{prop: TateStrong}(a) to the family $f : \sX^\circ \to \sfM$, for which we only need to prove that $\IL := R^2 f_{*} \IQ_2(1)$ has constant $\lambda^\geo$ over $\Spec(\IZ[1/6])$. 

Let $\sfV$ be the VHS on $\sfM_\IC$ given by $R^2 f_{\IC *} \IQ(1)$. For convenience, we say that an irreducible smooth $\IC$-variety $Z$ admitting an understood morphism to $\sfM_\IC$ \textit{admissible} if its image is not contained in the Noether-Lefschetz loci of $\sfV$ and the restriction $\sfV|_Z$ is non-isotrivial. Let $p$ be a prime $\ge 5$. By \ref{lem: Step 1}, it suffices to find a smooth connected $W$-curve $C$ with a morphism to $\sfM$ such that $C_\IC \to \sfM_\IC$ is admissible and $C$ has a good compactification over $W$.

We construct $C$ in two steps. First, let $b \in B(k)$ be any point and $\what{B}_b$ be the formal completion of $B_W$ at $b$. We claim that for some $\wt{b} \in \what{B}_b(W)$, $\sfM_{\wt{b}} \tensor \IC$ is admissible. Indeed, we first observe that since $\dim B_\IC = 2$, \ref{lem: KS for elliptic} implies that for a general point $z \in B_\IC$, $\sfV|_{\sfM_z}$ is non-isotrivial. Note that the $K$-points given by an analytically dense subset of $\what{B}_b(W)$ are Zariski dense on $B_K$ (and hence also on $B_{\bar{K}}$). As the Noether-Lefschetz (NL) loci of $\sfV$ is a countable union of proper closed subvarieties, we may now apply \ref{lem: countability trick} to the case $S$ being $\sfM$ and $N$ being the NL loci to conclude that the desired lifting $\wt{b}$ exists. 

Next, we take the base $B$ in the context of \ref{constr: curves dominating moduli of ES} to be $\wt{b}$ above (and hence $\IA(\sV_4 \oplus \sV_6)^*$ and ${\fU}$ in \ref{constr: curves dominating moduli of ES} are $\wt{\sfM}_{\wt{b}}$ and $\sfM_{\wt{b}}$ respectively in the current context). Then we obtain an irreducible smooth $W$-scheme $T$ and a morphism $\varphi: \IA^1_T \to \wt{\sfM}_{\wt{b}}$ such that $\varphi_\IC$ defines a strongly dominating families of curves on $\wt{\sfM}_{\wt{b}} \tensor \IC$ by \ref{prop: (4, 6) curve strongly dominates}. Let $t$ be a general $k$-point on $T_k$ such that the conclusion of \ref{lem: thread the needle} holds.\footnote{Note that by \ref{rmk: codimension 1 disc}, the condition ``$\fD_K$ has codimension $1$'' in \ref{lem: thread the needle} can be easily checked.} Since $\sfV|_{\sfM_{\wt{b}} \tensor \IC}$ is non-isotrivial, for a general $z \in T_\IC$, the restriction of $\sfV$ to $\varphi_z^*(\sfM_{\wt{b}} \tensor \IC)$ is also non-isotrivial. Similarly, applying \ref{lem: FAP for curves} to the NL loci on $\sfM_{\wt{b}} \tensor \IC$ again, we conclude that for some $\wt{t} \in T(W)$ lifting $t$, $\varphi_{\wt{t} \tensor \IC}^*(\sfM_{\wt{b}} \tensor \IC)$ is admissible. Therefore, if we set $C$ to be $\varphi^*_{\wt{t}}(\sfM_{\wt{b}})$, then $C_\IC \to \sfM_\IC$ is admissible, and \ref{lem: thread the needle}(c) guarantees that this $C$ admits a good compactification, as desired. 
\end{proof}

\begin{remark}
\label{rmk: EGW}
    The bound in Lemma \ref{lem: KS for elliptic} is now superseded by a recent result of Engel-Greer-Ward \cite[Thm.~1.1]{EGW}, which tells us that the period morphism over $\IC$ is dominant (or equivalently the Kodaira-Spencer map at a general $\IC$-point has rank 10). Nevertheless, we retain \ref{lem: KS for elliptic} to demonstrate that the proof of \ref{thm: BSD as Tate} requires much weaker inputs—a feature that may be useful for applications to other varieties (see \ref{rmk: GM}). Also, as can be seen from \cite{FuMoonen}, it is in general highly nontrivial to translate properties of the complex period morphism to its mod $p$ reduction when one seeks effective results, so for the current method, the arguments about finding curves in sections \ref{sec: deform curves} and \ref{sec: Thm A} remain necessary. 
\end{remark}

We note that \cite{EGW} has the following consequence. 
\begin{proposition}
\label{rmk: highest MW} 
    For some finite field $k$ of characteristic $p \ge 5$, the highest analytic (or Mordell-Weil) rank of the elliptic curves considered in Theorem~A is $10$. 
\end{proposition}
\begin{proof}
    We give a sketch which allows the reader to easily fill in the details: Let $\rho : \sfM \to \shS$ denote the period morphism for $\sfM$, where $\shS$ is a suitable orthogonal Shimura variety.\footnote{We are suppressing the choice of level structures and suitable \'etale covers of $\sfM$.} Then \cite{EGW} tells us that the complex fiber $\rho_\IC$ is dominant. Since the set of CM points on $\shS(\IC)$ is dense, we can find some $x \in \sfM(\IC)$ such that $\rho(x)$ is a CM point. Let $X_\IC$ be the corresponding elliptic surface. That $\rho(x)$ is a CM point implies that $X$ has CM, i.e., $\MT(\H^2(X_\IC, \IQ))$ is a torus, or equivalently, the endomorphism algebra of the Hodge structure $T(X) := \NS(X)^\perp \subseteq \H^2(X_\IC, \IQ(1))$ is a CM field $E$ and $\dim_E T(X_\IC) = 1$ by Zarhin's result in \cite{Zarhin} (cf. \cite[Ch.~3, Thm~3.9]{HuyK3Book}). 
    
    Thanks to \ref{thm: period char 0}, the fundamental theorem of CM for K3 surfaces (\cite[Cor.~3.9.2]{Rizov}) applies verbatim to our elliptic surfaces. This allows us to apply \cite[Thm.~1.1]{ItoCM} up to the obvious adaptations. Moreover, $X_\IC$ is defined over some number field $K \supseteq E$. Let $X_K$ be a $K$-model and $F$ be the maximal totally real subfield of $E$. We can always find some place $\mathfrak{q}$ of $F$ with sufficiently large residue characteristic and a place $v$ of $K$ above $\mathfrak{q}$ such that $\mathfrak{q}$ does not split in $E$ and $X_K$ has good reduction at $v$ (or more precisely, $x \in \mathrm{im}(\sfM(\sO_{K, (v)}) \to \sfM(\IC ))$ where $\sO_{K, (v)}$ is the localization of the ring of integers $\sO_K$ at $v$). By Ito's theorem, the mod $v$ reduction $X_{k(v)}$ is supersingular, where $k(v)$ is the residue field. 

    Let $k$ be a finite extension of $k(v)$ and set $X = X_{k(v)} \tensor k$. Let $\pi : X \to C$ be the elliptic fibration of $X$ and let $\sE$ be the generic fiber over $k(C)$. Over the algebraic closure $\bar{k}$, $\mathrm{rank\,} \NS(X_{\bar{k}}) = \mathrm{rank\,} \H^2_\et(X_{\bar{k}}, \IQ_\ell(1)) = 12$. As $X_k$ is given by a $k$-point on $\sfM$, all singular fibers in its elliptic fibration are geometrically irreducible. Therefore, by the Shioda-Tate formula \cite[Cor.~6.13]{SchSh}, up to replacing $k$ by a finite extension, $\mathrm{rank\,} \sE(k(C)) = 12 - 2 = 10$. 
\end{proof}

\section{Surfaces with $p_g=K^2=1$ and $q=0$}


Recall our notations for weighted projective spaces in \ref{not: weighted projective} and set $Q := (1, 2, 2, 3, 3)$. \textit{Throughout this section, $k$ denotes an arbitrary algebraically closed field of characteristic $\neq 2, 3$ unless otherwise specified.} 

\begin{theorem}
\label{thm: classify Catanese}
     Let $X$ be a minimal surface over $k$ with $p_g = K_X^2 = 1$ and $q = 0$. Then the canonical model $X'$ of $X$ only has rational double point singularities. Moreover, if we let $(x_0,x_1, x_2,x_3,x_4)$ be the coordinates of $\IP_Q(k^{\oplus 5})$, then for $j = 1, 2$ there exist linear forms $\alpha_j$ and cubic forms $F_j$ such that $X'$ is isomorphic to the subvariety of $\IP_Q(k^{\oplus 5})$ defined by the two equations: 
    \begin{align}
    \label{eqn: Todorov canonical equation}
    x_3^2+x_0 x_4\alpha_j(x_0^2,x_1,x_2)+F_j(x_0^2,x_1,x_2)=0
    \end{align} 
\end{theorem}
\begin{proof}
    When $k = \IC$ (and hence also any $k$ with $\mathrm{char\,} k = 0$), these surfaces are classified in \cite{Catanese,Todorov}, and the statements above are taken from \cite[Thm~1.7, Prop.~1.8]{Catanese}. In \textit{loc. cit.}, Prop.~1.8 is a consequence of Thm~1.7, and the only only ingredient in Thm 1.7 (via the proof of Thm~1.4) that requires characteristic $0$ is \cite[Thm 2]{Bombieri}, which states that the linear series $|4K_X|$ is base point free. This is shown to hold in positive characteristic by \cite[Main Thm (ii)]{Ekedahl}.  Also, in \cite{Catanese} it is shown that over $\IC$ the condition $q=0$ follows from $p_g=K_X^2=1$, but here we put $q=0$ as part of the assumption. 
\end{proof}

From now on we set $U := \IZ[1/6]$ and $\IP := \IP_{Q}(\sO_U^{\oplus 5})$, and view $k$ as a geometric point on $U$. In particular, $\IP_k = \IP_{Q}(k^{\oplus 5})$. 

\begin{lemma}
\label{lem: very ample on WPS}
The line bundle $\sO_{\IP_k}(6)$ on $\IP_k$ is very ample.
\end{lemma}
\begin{proof}
By \cite[Thm~4B.7(c)]{WPS}, it suffices to show that for any positive integer $h$, and every tuple of natural numbers $(A, B, C)$ such that $A + 2B + 3C = 6 + 6h$, there exist a tuple of natural numbers $(a, b, c)$ such that $a \le A, b \le B, c \le C$ and $a + 2b + 3c = 6h$. One easily checks this with an elementary inductive argument. 
\end{proof}


\begin{proposition}
\label{prop: Lefschetz embedding}
For every smooth hypersurface $Y \subseteq \IP_k$, the embedding $\iota \colon Y \into |\sO_{\IP_k}(6)|$ is Lefschetz in the sense of \cite[{\S}XVII~2.3]{SGA7II}. 
\end{proposition}
Note that unlike $\IP_k$, the complete linear system $|\sO_{\IP_k}(6)|$ is a usual projective space (i.e., the weight is $(1, 1, \cdots, 1)$). 
\begin{proof}
We slightly adapt the proof of \cite[{\S}XVII~2.5.1]{SGA7II} to take care of the weights of variables. Since the coordiante $x_0$ has degree $1$, the $x_0 \neq 0$ locus of $\IP_k$ is an affine chart isomorphic to $\IA^4_k$ with coordinates $y_i \coloneqq x_i / x_0^{\deg \, x_i}$. We view the affine variable $y_i$ as having degree $\deg \, x_i$. Let $P$ be a point on $Y$ and assume that the projection of $\IA^k_4$ to the $y_4$-axis is etale at $P$. Up to a translation, we further assume that $y_i = 0$ at $P$ for $i = 1, 2, 3$. By Prop.~3.3 and Cor.~3.5.0 in \textit{loc. cit.}, it suffices to show that for some $H \in |\sO_{\IP_k}(6)|$, the intersection $H \cap Y$ has an ODP at $P$. To this end, it suffices to present a polynomial in variables $y_1, y_2, y_3$ of weighted degree $\le 6$ such that its vanishing locus has an ODP at the origin, for which one may simply take $y_3^2 + y_1 y_2 = 0$. 
\end{proof}


We now parametrize the surfaces in Theorem~C. 
\begin{set-up}
Let $B \subseteq \IA^{13}_U$ (resp. $B'$) be the open subscheme whose geometric points are pairs $(\alpha, F)$ such that the degree $6$ hypersurface defined by (\ref{eqn: Todorov canonical equation}) has isolated singularities (resp. is non-singular). Let $\bsfM$ be the trivial projective bundle $|\sO_\IP(6)| \times_U B$ over $B$. Let $\wt{\sX} \to \bsfM$ be the natural family whose fiber over a point $(H, H_0) \in \bsfM$ is given by $H \cap H_0$. Set $\fD := \Disc(\wt{\sX}/\bsfM)$ (see \ref{def: generalized disc}), $\sfM := \bsfM - \fD$, and $\sX^\circ := \wt{\sX}|_{\sfM}$. Then every surface of general type with ample canonical bundle, $p_g = K^2 = 1$ and $q = 0$ defined over an algebraically closed field of characteristic $p > 3$ can be found as a geometric fiber of $\sX^\circ$ on $\sfM$. 
\end{set-up}

\begin{proposition}
\label{prop: general type no gl sec}
For every geometric point $s \to B'$, the fiber $\fD_s$ is generically reduced.
\end{proposition}
\begin{proof} 
Let $H_0 \subseteq \IP_{k(s)}$ be the smooth degree $6$ hypersurface defined by $s$. By \ref{prop: Lefschetz embedding}, $H_0 \into |\sO_{\IP_s(6)}|$ is a Lefschetz embedding. In particular, the codimension of $\fD_s$ in $|\sO_{\IP_s}(6)|$ is $1$. This implies that $\fD$ is flat over $B'$, and hence the degree of $\fD_s$ in the projective space $|\sO_{\IP_s}(6)|$ is independent of $s$. 

Therefore, it suffices to show that $d_s := \deg\,(\fD_s)_{\red}$ is independent of $s$, because we know $\fD_s$ is reduced if $s$ is a geometric generic point over $B'$. The point then is that we can compute $\deg\,(\fD_s)_{\red}$ by topology and this computation is independent of $s$. Let $\gamma \colon \IP^1_{k} \into |\sO_{\IP_{s}}(6)|$ be a general pencil, which by \ref{prop: Lefschetz embedding} is necessarily a Lefschetz pencil for $H_0$. Then $d_s$ is simply the number of singular fibers on $\wt{\sX}|_{\gamma}$. Note that the total space of $\wt{\sX}|_{\gamma}$ is smooth and every singular fiber is smooth away from an ODP. Let $\bar{\eta}$ be the geometric generic point of $\IP^1_{k(s)}$. We apply the key argument in Step 3 of \ref{lem: thread the needle} again: By the Leray spectral sequence and the Grothendieck-Ogg-Shafarevich formula, we have 
\begin{equation}
    d_s = \chi(\overline{\sX}|_{\gamma}) - \chi(\IP^1) \chi(\sX_{\bar{\eta}}). 
\end{equation}
Again this uses the fact that $\mathrm{char\,}k(s) \neq 2$. By a simple deformation argument, one sees that the RHS of the above equation is independent of $s$, and hence so is the LHS. 
\end{proof}

Now we are ready to prove Theorem~C. We recall that over $\IC$, a minimal surface $X$ with the given invariants is known to be simply connected \cite{Todorov}, hence $\H^2(X, \IZ)$ is torsion-free. 

\begin{proof}
Let us polarize the family $f : \sX^\circ \to \sfM$ with the relative canonical bundle $\sK$. Then for each $s \in \sfM(\IC)$, $\PH^2(\sX_s, \IZ)$ is unimodular, as $\H^2(\sX_s, \IZ)$ is unimodular and $\sK^2_s = 1$. Let $\sfV$ be the VHS on $R^2 f_{\IC*} \IQ(1)$. Then since $\sfV$ over $\sfM_\IC$ has period dimension $18$ (\cite{Catanese}, \cite{Todorov}), we know that the family $(\sX^\circ/\sfM, \sK)|_{\sfM_\IC}$ has maximal monodromy by \ref{prop: dim = 1 for R2'}. 

The proof is entirely similar to the proof of \ref{thm: BSD as Tate}, so we only explain the adaptations. Again, we apply \ref{prop: TateStrong} to the family $\sX^\circ/\sfM$ together with \ref{lem: Step 1} for $\IL := R^2 f_* \IQ_2(1)$ and any $p \ge 5$. Set $W := W(\bar{\IF}_p)$, $K := W[1/p]$, choose $\bar{K} \iso \IC$, and call a smooth irreducible $\IC$-variety $Z$ admitting an understood morphism to $\sfM_\IC$ \textit{admissible} if $\sfV|_{Z}$ is non-isotrivial, and the image of $Z$ is not contained in the Noether-Lefschetz (NL) loci. The task is to find a smooth connected $W$-curve $C$ with a morphism to $\sfM$ such that $C_\IC \to \sfM_\IC$ is admissible and $C$ has a good compactification over $W$.

Using that $\sfV$ over $\sfM_\IC$ has period dimension $18 > \dim B_\IC$, for a general point on $B_\IC$, the restriction of $\sfV$ is not isotrivial. Now let $b \in B'_k$ be a closed point and $\what{B}'_b$ be the formal completion of $B'_W$ at $b$. Then the $K$-points given by any analytically dense subset of $\what{B}_b(W)$ are Zariski dense on $B_K$. By applying \ref{lem: countability trick} to the NL loci of $\sfV$, we can find a lifting $\wt{b} \in B'(W)$ such that $\sfM_{\wt{b}} \tensor \IC$ is admissible. Next, note that $\bsfM_{\wt{b}} = |\sO_{\IP_{\wt{b}}}(6)|$ over $W$. By \ref{prop: general type no gl sec}, $\fD_b \subseteq \bsfM_b$ is generically reduced. By considering the relative Grassmanian of lines on $\bsfM_{\wt{b}}$ and applying \ref{lem: FAP for curves}, one finds a line $L \iso \IP^1_W$ in $\bsfM_{\wt{b}}$ such that $L \cap \fD_{\wt{b}}$ is \'etale over $W$ and $(L - \fD_{\wt{b}} \cap L)_\IC \subseteq \sfM_\IC$ is admissible. Hence we may take $C = L - \fD_{\wt{b}} \cap L$. 
\end{proof}



\paragraph{Acknowledgements} We wish to thank Jordan Ellenberg, Philip Engel, Shizhang Li, Yuchen Fu, Davesh Maulik, Martin Orr, Alex Petrov, Ananth Shankar, Mark Shusterman, Yunqing Tang, Lenny Taelman, Daichi Takeuchi, Kai Xu, Ruijie Yang for helpful conversations or email correspondences. A special thank goes to Dori Bejleri for generously helping us with lots of technical questions about elliptic surfaces. We also thank the anonymous reviewer for many helpful comments that improved the exposition of the paper. 

P.H.\ is partially supported by ERC Consolidator Grant 770936: NewtonStrat. X.Z.\ is partially supported by the Simons Collaborative Grant 636187, NSF grant DMS-2101789, and NSF FRG grant DMS-2052665. Z.Y. benefited from several extended research visits to discuss this work and is thankful to Shanghai Center of Mathematics, University of Amsterdam, and the Max Planck Institute for their hospitality. 

Finally, Z.Y. wishes to thank Davesh Maulik, Ben Moonen, Ananth Shankar and his PhD advisor Mark Kisin for their kind encouragements.

\printbibliography

{\footnotesize
\paragraph{Paul Hamacher} Technische Universit{\"a}t M{\"u}nchen, Zentrum Mathematik - M11, Boltzmannstrasse 3, 85748 Garching, Deutschland. \,\,\, Email: \url{paul.hamacher@gmx.de}

\paragraph{Ziquan Yang} The Institute of Mathematical Sciences and Department of Mathematics, The Chinese University of Hong Kong, Shatin, N.T., Hong Kong.  \,\,\,  Email: \url{zqyang@cuhk.edu.hk} (corresponding author \Letter)

\paragraph{Xiaolei Zhao} Department of Mathematics, University of California, Santa Barbara, 6705 South Hall, Santa Barbara, CA 93106, USA. \,\,\, Email: \url{xlzhao@math.ucsb.edu}}

\end{document}